\documentclass[10pt]{amsart}
\usepackage{amsmath,amscd,amsbsy,amssymb,amsthm,amsfonts}
\usepackage{latexsym,url,bm}
\usepackage{cite,enumitem}
\usepackage{fullpage}
\usepackage{color}
\usepackage[colorlinks=true]{hyperref}
\usepackage{chngcntr}
\usepackage{apptools}
\usepackage{graphicx}

\newtheorem{theorem}{Theorem}[section]
\newtheorem{proposition}[theorem]{Proposition}
\newtheorem{lemma}[theorem]{Lemma}

\theoremstyle{remark}
\newtheorem{remark}{Remark}[section]

\theoremstyle{definition}
\newtheorem{definition}[theorem]{Definition}

\allowdisplaybreaks

\numberwithin{equation}{section}

\newcommand{\p}{\partial}

\newcommand{\R}{\mathbb{R}}

\renewcommand{\H}{{\mathcal H }}
\newcommand{\al}{\alpha}
\newcommand{\f}{\frac}

\newcommand{\bw}{\mathbf{W}}

\newcommand{\nP}{{\mathbf P}}
\newcommand{\CalAZ}{{\mathcal{A}_0}}
\newcommand{\ASSharp}{{\mathcal{A}_{\sharp,\frac{7}{4}}^2}}

\author{Lizhe Wan}
\address{Beijing International Center for Mathematical Research, Peking University}
\curraddr{}
\email{wanlizhe@pku.edu.cn}

\author{Jiaqi Yang}
\address{School of Mathematics and Statistics, Northwestern Polytechnical University}
\curraddr{}
\email{yjqmath@nwpu.edu.cn, yjqmath@163.com}

\begin{document}

\title{Low regularity well-posedness for two-dimensional hydroelastic waves}

\begin{abstract}
	We investigate the low regularity local well-posedness of two-dimensional irrotational deep hydroelastic waves.
   Building on the approach of Ifrim-Tataru \cite{MR3667289} and Ai-Ifrim-Tataru \cite{AIT}, in particular by constructing a cubic modified energy that incorporates a paradifferential weight chosen carefully, we prove that the hydroelastic waves are locally well-posed in $\H^s$ for $s>\f34$.
	\end{abstract}

	\maketitle
    \tableofcontents

\section{Introduction}
The \textbf{hydroelastic wave problem} describes the interaction between elastic structures and hydrodynamic excitation. 
It arises in a wide range of applications, including biology, medical science, and ocean engineering; see, for example, \cite{MR2333062, MR1312617} and the references therein. 
Based on the Cosserat shell theory under Kirchhoff’s hypotheses, which accounts for both bending stresses in the sheet and membrane-stretching tension, Toland \cite{MR2413099} introduced a fully nonlinear elastic model for two-dimensional hydroelastic waves with a clear Hamiltonian structure. This model was later extended to three spatial dimensions by Plotnikov \& Toland \cite{MR2812947}. 
A comprehensive review by Părău \textit{et al.} \cite{MR2812939,MR4823877} summarizes recent advances in the analysis, numerical simulation, experimentation, and applications of hydroelastic waves.

\begin{figure}[htbp] 
    \centering 
\includegraphics[width=0.7\textwidth]{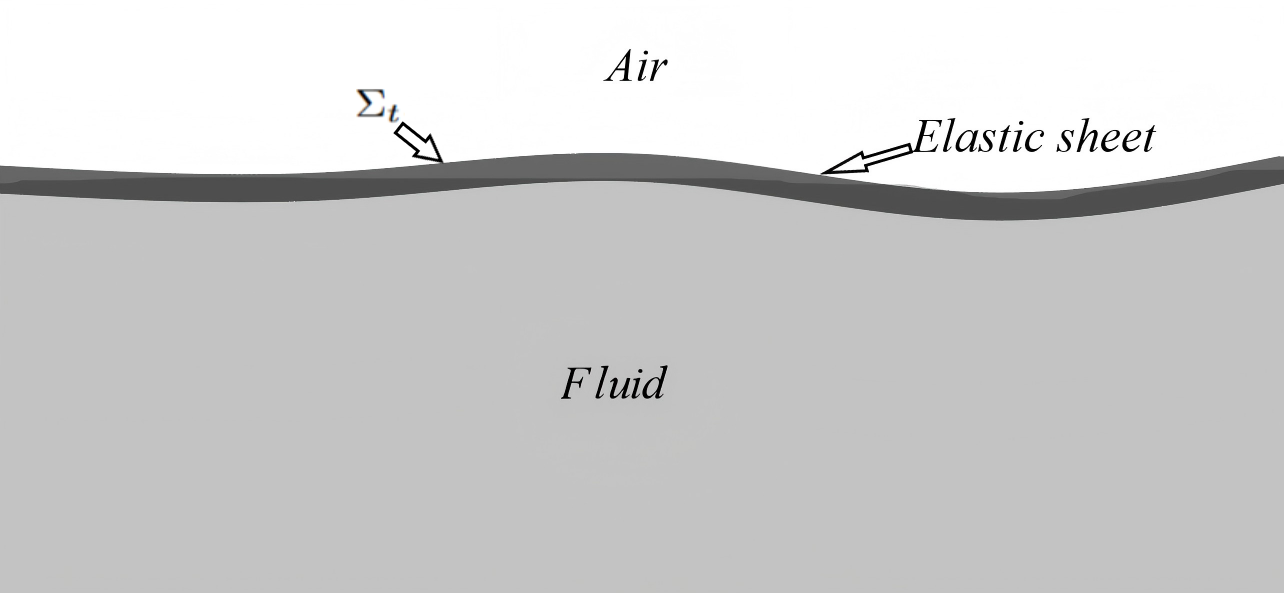}
\caption{Hydroelastic waves} 
    \label{fig}
\end{figure}
In this paper, we investigate a hydroelastic model in which a two-dimensional, inviscid, incompressible fluid undergoes irrotational motion beneath a frictionless thin elastic sheet, as illustrated in Figure \ref{fig}. A typical example of this configuration occurs in polar regions, where water freezes to form an ice sheet during winter. Such ice sheets are often used as roads or runways and can subsequently be fractured by air-cushioned vehicles.

We begin by recalling its mathematical formulation. The fluid domain at time \( t \) is denoted by \( \Omega_t \subset \mathbb{R}^2 \), defined as the region below the graph of a function \( \eta: \mathbb{R}_t \times \mathbb{R}_x \to \mathbb{R} \),
\[
\Omega_t = \{ (x, y) \in \mathbb{R}^2 : y < \eta(t, x) \}.
\]
Its free boundary, corresponding to the deformed elastic sheet, is given by
\[
\Sigma_t = \{ (x, y) \in \mathbb{R}^2 : y = \eta(t, x) \}.
\]
In the absence of gravity, the fluid motion is governed by the following system:
\begin{equation}\label{HWE1}
    \begin{cases}
        \mathbf{u}_t + \mathbf{u} \cdot \nabla \mathbf{u} = -\nabla p & \text{in } \Omega_t, \\
\operatorname{div} \mathbf{u} = 0, \quad \operatorname{curl} \mathbf{u} = 0 & \text{in } \Omega_t, \\
\eta_t=u_2-u_1\eta_x&\text{on } \Sigma_t,\\
        p = \sigma \mathbf{E}(\eta) & \text{on } \Sigma_t, \\
        \mathbf{u}(0, x) = \mathbf{u}_0(x) & \text{in } \Omega_t,
    \end{cases}
\end{equation}
where \( \mathbf{u} = (u_1, u_2) \in \mathbb{R}^2 \) denotes the fluid velocity, \( p \) is the pressure, and \( \sigma \) represents the coefficient of flexural rigidity. The term
\begin{equation}\label{Elastic}
    \sigma\mathbf{E}(\eta) =\sigma\left\{\frac{1}{\sqrt{1+\eta_x^2}} 
    \left[ \frac{1}{\sqrt{1+\eta_x^2}} 
    \left( \frac{\eta_{xx}}{(1+\eta_x^2)^{3/2}} \right)_x \right]_x
    + \frac{1}{2} \left( \frac{\eta_{xx}}{(1+\eta_x^2)^{3/2}} \right)^3\right\}
\end{equation}
represents the restoring force generated by the elastic sheet, expressed as a pressure jump across the interface;
 see Toland \textit{et al.} \cite{MR2413099,MR2812947}, Guyenne and P\u ar\u au \cite{MR3002503} and Groves \textit{et al.} \cite{MR4733013,MR3566507} for related discussions. 
 
Given the irrotational condition \( \operatorname{curl} \mathbf{u} = 0 \), we may introduce a velocity potential \( \phi \) such that \( \mathbf{u} = \nabla \phi \).
Since $\phi$ solves the Laplace equation, by the theory of elliptic PDEs, it suffices to consider its evolution on the boundary $\Gamma_t$.
Let \( \varphi(t, x) = \phi(t, x, \eta(t, x)) \) denote the trace of the velocity potential on the free surface, the kinematic and dynamic boundary conditions (\eqref{HWE1}$_3$ and \eqref{HWE1}$_4$) become:
\begin{align}
    \eta_t &=(\phi_y - \eta_x \phi_x)|_{y=\eta(t,x)}=: G(\eta)\varphi, \label{kin} \\
    \varphi_t &=\left[\frac{1}{2} \left( \phi_y^2 - \phi_x^2 \right) - \eta_x \phi_x \phi_y-p\right]_{y=\eta(t,x)}=-\frac{1}{2} \varphi_x^2+ \f12\f{(\eta_x\varphi_x+G(\eta)\varphi)^2}{1+\eta^2_x}- \sigma \mathbf{E}(\eta). \label{dyn}
\end{align}

The system of \eqref{kin} and \eqref{dyn} is the well-known Zakharov-Craig-Sulem formulation; see \cite{MR2129655,MR1158383}.
 The aim of this paper is to establish the \textbf{low-regularity well-posedness} of the hydroelastic waves.
We begin by recalling some known well-posedness results for water waves and hydroelastic waves.

\subsection{Local well-posedness results of water waves and hydroelastic waves.}
The local well-posedness of water waves has been extensively studied.
These studies typically consider models in which the elastic term \(\sigma\mathbf{E}(\eta)\) in \eqref{dyn} is replaced by one of the following:
\begin{itemize}
    \item a gravity term \(g\eta\) (where \(g\) is the gravitational acceleration),
    \item a capillary term \(-\kappa \eta_{xx}/(1+\eta_x^2)^{3/2}\) (where \(\kappa\) is the coefficient of surface tension),
    \item or a combination of both.
\end{itemize}

The small-data problem was first treated by Nalimov \cite{MR609882} \& Yoshihara \cite{MR660822, MR728155}. A major breakthrough for the local well-posedness with general data was achieved by Wu \cite{MR1471885, MR1641609}. 
Subsequent developments include the works of Christodoulou \& Lindblad \cite{MR1780703}, Lannes \cite{MR2138139}, Coutand \& Shkoller \cite{MR2291920}, Ambrose \& Masmoudi \cite{MR2162781, MR2514378}, Shatah \& Zeng \cite{MR2388661,MR2400608,MR2763036}, Alazard, Burq \& Zuily \cite{MR2805065, MR3260858}, Hunter, Ifrim \& Tataru \cite{MR3535894}, Ifrim \& Tataru \cite{MR3667289}, Ai \cite{MR4161284, MR4098033}, and Ai, Ifrim \& Tataru\cite{AIT}.

Regarding the Strichartz estimates of water waves, Christianson \textit{et al.} \cite{MR2763354}, Alazard \textit{et al.} \cite{MR2931520,MR3852259}, de Poyferr\'{e} \& Nguyen  \cite{MR3487264}, Nguyen \cite{MR3724757}, and Ai \cite{MR4161284, MR4098033, A2023} proved the Strichartz estimate arising directly from the dispersive property of free-surface water waves.

In contrast to water waves, the well-posedness theory for hydroelastic waves remains less developed. Initial progress was made by Ambrose \& Siegel \cite{MR3656704}, who established the local well-posedness of two-dimensional hydroelastic waves using a vortex sheet formulation. This was extended by Liu \& Ambrose \cite{MR3608168}, who incorporated the mass of the elastic sheet into the analysis. In later work \cite{MR4850536}, the same authors further examined the asymptotic behavior of two-dimensional hydroelastic waves with respect to various physical parameters. Recently, the second author and Wang \cite{MR4104949} proved the local well-posedness of hydroelastic waves with vorticity in arbitrary spatial dimensions. Nevertheless, to the best of our knowledge, there are no results addressing the low regularity problem for hydroelastic waves.

\subsection{Hydroelastic waves in holomorphic coordinates}
In \cite{MR3535894,MR3667289}, holomorphic coordinates were used to develop the well-posedness theory for both deep  gravity and capillary water waves, respectively.
This formulation was also used to study a variety of other water wave problems; see  \cite{MR3535894, MR3499085, MR3625189,  AIT, MR4483135, MR4462478, MR4858219, MR4891579}.
In Section $2$ of Yang \cite{MR4846707}, the second author derived the two-dimensional deep hydroelastic wave equations in holomorphic coordinates.
We refer the interested reader to \cite{MR4846707} for the detailed derivation.
Although other formulations may be applicable, we utilize holomorphic coordinates to study the well-posedness theory in this paper.

Let $\mathbf{P}:= \frac{1}{2}(\mathbf{I} - iH)$ be the holomorphic projection that selects the holomorphic portion of the complex-valued function, with $H$ being the Hilbert transform.
On the Fourier side, $\nP$ projects onto the negative frequencies.
The complex conjugate operator $\bar{\nP}$ then projects onto the anti-holomorphic component.

    Let $W$ denote the holomorphic position variable and  $Q$ the holomorphic velocity potential.
    These functions are defined on $\mathbb{R}_t \times \mathbb{R}_\al$ and take values in $\mathbb{C}$.
    Then, according to the derivations in \cite{MR4846707}, the free boundary irrotational Euler equations \eqref{HWE1} and \eqref{Elastic} are equivalent to the following system of equations:
	\begin{equation}\label{HF14}
		\begin{cases}
			&W_t+F(1+W_{\al})=0,\\
			&Q_t+FQ_{\al}+\nP\left[\f{|Q_{\al}|^2}{J}\right]
			-i\sigma \nP\left\{\f{1}{J^{\f12}}\f{d}{d\al}\left[\f{1}{J^{\f12}}\f{d}{d\al}\left(\f{W_{\al\al}}{J^{\f12}(1+W_{\al})}
			-\f{\bar{W}_{\al\al}}{J^{\f12}(1+\bar{W}_{\al})}\right)\right]\right\}\\
			&-\f{i}{2}\sigma \nP\left\{\left[\f{W_{\al\al}}{J^{\f12}(1+W_{\al})}
			-\f{\bar{W}_{\al\al}}{J^{\f12}(1+\bar{W}_{\al})}\right]^3\right\}=0,
		\end{cases}
	\end{equation}
     where $J := |1+W_\alpha|^2$ is the Jacobian, and $F = \nP\left[\frac{Q_\alpha - \bar Q_\alpha}{J}\right]$.
The system \eqref{HF14} is  fully nonlinear.
By taking the $\alpha$ derivative and diagonalizing, we introduce the differentiated variables
    \begin{equation}\label{HF18}
		\mathbf{W}=W_\al, \quad \,R=\frac{Q_\al}{1+W_\al}.
	\end{equation}
    The pair $(\bw, R)$ solves the differentiated system
    \begin{equation}\label{HF19}
		\begin{cases}
&\mathbf{W}_t+b\mathbf{W}_{\al}+\f{(1+\mathbf{W})R_{\al}}{(1+\bar{\mathbf{W}})}=(1+\mathbf{W})M,\\
&R_t+bR_{\al}+\f{ia}{1+\mathbf{W}}\\
&-\f{i\sigma}{1+\mathbf{W}} \nP\left\{\f{1}{J^{\f12}}\f{d}{d\al}\left[\f{1}{J^{\f12}}\f{d}{d\al}\left(\f{\mathbf{W}_{\al}}{J^{\f12}(1+\mathbf{W})}\right)\right]\right\}_{\al}-\f12\f{i\sigma}{1+\mathbf{W}} \nP\left[\f{\mathbf{W}^3_{\al}}{J^{\f32}(1+\mathbf{W})^3}-\f{3\mathbf{W}_{\al}|\mathbf{W}_{\al}|^2}{J^{\f52}(1+\mathbf{W})}\right]_{\al}\\
&=-\f{i\sigma}{1+\mathbf{W}} \nP\left\{\f{1}{J^{\f12}}\f{d}{d\al}\left[\f{1}{J^{\f12}}\f{d}{d\al}\left(\f{\bar{\mathbf{W}}_{\al}}{J^{\f12}(1+\bar{\mathbf{W}})}\right)\right]\right\}_{\al}-\f12\f{i\sigma}{1+\mathbf{W}} \nP\left[\f{\bar{\mathbf{W}}^3_{\al}}{J^{\f32}(1+\bar{\mathbf{W}})^3}
-\f{3\bar{\mathbf{W}}_{\al}|\mathbf{W}_{\al}|^2}{J^{\f52}(1+\bar{\mathbf{W}})}\right]_{\al},
\end{cases}
\end{equation}
where the real-valued \textit{frequency-shift} $a$ and the \textit{advection velocity} $b$ are given by
\begin{equation}\label{HF17}
		a:=i\left(\bar{P}[\bar{R}R_{\al}]-P[R\bar{R}_{\al}]\right),\, \quad b:= \nP\left[ \frac{R}{1+\bar{\bw}}\right] + \bar{\nP}\left[ \frac{\bar{R}}{1+\bw}\right].
\end{equation}
The auxiliary function $M$ is given by
\begin{equation} \label{MDef}
		M:=\f{R_{\al}}{1+\bar{\mathbf{W}}}+\f{\bar{R}_{\al}}{1+\mathbf{W}}-b_{\al}=\bar{\nP}[\bar{R}Y_{\al}-R_{\al}\bar{Y}]+\nP[R\bar{Y}_{\al}-\bar{R}_{\al}Y].
	\end{equation}
Here we use the notation $Y:=\f{\bw}{1+\bw}$.
 The variable $R$ also has an intrinsic meaning:  it represents the complex conjugate of the complex velocity evaluated at (or restricted to) the water surface.
    
We remark that \eqref{HF19} is a diagonal and self-contained system.
It is invariant under spatial translations and admits the scaling
\begin{equation} \label{Scaling}
 \left(\bw(t,\alpha), R(t,\alpha)\right) \to \left( \bw(\lambda^\f52 t,\lambda \alpha), 
\lambda^{\frac32} R(\lambda^\f52 t,\lambda \alpha)\right).
\end{equation}
In the remainder of this paper, we study the low-regularity well-posedness of \eqref{HF19}.
For simplicity, the flexural rigidity coefficient $\sigma$ is normalized to be $1$.
 We also note that the setting and analysis in the periodic case differ slightly from those on the real line; see Appendix $A.4$ in \cite{MR3535894} for an explanation. 
We focus exclusively on the analysis on $\R$.
	
\subsection{Main results of this paper}
We state the main results of the paper. 
As in \cite{MR3667289}, we first introduce function spaces for measuring the Sobolev regularity of hydroelastic waves. 
A simplified model for \eqref{HF19} is given by its linearization around the zero solution:
\begin{equation}
\left\{
             \begin{array}{lr}
             \partial_t w + \p_\alpha r = 0, &  \\
             \p_t r  -i\p^4_\alpha w =0,&
             \end{array}
\right.\label{e:ZeroLinear}
\end{equation}
restricted to holomorphic functions.
\eqref{e:ZeroLinear} is a system of equations that can be written as a linear dispersive equation
\begin{equation*}
    \p_t^2w+  i \p^5_\al w =0.
\end{equation*}
Its dispersion relation is given by
\begin{equation*}
    \tau^2 + \xi^5 =0, \quad \xi< 0,
\end{equation*}
which is different from the dispersion relation of linearized gravity or capillary water waves.

A conserved energy associated with \eqref{e:ZeroLinear} is given by
\begin{equation}
    \mathcal{E}_0(w,r) = \int -iw_\al \bar{w}_{\al \al}  + |r|^2\, d\alpha = \|w\|_{\dot{H}^\f32}^2 + \|r\|_{L^2}^2. \label{ConservedEnergy}
\end{equation}
This conserved energy suggests the functional framework to study \eqref{HF19}.
The system \eqref{e:ZeroLinear} is well-posed in the product space $\dot{H}^\f32 \times L^2$.
Motivated by the linearized analysis, we introduce inhomogeneous product Sobolev spaces and the corresponding inhomogeneous product Zygmund spaces:
	\[
	\mathcal{H}^s :=H^{s+\f32}\times H^{s}, \quad \mathcal{W}^r :=  C^{r+\f32}_{*}\times C^r_{*}.
	\]
The hydroelastic wave system \eqref{HF19} is a quasilinear dispersive PDE system of order $\f{5}{2}$.

Ambrose and Siegel \cite{MR3656704} showed that the system \eqref{HF14} is locally well-posed in $\mathcal{H}^{n}$ when $n$ is large enough.
Wang and Yang obtained a refined result for hydroelastic waves  in \cite{MR4104949} showing that the system \eqref{HF19} is locally well-posed in $\mathcal{H}^{\f52k}$ for $k>2$. 
Since the flexural elasticity $\mathbf{E}(\eta)$ is highly nonlinear and strongly geometry-dependent, they adopted the geometric framework developed by Shatah \& Zeng \cite{MR2388661,MR2400608,MR2763036}, formulating the problem as a dispersive equation for $\mathbf{E}(\eta)$. 
This approach, while robust, necessitates relatively high regularity assumptions on the initial data.
	
To state the main results in this paper, we define the control norms that will be used in the energy estimate.
Let $0<\epsilon<\epsilon^{'} \ll 1$ be two positive constants, and $s\geq 0$. Define
\begin{equation*}
\mathcal{A}_s : = \|(\bw, R) \|_{\mathcal{W}^{s- \f32+\epsilon}}, \quad \mathcal{A}_{\sharp,s} : = \|(\bw, R) \|_{W^{s+\f14+\epsilon^{'},4}(\mathbb{R})\times W^{s-\f54+\epsilon^{'},4}(\mathbb{R})}.
\end{equation*}
Using Sobolev embedding, $\mathcal{A}_s \lesssim \mathcal{A}_{\sharp, s}$.
It is noted that $\mathcal{A}_0$ is scale-invariant with respect to the scaling \eqref{Scaling} up to the presence of the small parameter $\epsilon$.
The small constants $\epsilon$ and $\epsilon^{'}$ are needed to avoid the delicate endpoint Sobolev estimates.
	
The first main result of this paper is an improved cubic energy estimate. 
\begin{theorem} \label{t:MainEnergyEst}
 Let $(\bw, R)$ solve the hydroelastic wave system \eqref{HF19}.
 For any $s> 0$, there exists an energy functional $E_s(\bw, R)$ that has the following properties:
 \begin{enumerate}
     \item Norm equivalence:
       \begin{equation}
           E_s(\mathbf{W},R) \approx_\CalAZ \|(\bw,  R)\|^2_{\mathcal{H}^s}.\label{normEquivalence}
       \end{equation}
     \item Improved energy estimates:
      \begin{equation}
     \frac{d}{dt}E_s\lesssim_\CalAZ \ASSharp  E_s.\label{strongcubicest}
 \end{equation} 
 \end{enumerate} 
\end{theorem}

\begin{remark}
For comparison, in the cubic energy estimate in \cite{MR4846707}, the constant in the energy estimate is $\CalAZ\mathcal{A}_{\f52}$.
The Sobolev index $\f74$ in \eqref{strongcubicest} is much smaller than $\f52$.
In addition, Theorem \ref{t:MainEnergyEst} works for any positive index $s>0$, while the corresponding energy estimate in \cite{MR4846707} only states the result for $s= n+\f12$, where $n\geq 2$ is an integer.
\end{remark}
\begin{remark}
When the gravity is taken into account, one can prove almost the same result as Theorem \ref{t:MainEnergyEst}.
The only difference is that in \eqref{strongcubicest}, the constant $\ASSharp$ is replaced by $g+\ASSharp$.
This is because the gravity term is of lower order than the elastic terms and does not play a significant role in the analysis.
\end{remark}
	
The second main result of this paper is the local well-posedness result of the hydroelastic wave system \eqref{HF19}.

\begin{theorem} \label{t:HydroWellPosed}
Let $s> \f34$, the system \eqref{HF19} is locally well-posed  in $\H^s(\mathbb{R})$ $(\text{or }\, \H^s(\mathbb{T}) )$.
Moreover, the solutions of \eqref{HF19} exist on $[0,T]$ as long as $\mathcal{A}_0(t) \lesssim 1$ and $\ASSharp(t) \in L^2([0,T])$.
\end{theorem}
Well-posedness is understood in the sense of Hadamard, namely existence, uniqueness, and continuous dependence on initial data.

We also address the related work of Alazard, Kukavica \& Tuffaha  \cite{MR4905081} and Alazard, Shao \& Yang \cite{ACH}.
They proved the global existence of water waves where the interface evolution is governed by the law of linear elasticity.
The equation in their model can be reduced to a Schr\"odinger type equation whose leading dispersive term has a constant coefficient.
In our model, the dynamic boundary condition \eqref{dyn} with nonlinear elasticity is different.
The hydroelastic waves \eqref{HF19} can be rewritten as a quasilinear dispersive equation of order $\f52$ instead of a Schr\"odinger type equation.
Consequently, both the structure of the governing equations and the analytical methods employed here differ substantially from those in \cite{MR4905081, ACH}.

\subsection{Key Difficulties and Strategy of Proof}
The primary challenge in establishing the low-regularity well-posedness for the hydroelastic wave system \eqref{HF19} lies in its quasilinear nature and the high order of the dispersive term. 
The system is a dispersive equation of order $\f52$. 
In standard Sobolev spaces, the nonlinearity induces a ``loss of derivatives" in the energy estimates, where the time derivative of the energy norm is controlled only by norms of higher order. 
To overcome this, we employ a strategy based on paradifferential calculus and the method of modified energy, inspired by the works of Hunter-Ifrim-Tataru \cite{MR3535894}, and Ai-Ifrim-Tataru \cite{AIT} on gravity water waves.

\textbf{1. Quasilinear Structure and Holomorphic Coordinates.} We work in holomorphic coordinates, which effectively diagonalizes the system and simplifies the structure of the Dirichlet-Neumann operator $G(\eta)$. 
However, the resulting system \eqref{HF19} remains quasilinear. 
Our first step is to paralinearize the equations. 
By decomposing the nonlinearity into paradifferential operators (low-high interactions) and  remainders (high-high interactions), we isolate the principal transport and dispersive terms. 
Roughly speaking, this reduction allows us to treat the equations as a linear dispersive system with rough, variable coefficients, plus perturbative source terms. Notably, the highly nonlinear nature of the elastic term $\mathbf{E}(\eta)$ makes this step particularly intricate and subtle compared to the case of gravity/capillary water waves.

\textbf{2. The Modified Energy Functional.}
Standard energy estimates are insufficient due to the variable coefficients in the leading-order terms. 
A crucial component of our proof is the construction of a modified energy functional $E_s(\bw, R)$ in Theorem \ref{t:MainEnergyEst} and the linearized modified energy $E^{s,para}_{lin}(w,r)$ in Proposition \ref{t:Hswellposedflow}. 
Unlike the standard $H^s$ energies, the principal part of our functional includes a carefully chosen paradifferential weight that depends on the Jacobian $J$ and the parameter $s$.
Specifically, for the linearized variables $(w,r)$, we construct the principal part of the modified energy of the form in Section \ref{s:HomoFlow}:
\begin{equation*} 
    E_{lin}(w,r) \approx \int \Im \left(T_{J^{-\f54}}w_{\al}\cdot \bar{w}_{\al \al} \right) + \Re \left(r\cdot T_{J^{\f14}}\bar{r}\right) \,d\al .
\end{equation*}
The specific powers of the Jacobian $J^{-\f54}$ and $J^{\f14}$ are critical.
They are selected to ensure that the leading-order contributions from the time derivatives of the para-coefficients cancel out exactly with the leading-order non-perturbative terms arising from the elastic nonlinearity. This cancellation is essential to close the energy estimates without losing regularity.

\textbf{3. Normal Forms and Low Regularity.}
To reach the low regularity threshold $s>3/4$, we must control cubic and quartic nonlinear interactions that essentially behave as perturbative source terms in the energy estimate. 
We employ paradifferential normal form transformations to eliminate the non-perturbative portions of these terms. 
We define a transformation $(\bw,R) \mapsto( \bw_{NF}, R_{NF})$ such that the new variables satisfy a ``better" paradifferential system where the cubic nonlinearities are either null forms or have favorable derivative structures.
The feasibility of this normal form method relies on the non-resonant structure of the hydroelastic dispersion relation \eqref{TauXiRelation}.
As discussed in Appendix \ref{s:Discuss}, the strict convexity of the function $|\xi|^\f52$ ensures that three-wave resonances cannot occur for nonzero frequencies, and four-wave resonances are avoided in the relevant interaction regimes. 
This allows us to define bounded bilinear and trilinear symbols that remove the problematic terms from the energy inequality.

The remainder of this paper is organized as follows.
Section \ref{s:ParaMat} is devoted to the computation of para-material derivatives for both the full hydroelastic waves and the linearized hydroelastic waves.
In Section \ref{s:EstLin}, we derive the modified energy estimate of the linearized hydroelastic waves \eqref{linearizedeqn}, namely Theorem \ref{t:LinearizedWellposed}.
Then, in Section \ref{s:Energy}, we prove the modified energy estimate for the full hydroelastic waves Theorem \ref{t:MainEnergyEst}.
We briefly explain how to obtain the low regularity well-posedness of two-dimensional deep hydroelastic waves following the argument in \cite{AIT} in Section \ref{s:LowWellPosed}.
As for the appendix, we put the paradifferential estimates that we will need in this paper in Appendix \ref{s:Norms}.
We will rewrite  different variants of hydroelastic waves in the paradifferential format, so that these paradifferential estimates play a crucial role.
In Appendix \ref{s:HydroWaveEst}, we recall some estimates for several auxiliary functions in Sobolev and Zygmund spaces.
Finally, in Appendix \ref{s:Discuss}, we discuss three-wave and four-wave interactions. 
We will perform paradifferential quadratic/cubic normal forms and construct cubic/quartic modified energy later in the analysis, which require appropriate non-resonant conditions.
   
\section{Para-material derivatives and the linearized hydroelastic waves} \label{s:ParaMat}
In holomorphic coordinates, the material derivative is defined to be $D_t = \partial_t + b \partial_\alpha$.
At the paradifferential level, it is replaced by the para-material derivative $T_{D_t} = \partial_t + T_b \partial_\alpha$. 
In this section, we compute the leading contributions to the para-material derivatives of variables $(\bw, R)$ and the weight $J^s$.  
In addition, we compute and simplify the linearized hydroelastic wave system, and identify the principal terms of para-material derivatives of linearized variables $(w, r)$. 
These formulas will play a key role in estimating the time derivative of the modified energy in later sections.
We rewrite each equation as the paradifferential equation.
By carefully separating principal contributions from lower-order terms, and by treating all perturbative components as error terms, we obtain the leading terms of para-material derivatives of each variable.
  
\subsection{Leading terms of para-material derivatives}
We first compute the leading terms of the para-material derivatives of $\bw$.
\begin{lemma} 
The unknown $\bw$ satisfies the paradifferential equation:
\begin{equation} \label{WParaMat}
       T_{D_t}\bw  + T_{(1+\bw)(1-\bar{Y})}R_\alpha + T_{1-\bar{Y}}T_{\bw_\alpha}R -T_{(1+\bw)(1-\bar{Y})^2}T_{\bar{\bw}_\alpha}R = G,  
\end{equation}
where the source term $G$  satisfies the bound
\begin{equation*}
 \|G\|_{C_{*}^1} \lesssim_\CalAZ \mathcal{A}^2_\f74, \quad  \|G\|_{H^\f32} \lesssim_\CalAZ \ASSharp.
\end{equation*}
\end{lemma}
\begin{proof}
We rewrite the first equation in \eqref{HF19} in the form of paradifferential equations 
\begin{align*}
&T_{D_t}\bw + T_{(1+\bw)(1-\bar{Y})} R_\alpha + T_{\bw_\alpha}T_{1-\bar{Y}}R -T_{1+\bw}T_{\bar{Y}_\alpha}R\\
=&  \left(T_{(1+\bw)(1-\bar{Y})}-T_{1+\bw}T_{1-\bar{Y}}\right)R_\alpha 
  -T_{b_\alpha}\bw   + \partial_\alpha T_{1+\bw}\nP\Pi(\bar{Y}, R)-\partial_\alpha\nP\Pi(\bw, b).
\end{align*}
Using \eqref{CompositionTwo}, \eqref{BCOneStar}, \eqref{CCCEstimate} and \eqref{CsCmStar} to estimate the Zygmund bound of the right-hand side of above equation, we write
\begin{align*}
&\| (T_{1+\bw}T_{1-\bar{Y}}-T_{(1+\bw)(1-\bar{Y})})R_\alpha \|_{C_{*}^1} \lesssim_\CalAZ \mathcal{A}_\f74 \|R\|_{C_{*}^{\f14}}  \lesssim_\CalAZ \mathcal{A}^2_\f74, \\
& \|T_{b_\alpha}\bw \|_{C_{*}^1}\lesssim \|b\|_{C^{\f14}_{*}}\| \bw\|_{C_{*}^{\f74}}  \lesssim_\CalAZ \mathcal{A}^2_\f74, \\
& \|\partial_\alpha T_{1+\bw}\nP\Pi(\bar{Y}, R) \|_{C_{*}^1}\lesssim_\CalAZ \|Y\|_{C^{\f74}_*}\|R\|_{C^\f14_*}\lesssim_\CalAZ \mathcal{A}^2_\f74, \\
&\|\partial_\alpha\nP\Pi(\bw, b) \|_{C_{*}^1}\lesssim_\CalAZ \|\bw\|_{C^{\f74}_*}\|b\|_{C^\f14_*}\lesssim_\CalAZ \mathcal{A}^2_\f74.
\end{align*}
Similarly, for Sobolev bound, we obtain
\begin{align*}
&\| (T_{1+\bw}T_{1-\bar{Y}}-T_{(1+\bw)(1-\bar{Y})})R_\alpha \|_{H^\f32} \lesssim_\CalAZ \mathcal{A}_{\sharp,\f74} \|R\|_{W^{\f12,4}}  \lesssim_\CalAZ \ASSharp, \\
& \|T_{b_\alpha}\bw \|_{H^\f32}\lesssim \|b\|_{W^{\f12,4}}\| \bw\|_{W^{2,4}}  \lesssim_\CalAZ \ASSharp, \\
& \|\partial_\alpha T_{1+\bw}\nP\Pi(\bar{Y}, R) \|_{H^\f32}\lesssim_\CalAZ \|Y\|_{W^{2,4}}\|R\|_{W^{\f12,4}}\lesssim_\CalAZ \ASSharp, \\
&\|\partial_\alpha\nP\Pi(\bw, b) \|_{H^\f32}\lesssim_\CalAZ \|\bw\|_{W^{2,4}}\|b\|_{W^{\f12,4}}\lesssim_\CalAZ \ASSharp.
\end{align*}
Furthermore, applying the symbolic calculus rules \eqref{CompositionPara} and \eqref{CompositionTwo} gives
\begin{equation*}
T_{\bw_\alpha}T_{1-\bar{Y}}R -T_{1+\bw}T_{\bar{Y}_\alpha}R =  T_{1-\bar{Y}}T_{\bw_\alpha}R -T_{(1+\bw)(1-\bar{Y})^2}T_{\bar{\bw}_\alpha}R + G.
\end{equation*}
This leads to the equation for $T_{D_t}\bw$ \eqref{WParaMat}.
\end{proof}
We note that estimating the error term in $L^\infty$ norm, then
\begin{equation*}
    \bw_t = - T_{(1+\bw)(1-\bar{Y})}R_\alpha + G, \quad \|G\|_{L^\infty} \lesssim_\CalAZ \mathcal{A}^2_\f74.
\end{equation*}

Next, we compute the leading terms of the para-material derivatives of $R$.
\begin{lemma}
The leading term of the para-material derivative of $R$ is given by
\begin{equation}
T_{D_t}R =  iT_{J^{-\frac{3}{2}}(1-Y)^2}\p^4_\alpha\bw -\frac{15}{2}iT_{J^{-\f32}(1-Y)^3 \bw_\al} \p_\al^3 \bw - \f52iT_{J^{-\f52}(1-Y)\bar{\bw}_\al}\p_\al^3 \bw  +K, \label{RParaMat}
\end{equation}
and the error term $K$ satisfies the estimate
\begin{equation*}
 \| K\|_{L^2} \lesssim_{\CalAZ} \ASSharp.
\end{equation*}
\end{lemma}

\begin{proof}
The second equation of the system \eqref{HF19} can be rewritten as
\begin{equation} \label{RParaEqn}
T_{D_t}R = -i\nP[a(1-Y)] - \nP\Pi(R_\alpha, b) - \nP T_{R_\alpha}b + \text{elastic terms},  
\end{equation}
where the elastic terms are given by 
\begin{align*}
    &i(1-Y)\nP\p_\al\{J^{-\f12}\p_\al[J^{-\f12}\p_\al(J^{-\f12}(1-Y)\bw_\alpha)] \} \\
    -& i(1-Y)\nP\p_\al\{J^{-\f12}\p_\al[J^{-\f12}\p_\al(J^{-\f12}(1-\bar{Y})\bar{\bw}_\alpha)] \} \\
    +&\frac{i}{2}(1-Y)\nP\p_\al[J^{-\f32}(1-Y)^3 \bw_\alpha^3 -3J^{-\f52}(1-Y)\bw_\alpha |\bw_\al|^2]\\
    -&\frac{i}{2}(1-Y)\nP\p_\al[J^{-\f32}(1-\bar{Y})^3 \bar{\bw}_\alpha^3 -3J^{-\f52}(1-\bar{Y})\bar{\bw}_\alpha |\bw_\al|^2].
\end{align*}

For the non-elastic terms on the right-hand side of \eqref{RParaEqn}, we compute using \eqref{aEst}
\begin{align*}
 &\|a(1-Y) \|_{L^2} \lesssim (1+ \|Y\|_{L^\infty})\|a\|_{L^2} \lesssim_\CalAZ \ASSharp, \\
 &\|\nP\Pi(R_\alpha, b)\|_{L^2} + \|\nP T_{R_\alpha}b \|_{L^2} \lesssim \|R\|_{W^{\frac{1}{2}+\epsilon,4}} \|b\|_{W^{\f12,4}} \lesssim \ASSharp,
\end{align*}
so that these three terms can be absorbed into $K$.
It suffices to simplify these elastic terms.

We compute
\begin{equation*}
 J^{-\f12}\p_\al(J^{-\f12}(1-Y)\bw_\alpha) = J^{-1}(1-Y)\bw_{\al \al} -\f32 J^{-1}(1-Y)^2\bw^2_\al -\f12 J^{-2}|\bw_\al|^2,
\end{equation*}
so that we get
\begin{align*}
   &J^{-\f12}\p_\al[J^{-\f12}\p_\al(J^{-\f12}(1-Y)\bw_\alpha)] \\
   =&  J^{-\f32}(1-Y)\p_\al^3\bw -2J^{-\f32}(1-Y)^2\bw_\al \bw_{\al \al}-J^{-\f52}\bar{\bw}_\al \bw_{\al \al} \\
   &-3J^{-\f32}(1-Y)^2\bw_\al \bw_{\al \al} + \f92 J^{-\f32}(1-Y)^3\bw_\al^3+\f32 J^{-\f52}(1-Y)|\bw_\alpha |^2\bw_\al\\
   &- \f12 J^{-\f52}\bw_{\al \al}\bar{\bw}_\al - \f12 J^{-\f52}\bar{\bw}_{\al \al}\bw_\al +  J^{-\f52}(1-Y) |\bw_\alpha |^2\bw_\al \\
   &+ J^{-\f52}(1-\bar{Y}) |\bw_\alpha |^2\bar{\bw}_\al.
\end{align*}
We then use the above computation to expand the elastic terms,
\begin{align*}
 \text{Elastic terms} =& i(1-Y)\nP\p_\al [J^{-\f32}(1-Y)\p_\al^3\bw -5J^{-\f32}(1-Y)^2\bw_\al \bw_{\al \al}\\
 &-J^{-\f52}\bar{\bw}_\al \bw_{\al \al}  + 5 J^{-\f32}(1-Y)^3\bw_\al^3 -J^{-\f32}(1-\bar{Y})\p_\al^3\bar{\bw} \\
 & +5J^{-\f32}(1-\bar{Y})^2\bar{\bw}_\al \bar{\bw}_{\al \al}+ J^{-\f52}\bw_\al \bar{\bw}_{\al \al}-5 J^{-\f32}(1-\bar{Y})^3\bar{\bw}_\al^3].
\end{align*}
Putting perturbative terms into $K$, we get  
\begin{align*}
 \text{Elastic terms} =& i(1-Y)\Big[T_{J^{-\f32}(1-Y)}\p_\al^4 \bw -\f52T_{J^{-\f32}(1-Y)^2 \bw_\al} \p_\al^3 \bw - \f32T_{J^{-\f52}\bar{\bw}_\al}\p_\al^3 \bw \\
 &-5T_{J^{-\f32}(1-Y)^2 \bw_\al} \p_\al^3 \bw - T_{J^{-\f52}\bar{\bw}_\al}\p_\al^3 \bw\Big] + K \\
 =&iT_{1-Y}\Big[T_{J^{-\f32}(1-Y)}\p_\al^4 \bw -\frac{15}{2}T_{J^{-\f32}(1-Y)^2 \bw_\al} \p_\al^3 \bw - \f52T_{J^{-\f52}\bar{\bw}_\al}\p_\al^3 \bw \Big] + K \\
 =& iT_{J^{-\frac{3}{2}}(1-Y)^2}\p^4_\alpha\bw -\frac{15}{2}iT_{J^{-\f32}(1-Y)^3 \bw_\al} \p_\al^3 \bw - \f52iT_{J^{-\f52}(1-Y)\bar{\bw}_\al}\p_\al^3 \bw + K.
\end{align*}
Collecting the expressions of all terms in the equation \eqref{RParaEqn}, we obtain the principal part of $T_{D_t}R$ in \eqref{RParaMat}.
\end{proof}

We now turn to the computation of the time derivative of the weight $J^s$ for $s\neq 0$.
\begin{lemma} \label{t:JsParaMat}
The time derivative of $J^s$ satisfies 
\begin{equation*}
\partial_t J^s = -sJ^s b_\alpha +E=  -s T_{J^s(1-\bar{Y})}R_\alpha -s T_{J^s (1-Y)}\bar{R}_\alpha+E, \quad   \| E\|_{C_{*}^{\epsilon}} \lesssim_\CalAZ \mathcal{A}^2_{\f74}.
\end{equation*}
\end{lemma}
\begin{proof}
A direct computation yields
\begin{equation*}
\p_t J^s =sJ^{s}(1-Y)\p_t\bw + sJ^{s}(1-\bar{Y})\p_t\bar{\bw}.
\end{equation*}
The first equation of \eqref{HF19} can be rewritten as
\begin{align*}
\p_t\bw =& -(1+\bw)(1-\bar{Y})R_\alpha -b\bw_\al  +(1+\bw)M \\
=& -(1+\bw)(1-\bar{Y})R_\alpha +E,
\end{align*}
where, as in the derivation of \eqref{WParaMat}, the terms $b\bw_\al  +(1+\bw)M$ are absorbed into the error term $E$.
As a result of the identity $(1+\bw)(1-Y)=1$, we have 
\begin{equation*}
\p_tJ^s = -sJ^{s}[(1-\bar{Y})R_\alpha + (1-Y)\bar{R}_\alpha]+E = -s T_{J^s(1-\bar{Y})}R_\alpha -s T_{J^s (1-Y)}\bar{R}_\alpha + E.
\end{equation*}
Finally, invoking identity \eqref{MDef} together with estimate \eqref{MBound} for $M$, we obtain
\begin{equation*}
 \p_t J^s  =  -sJ^{s}(b_\alpha +M)+E = -sJ^{s}b_\alpha +E,
\end{equation*}
which completes the proof for $s\neq 0$.
\end{proof}

\subsection{The derivation of the linearized hydroelastic waves}
In this section, we compute the linearized hydroelastic wave system.
Following the framework in gravity  water waves \cite{MR3535894} and capillary water waves \cite{MR3667289}, we denote the solutions for the linearized hydroelastic waves around a solution $(W, Q)$ to the equation \eqref{HF14} by $(w,q)$, and perturbative terms $(G,K)$ for source terms that satisfy
\begin{equation*}
    \|(G,K)\|_{\H^0} \lesssim_\CalAZ \ASSharp \|(w,r)\|_{\H^0}.
\end{equation*}
In what follows, we will repeatedly employ the symbolic calculus \eqref{CompositionPara} to combine or exchange para-coefficients.
We will also use para-associativity Lemma \ref{t:Paraassociavity} to commute the para-coefficients and the balanced paraproducts.
These computations are performed without explicit exposition, and commutator errors are absorbed into the perturbative source terms $(G,K)$.

The linearizations of non-elastic terms are the same as in the capillary water waves \cite{MR3667289}.
To compute the linearizations of the elastic terms, we first introduce the auxiliary functions $\tilde{c}$ and $c$ by
\begin{equation*}
i\tilde{c}:=\f{\mathbf{W}_{\al}}{J^{\f12}(1+\mathbf{W})}
-\f{\bar{\mathbf{W}}_{\al}}{J^{\f12}(1+\bar{\mathbf{W}})}, \quad
ic:=iJ^{-\f12} \p_\al(J^{-\f12}\tilde{c}_\al)+\f12(i\tilde{c})^3.
\end{equation*}
A direct computation shows that the linearization of $J^{-\f12}$ is
\begin{equation*}
\delta(J^{-\f12})=-\frac{(1+\mathbf{W})\bar{w}_{\al}+(1+\bar{\mathbf{W}})w_{\al}}{2J^{\f32}}.
\end{equation*}
We then obtain the linearization of $ic$, which is $ i\delta c=p-\bar{p}$, where
\begin{align*}
p:=&\f{1}{J^{\f12}}\f{d}{d\al}\left[\f{1}{J^{\f12}}\f{d}{d\al}\left(\f{w_{\al\al}}{J^{\f12}(1+\mathbf{W})}-\left(\f{3\mathbf{W}_{\al}}{2J^{\f12}(1+\mathbf{W})^2}-\f{\bar{\mathbf{W}}_{\al}}{2J^{\f32}}\right)w_{\al}\right)\right]\\
&-\f{1}{J^{\f12}}\f{d}{d\al}\left[\frac{(1+\bar{\mathbf{W}})w_{\al}}{2J^{\f32}}\f{d}{d\al}\left(\f{\mathbf{W}_{\al}}{J^{\f12}(1+\mathbf{W})}-\f{\bar{\mathbf{W}}_{\al}}{J^{\f12}(1+\bar{\mathbf{W}})}
\right)\right]\\
&-\frac{(1+\bar{\mathbf{W}})w_{\al}}{2J^{\f32}}\f{d}{d\al}\left[\f{1}{J^{\f12}}\f{d}{d\al}\left(\f{\mathbf{W}_{\al}}{J^{\f12}(1+\mathbf{W})}-\f{\bar{\mathbf{W}}_{\al}}{J^{\f12}(1+\bar{\mathbf{W}})}
\right)\right]\\
&+\f32\left(\f{w_{\al\al}}{J^{\f12}(1+\mathbf{W})}-\left(\f{3\mathbf{W}_{\al}}{2J^{\f12}(1+\mathbf{W})^2}-\f{\bar{\mathbf{W}}_{\al}}{2J^{\f32}}\right)w_{\al}\right)\left[\f{\mathbf{W}_{\al}}{J^{\f12}(1+\mathbf{W})}
-\f{\bar{\mathbf{W}}_{\al}}{J^{\f12}(1+\bar{\mathbf{W}})}\right]^2.
\end{align*}

By combining the above computations with those in \cite{MR3667289}, the linearized hydroelastic waves are as follows:
\begin{equation} \label{wqLinearized}
\begin{cases}
w_t+Fw_{\al}+(1+\mathbf{W})\nP[m-\bar{m}]=0,\\
q_t+Fq_{\al}+Q_{\al}\nP[m-\bar{m}]+\nP[n+\bar{n}]-i\nP[p-\bar{p}]=0.
\end{cases}
\end{equation}
Here, we introduce the diagonal linearized variable $r:=q-Rw$, and define
\begin{align*}
&m:=\f{q_{\al}-Rw_{\al}}{J}+\f{\bar{R}w_{\al}}{(1+\mathbf{W})^2}=\f{r_{\al}+R_{\al}w}{J}+\f{\bar{R}w_{\al}}{(1+\mathbf{W})^2}, \\
&n:=\bar{R}\delta R=\f{\bar{R}(q_{\al}-Rw_{\al})}{1+\mathbf{W}}=\frac{\bar{R}(r_{\al}+R_{\al}w)}{1+\mathbf{W}}.
\end{align*}
Introducing the transport coefficient  $b=F+\frac{\bar{R}}{1+\mathbf{W}}$, and writing the equations in terms of the associated material derivative, the system \eqref{wqLinearized} may be rewritten as
\begin{equation} \label{linearizedeqn}
\begin{cases}
(\p_t+b\p_{\al})w+\f{1}{1+\bar{\mathbf{W}}}r_{\al}+\frac{R_{\al}}{1+\bar{\mathbf{W}}}w=\mathcal{G}_0(w,r),\\
(\p_t+b\p_{\al})r-i\frac{a}{1+\mathbf{W}}w-i\nP p-\frac{w}{1+\mathbf{W}}\p_{\al}\nP c=\mathcal{K}_0(w,r),
\end{cases}
\end{equation}
where on the right-hand side,
\begin{equation*}
\mathcal{G}_0(w,r)=(1+\mathbf{W})(\nP\bar{m}+\bar{\nP}m),\quad
\mathcal{K}_0(w,r)=\bar{\nP}n-\nP\bar{n}-i\nP\bar{p}.
\end{equation*}

In the remainder of the paper, we devote substantial effort to the analysis of the linearized system \eqref{linearizedeqn} and various variants of \eqref{linearizedeqn}.
In order to obtain a modified energy estimate for \eqref{linearizedeqn} at low regularity, we will first rewrite the system in paradifferential form.

The reduction of \eqref{linearizedeqn} to a paradifferential system follows closely the corresponding computation for gravity water waves \cite{AIT}.
The principal new difficulty lies in the elastic contribution $-i\nP p$, whose structure requires additional analysis.
To facilitate the computation, we first prove the following lemma.
\begin{lemma}
The term $(1+\mathbf{W})p$ admits the following decomposition:
\begin{equation} \label{pexpansion}
\begin{split}
(1+\mathbf{W})p=&\p_{\al}(J^{-\f12}\p_{\al}(J^{-\f12}\p_{\al}(J^{-\f12}w_{\al})))\\
&-i\left[\p_{\al}\left(J^{-\f12}\p_{\al}\left(\tilde{c}J^{-\f12}w_{\al}\right)\right)
+\p_{\al}\left(\tilde{c}J^{-\f12}\p_{\al}\left(J^{-\f12}w_{\al}\right)\right)\right]\\
&-\f{11}4\p_{\al}\left(\tilde{c}^2J^{-\f12}w_{\al}\right)
-icw_{\al} + K.
\end{split}    
\end{equation}
\end{lemma}

\begin{proof}
Recall that $p$ is given by
\begin{align*}
p:=&\f{1}{J^{\f12}}\f{d}{d\al}\left[\f{1}{J^{\f12}}\f{d}{d\al}\left(\f{w_{\al\al}}{J^{\f12}(1+\mathbf{W})}-\left(\f{3\mathbf{W}_{\al}}{2J^{\f12}(1+\mathbf{W})^2}-\f{\mathbf{W}_{\al}}{2J^{\f32}}\right)w_{\al}\right)\right]\\
&-\f{1}{J^{\f12}}\f{d}{d\al}\left[\frac{(1+\bar{\mathbf{W}})w_{\al}}{2J^{\f32}}\f{d}{d\al}\left(\f{\mathbf{W}_{\al}}{J^{\f12}(1+\mathbf{W})}-\f{\bar{\mathbf{W}}_{\al}}{J^{\f12}(1+\bar{\mathbf{W}})}
\right)\right]\\
&-\frac{(1+\bar{\mathbf{W}})w_{\al}}{2J^{\f32}}\f{d}{d\al}\left[\f{1}{J^{\f12}}\f{d}{d\al}\left(\f{\mathbf{W}_{\al}}{J^{\f12}(1+\mathbf{W})}-\f{\bar{\mathbf{W}}_{\al}}{J^{\f12}(1+\bar{\mathbf{W}})}
\right)\right]\\
&+\f32\left(\f{w_{\al\al}}{J^{\f12}(1+\mathbf{W})}-\left(\f{3\mathbf{W}_{\al}}{2J^{\f12}(1+\mathbf{W})^2}-\f{\bar{\mathbf{W}}_{\al}}{2J^{\f32}}\right)w_{\al}\right)\left[\f{\mathbf{W}_{\al}}{J^{\f12}(1+\mathbf{W})}
-\f{\bar{\mathbf{W}}_{\al}}{J^{\f12}(1+\bar{\mathbf{W}})}\right]^2,
\end{align*}
which we decompose as
\begin{equation*}
   p :=A-B-C+D.
\end{equation*}
A direct computation shows that
\begin{align*}
\f{w_{\al\al}}{J^{\f12}(1+\mathbf{W})}-\left(\f{3\mathbf{W}_{\al}}{2J^{\f12}(1+\mathbf{W})^2}-\f{\bar{\mathbf{W}}_{\al}}{2J^{\f32}}\right)w_{\al}=\f{\p_{\al}(J^{-\f12}w_{\al})-i\tilde{c}w_{\al}}{1+\mathbf{W}}=:\f{A_0}{1+\mathbf{W}}.
\end{align*}
A straightforward calculation gives that
\begin{align*}
&J^{-\f12}\p_{\al}\left(\frac{A_0}{1+\mathbf{W}}\right)=\p_{\al}\left(\frac{J^{-\f12}A_0}{1+\mathbf{W}}\right)
-\frac{\p_{\al}J^{-\f12}A_0}{1+\mathbf{W}}\\
=&\frac{\p_{\al}\left(J^{-\f12}A_0\right)}{1+\mathbf{W}}-\f{\mathbf{W}_{\al}J^{-\f12}A_0}{(1+\mathbf{W})^2}
-\frac{\p_{\al}J^{-\f12}A_0}{1+\mathbf{W}}
=\f{\p_{\al}(J^{-\f12}A_0)-\f{i\tilde{c}A_0}{2}}{1+\mathbf{W}}=:\f{A_1}{1+\mathbf{W}}.
\end{align*}
Replace $A_0$ by $A_1$ in the above equations, we then get
\begin{align*}
A=&J^{-\f12}\p_{\al}\left(\frac{A_1}{1+\mathbf{W}}\right)=\f{\p_{\al}(J^{-\f12}A_1)-\f{i\tilde{c}A_1}{2}}{1+\mathbf{W}}\\
=&\f{1}{1+\mathbf{W}}\left(\p_{\al}(J^{-\f12}\p_{\al}(J^{-\f12}A_0))-\f12\p_{\al}\left((i\tilde{c})J^{-\f12}A_0\right)-\f{i\tilde{c}\p_{\al}(J^{-\f12}A_0)}{2}+\f{(i\tilde{c})^2A_0}{4}\right).
\end{align*}
Similarly, $B$ is given by
\begin{align*}
B=\f12J^{-\f12}\p_{\al}\left(\f{J^{-\f12}(i\tilde{c})_{\al}w_{\al}}{1+\mathbf{W}}\right)=\f12\f{\p_{\al}(J^{-\f12}J^{-\f12}(i\tilde{c})_{\al}w_{\al})-\f{i\tilde{c}J^{-\f12}(i\tilde{c})_{\al}w_{\al}}{2}}{1+\mathbf{W}}.
\end{align*}
The term $C$ is
\begin{align*}
C=\f12\f{1}{1+\mathbf{W}}J^{-\f12}\p_{\al}\left(J^{-\f12}(ic)_{\al}\right)w_{\al}.
\end{align*}
For $D$, we have
\begin{align*}
D=\f32\f{A_0}{1+\mathbf{W}}(i\tilde{c})^2.
\end{align*}
In the following, we will simplify terms in $(1+\bw)p$ according to the number of derivatives that may act on $w$.

{\bf{Fourth-order term of $w$ in $(1+\mathbf{W})p$}:}
\begin{equation*}
\p_{\al}(J^{-\f12}\p_{\al}(J^{-\f12}\p_{\al}(J^{-\f12}w_{\al}))).   
\end{equation*}

{\bf{Third-order terms of $w$ in $(1+\mathbf{W})p$:}}
\begin{align*}
&-\p_{\al}(J^{-\f12}\p_{\al}(J^{-\f12}(i\tilde{c})w_{\al}))
-\f12\p_{\al}\left((i\tilde{c})J^{-\f12}\p_{\al}(J^{-\f12}w_{\al})\right)
-\f{i\tilde{c}\p_{\al}(J^{-\f12}\p_{\al}(J^{-\f12}w_{\al}))}{2}\\
=&-\p_{\al}(J^{-\f12}\p_{\al}(J^{-\f12}(i\tilde{c})w_{\al}))
-\p_{\al}\left((i\tilde{c})J^{-\f12}\p_{\al}(J^{-\f12}w_{\al})\right)
+\f{(i\tilde{c})_{\al}J^{-\f12}\p_{\al}(J^{-\f12}w_{\al})}{2}.
\end{align*}

{\bf{Second-order terms of $w$ in $(1+\mathbf{W})p$:}}
\begin{align*}
&\f12\p_{\al}\left((i\tilde{c})^2J^{-\f12}w_{\al}\right)+\f{i\tilde{c}\p_{\al}(J^{-\f12}(i\tilde{c})w_{\al})}{2}+\f{(i\tilde{c})^2\p_{\al}(J^{-\f12}w_{\al})}{4}\\
&+\f{3(i\tilde{c})^2\p_{\al}(J^{-\f12}w_{\al})}{2}+\f{(i\tilde{c})_{\al}J^{-\f12}\p_{\al}(J^{-\f12}w_{\al})}{2}-\f12\p_{\al}(J^{-\f12}J^{-\f12}(i\tilde{c})_{\al}w_{\al})\\
=&\f{11}4\p_{\al}\left((i\tilde{c})^2J^{-\f12}w_{\al}\right)-4(i\tilde{c})_{\al}(J^{-\f12}(i\tilde{c})w_{\al})
-\f12J^{-\f12}\p_{\al}(J^{-\f12}(i\tilde{c})_{\al})w_{\al}.
\end{align*}

{\bf{First-order terms of $w$ in $(1+\mathbf{W})p$:}}
\begin{align*}
&-4(i\tilde{c})_{\al}(J^{-\f12}(i\tilde{c})w_{\al})+\f14(i\tilde{c})_{\al}(J^{-\f12}(i\tilde{c})w_{\al}) -\f12J^{-\f12}\p_{\al}(J^{-\f12}(i\tilde{c})_{\al})w_{\al}
\\
&-\f12J^{-\f12}\p_{\al}(J^{-\f12}(i\tilde{c})_{\al})w_{\al}
-\f14(i\tilde{c})^3w_{\al}-\f32(i\tilde{c})^3w_{\al}\\
=&-icw_{\al}+\f{15}{4}\tilde{c}_{\al}\tilde{c}J^{-\f12}w_{\al}-\f{5i}{4}\tilde{c}^3w_{\al}.
\end{align*}
Note that 
\begin{equation*}
 \left\|\f{15}{4}\tilde{c}_{\al}\tilde{c}J^{-\f12}w_{\al}-\f{5i}{4}\tilde{c}^3w_{\al} \right\|_{L^2} \lesssim \ASSharp \| w\|_{H^\f32} ,
\end{equation*}
so that this term can be put into the perturbative source term $K$.
Collecting the contributions at each derivative order in $(1+\bw)p$, we obtain the expansion in \eqref{pexpansion}. 
\end{proof}

Motivated by the expansion of $(1+\bw)p$ in \eqref{pexpansion}, we introduce the operator
\begin{equation*}
\begin{split}
L:=&\p_{\al}(J^{-\f12}\p_{\al}(J^{-\f12}\p_{\al}(J^{-\f12}\p_{\al}))) -i\p_{\al}\left(J^{-\f12}\p_{\al}\left(\tilde{c}J^{-\f12}\p_{\al}\right)\right)\\
&
-i\p_{\al}\left(\tilde{c}J^{-\f12}\p_{\al}\left(J^{-\f12}\p_{\al}\right)\right)-\f{11}4\p_{\al}\left(\tilde{c}^2J^{-\f12}\p_{\al}\right)
-ic\p_{\al}-i{\nP}c_{\al}.
\end{split}
\end{equation*}
Then the system becomes 
\begin{equation*}
\begin{cases}
(\p_t+b\p_{\al})w+\f{1}{1+\bar{\mathbf{W}}}r_{\al}+\frac{R_{\al}}{1+\bar{\mathbf{W}}}w=\mathcal{G}_0(w,r),\\
(\p_t+b\p_{\al})r-i\frac{a}{1+\mathbf{W}}w-i\nP \left[\f{Lw}{1+\mathbf{W}}\right]=\mathcal{K}_0(w,r)+K.
\end{cases}
\end{equation*}

Applying the projection $\nP$ and using $\mathfrak{M}_b f := \nP[b f]$, the linearized hydroelastic waves can be rewritten as
\begin{equation*}
\begin{cases}
(\p_t+ \mathfrak{M}_{b}\p_{\al})w+\nP\left[\f{1}{1+\bar{\mathbf{W}}}r_{\al}\right]+\nP\left[\frac{R_{\al}}{1+\bar{\mathbf{W}}}w\right]=\nP\mathcal{G}_0,\\
(\p_t+\mathfrak{M}_b\p_{\al})r-i\nP\left[\frac{a}{1+\mathbf{W}}w\right]-i\nP\left[\f{Lw}{1+\mathbf{W}}\right]=\nP\mathcal{K}_0+K.
\end{cases}
\end{equation*}
In the following, we compute and simplify the term $-i\nP\left[\f{Lw}{1+\mathbf{W}}\right]$.
\begin{equation*}
\begin{split}
\nP[(1-Y)Lw]
=&\nP[(1-Y)J^{-\f32}\p_{\al}^4w]+\nP \left\{(1-Y)\left[2\p_{\al}J^{-\f32}-2i\tilde{c}J^{-1}\right]\p_{\al}^3w \right\}\\
+&\nP\left\{(1-Y)\left[2J^{-\f12}\p^2_{\al}J^{-1}+3J^{-\f12}(\p_{\al}J^{-\f12})^2-3\p_{\al}(i\tilde{c}J^{-1})-\f{11}{4}\tilde{c}^2J^{-\f12}\right]\p_{\al}^2w \right\}\\
+&\nP \left\{(1-Y)[J^{-1}\p_{\al}^3(J^{-\f12})-2iJ^{-1}\tilde{c}_{\al\al}]w_{\al} \right\}-i\nP\left[(1-Y)J^{-\f32}\p_{\al}^3\tilde{c}w\right]+K.
\end{split}
\end{equation*}
We now simplify each term in $\nP[(1-Y)Lw]$ successively. 
For the first term
\begin{equation*}
\nP[(1-Y)J^{-\f32}\p_{\al}^4w]=T_{(1-Y)J^{-\f32}}\p_{\al}^4w+T_{\p_{\al}^4w}\nP[(1-Y)J^{-\f32}-1]+\nP\Pi(\p^4_{\al}w,(1-Y)J^{-\f32}-1).
\end{equation*}
Using the paralinearization in Lemma \ref{t:Paralinear}, we can write
\begin{equation*}
 (1-Y)J^{-\f32}-1 = -\f52T_{(1-Y)^2J^{-\f32}} \bw -\f32 T_{J^{-\f52}}\bar{\bw} + err, \quad \|err\|_{C_*^\f72}  \lesssim \mathcal{A}^2_{\f74}.
\end{equation*}
Hence, we obtain
\begin{align*}
 \nP[(1-Y)J^{-\f32}\p_{\al}^4w]=&T_{(1-Y)J^{-\f32}}\p_{\al}^4w  -\f52T_{(1-Y)^2J^{-\f32}}T_{\p_{\al}^4w} \bw  \\
 -&\f52T_{(1-Y)^2J^{-\f32}}\Pi(\p_{\al}^4w, \bw)  -\f32  T_{J^{-\f52}}\nP \Pi(\p^4_{\al}w,\bar{\bw}) + K.
\end{align*}

For the second term in $\nP[(1-Y)Lw]$, 
\begin{equation*}
\begin{split}
&\nP\left\{(1-Y)\left[2\p_{\al}J^{-\f32}-2i\tilde{c}J^{-1}\right]\p_{\al}^3w\right\}\\
=&-\nP\left\{(1-Y)J^{-\f32}\left[5(1-Y)\mathbf{W}_{\al}+(1-\bar{Y})\bar{\mathbf{W}}_{\al}\right]\p_{\al}^3w \right\}\\
=&-5T_{(1-Y)^2J^{-\f32}\mathbf{W}_{\al}}\p_{\al}^3w-T_{J^{-\f52}\bar{\bw}_{\al}}\p_{\al}^3w-5T_{\p_{\al}^3w}\nP\left[(1-Y)^2J^{-\f32}\mathbf{W}_{\al}\right] \\
&-T_{\p_{\al}^3w}\nP \left[J^{-\f52}\bar{\bw}_{\al} \right]-5\nP\Pi(\p^3_{\al}w,(1-Y)^2J^{-\f32}\mathbf{W}_{\al})
-\nP\Pi(\p^3_{\al}w,J^{-\f52}\mathbf{W}_{\al})\, \\
=&-5T_{(1-Y)^2J^{-\f32}\mathbf{W}_{\al}}\p_{\al}^3w-T_{J^{-\f52}\bar{\bw}_{\al}}\p_{\al}^3w-5T_{(1-Y)^2J^{-\f32}}T_{\p_{\al}^3w}\mathbf{W}_{\al} \\
&-5T_{(1-Y)^2J^{-\f32}}\Pi(\p^3_{\al}w,\mathbf{W}_{\al})
-T_{J^{-\f52}}\nP\Pi(\p^3_{\al}w,\bar{\bw}_{\al}) + K\, .
\end{split}
\end{equation*}

For the third term  in $\nP[(1-Y)Lw]$, 
\begin{equation*}
\begin{split}
&\nP\left \{(1-Y)\left[2J^{-\f12}\p^2_{\al}J^{-1}+3J^{-\f12}(\p_{\al}J^{-\f12})^2-3\p_{\al}(i\tilde{c}J^{-1})-\f{11}{4}\tilde{c}^2J^{-\f12}\right]\p_{\al}^2w \right\}\\
=&\nP \left\{\f{(1-Y)}{J^{\f32}}\left[-\f{5\mathbf{W}_{\al\al}}{1+\mathbf{W}}+\f{\bar{\mathbf{W}}_{\al\al}}{1+\bar{\mathbf{W}}}+15\f{\mathbf{W}^2_{\al}}{(1+\mathbf{W})^2}\right]\p_{\al}^2w\right\}\\
=&-5T_{(1-Y)^2J^{-\f32}\mathbf{W}_{\al\al}}w_{\al\al}+T_{J^{-\f52}\bar{\mathbf{W}}_{\al\al}}w_{\al\al}+15T_{(1-Y)^3J^{-\f32}\mathbf{W}^2_{\al}}w_{\al\al}\\
&-5T_{w_{\al\al}} \nP\left[(1-Y)^2J^{-\f32}\mathbf{W}_{\al\al}\right]+T_{w_{\al\al}}\nP \left[J^{-\f52}\bar{\mathbf{W}}_{\al\al}\right]+15T_{w_{\al\al}} \nP\left[(1-Y)^3J^{-\f32}\mathbf{W}^2_{\al}\right]\\
&-5\nP\Pi(w_{\al\al},(1-Y)^2J^{-\f32}\mathbf{W}_{\al\al})+\nP\Pi(w_{\al\al},J^{-\f52}\bar{\mathbf{W}}_{\al\al})\\
&+15\nP\Pi(w_{\al\al},(1-Y)^3J^{-\f32}\mathbf{W}^2_{\al})\, \\
=& -5T_{(1-Y)^2J^{-\f32}\mathbf{W}_{\al\al}}w_{\al\al}+T_{J^{-\f52}\bar{\mathbf{W}}_{\al\al}}w_{\al\al}+15T_{(1-Y)^3J^{-\f32}\mathbf{W}^2_{\al}}w_{\al\al}\\
&-5T_{(1-Y)^2J^{-\f32}}T_{w_{\al\al}} \mathbf{W}_{\al\al} -5T_{(1-Y)^2 J^{-\f32}}\Pi(w_{\al\al},\mathbf{W}_{\al\al})+T_{J^{-\f52}}\nP\Pi(w_{\al\al},\bar{\mathbf{W}}_{\al\al}) + K.
\end{split}
\end{equation*}

Next, for the fourth term  in $\nP[(1-Y)Lw]$,
\begin{equation*}
\begin{split}
&\nP \left\{(1-Y)[J^{-1}\p_{\al}^3(J^{-\f12})-2iJ^{-1}\tilde{c}_{\al\al}]w_\alpha \right\}\\
=&\nP \left\{(1-Y)J^{-\f32}\left[-\f52\f{\p^3_{\al}\mathbf{W}}{1+\mathbf{W}}+\f32\f{\p^3_{\al}\bar{\mathbf{W}}}{1+\bar{\mathbf{W}}}\right]w_{\al}\, \right\}+K\\
=&-\f52T_{(1-Y)^2J^{-\f32}\p^3_{\al}\mathbf{W}}w_{\al}+\f32T_{J^{-\f52}\p^3_{\al}\bar{\mathbf{W}}}w_{\al} -\f52T_{w_{\al}} \nP \left[(1-Y)^2J^{-\f32}\p^3_{\al}\mathbf{W}\right]\\
&+\f32T_{w_{\al}}\nP \left[J^{-\f52}\p^3_{\al}\bar{\mathbf{W}}\right]-\f52\nP\Pi(w_{\al},(1-Y)^2J^{-\f32}\p^3_{\al}\mathbf{W})+\f32\nP\Pi(w_{\al},J^{-\f52}\p^3_{\al}\bar{\mathbf{W}})+K\, \\
=&-\f52T_{(1-Y)^2J^{-\f32}\p^3_{\al}\mathbf{W}}w_{\al}+\f32T_{J^{-\f52}\p^3_{\al}\bar{\mathbf{W}}}w_{\al} -\f52 T_{(1-Y)^2J^{-\f32}}T_{w_{\al}} \p^3_{\al}\mathbf{W}\\
&-\f52 T_{(1-Y)^2J^{-\f32}}\Pi(w_{\al},\p^3_{\al}\mathbf{W})+\f32T_{J^{-\f52}}\nP\Pi(w_{\al},\p^3_{\al}\bar{\mathbf{W}}) + K.
\end{split}
\end{equation*}
Finally, for the last term  in $\nP[(1-Y)Lw]$,
\begin{equation*}
\begin{split}
&-\nP \left[(1-Y)J^{-\f32}\p_\al^3\nP(i\tilde{c})w \right]\\
=&-T_{(1-Y)^2J^{-\f32}\p^4_{\al}\mathbf{W}}w-T_{w}\nP[(1-Y)^2J^{-\f32}\p^4_{\al}\mathbf{W}]\\
&-\nP\Pi \left(w,(1-Y)^2J^{-\f32}\p^4_{\al}\mathbf{W}\right)+\nP\Pi(w,J^{-\f52}\p^4_{\al}\bar{\mathbf{W}})+K\,\\
=&-T_{(1-Y)^2J^{-\f32}\p^4_{\al}\mathbf{W}}w -T_{(1-Y)^2J^{-\f32}}T_{w}\p^4_{\al}\mathbf{W} -T_{(1-Y)^2J^{-\f32}}\Pi \left(w,\p^4_{\al}\mathbf{W}\right)+K.
\end{split}
\end{equation*}

Gathering the computations for all terms for $-i\nP\left[\f{Lw}{1+\mathbf{W}}\right]$, one can rewrite the linearized system in the
paradifferential form as follows:
\begin{equation} \label{ParadifferentialLinearEqn}
\begin{cases}
T_{D_t}w+T_{1-\bar{Y}}r_{\al}+T_{(1-\bar{Y})R_{\al}}w=\mathcal{G}^{\sharp}(w,r)+G,\\
T_{D_t}r-i\mathcal{L}_{para}w=\mathcal{K}^{\sharp}(w,r)+K,
\end{cases}
\end{equation}
where the operator $\mathcal{L}_{para}$, which can be viewed as the paradifferential version of the operator $L$, is given by
\begin{equation} \label{LLinDef}
\begin{split}
\mathcal{L}_{para}w=&T_{(1-Y)J^{-\f32}}\p_{\al}^4w-5T_{(1-Y)^2J^{-\f32}\mathbf{W}_{\al}}\p_{\al}^3w-T_{J^{-\f52}\bar{\bw}_{\al}}\p_{\al}^3w -5T_{(1-Y)^2J^{-\f32}\mathbf{W}_{\al\al}}w_{\al\al}
\\
&+T_{J^{-\f52}\bar{\mathbf{W}}_{\al\al}}w_{\al\al}+15T_{(1-Y)^3J^{-\f32}\mathbf{W}^2_{\al}}w_{\al\al} -\f52T_{(1-Y)^2J^{-\f32}\p^3_{\al}\mathbf{W}}w_{\al}\\
&+\f32T_{J^{-\f52}\p^3_{\al}\bar{\mathbf{W}}}w_{\al}-T_{(1-Y)^2J^{-\f32}\p^4_{\al}\mathbf{W}}w,   
\end{split}
\end{equation}
and source terms $\mathcal{G}^{\sharp}=\nP(\mathcal{G}_0+\mathcal{G}_1)$, $\mathcal{K}^{\sharp}=\nP(\mathcal{K}_0+\mathcal{K}_1)$ with
\begin{equation*}
\begin{aligned}
\mathcal{G}_1=&\Pi(r_{\al},\bar{Y})-(T_{w_{\al}}b+\Pi(w_{\al},b))-T_{w}((1-\bar{Y})R_{\al})-\Pi(w,(1-\bar{Y})R_{\al}),\\
\mathcal{K}_1
=&-T_{r_{\al}}T_{1-\bar{Y}}R-\Pi(r_{\al},T_{1-\bar{Y}}R)-\Pi(r_{\al},T_{1-Y}\bar{R})\\
&-\f52iT_{(1-Y)^2J^{-\f32}}T_{\p_{\al}^4w} \bw-\f52iT_{(1-Y)^2J^{-\f32}}\Pi(\p_{\al}^4w, \bw)  -\f32i  T_{J^{-\f52}} \Pi(\p^4_{\al}w,\bar{\bw})\\
&-5iT_{(1-Y)^2J^{-\f32}}T_{\p_{\al}^3w}\mathbf{W}_{\al} 
-5iT_{(1-Y)^2J^{-\f32}}\Pi(\p^3_{\al}w,\mathbf{W}_{\al})
-iT_{J^{-\f52}}\Pi(\p^3_{\al}w,\bar{\bw}_{\al}) \\
&-5iT_{(1-Y)^2J^{-\f32}}T_{w_{\al\al}} \mathbf{W}_{\al\al}-5iT_{(1-Y)^2J^{-\f32}}\Pi(w_{\al\al},\mathbf{W}_{\al\al})+iT_{J^{-\f52}}\Pi(w_{\al\al},\bar{\mathbf{W}}_{\al\al})\\
&-\f52iT_{(1-Y)^2J^{-\f32}}T_{w_{\al}}\p^3_{\al}\mathbf{W}-\f52i T_{(1-Y)^2J^{-\f32}}\Pi(w_{\al},\p^3_{\al}\mathbf{W}) +\f32iT_{J^{-\f52}}\Pi(w_{\al},\p^3_{\al}\bar{\mathbf{W}})\\
&-iT_{(1-Y)^2J^{-\f32}}T_{w}\p^4_{\al}\mathbf{W}-iT_{(1-Y)^2J^{-\f32}}\Pi(w,\p^4_{\al}\mathbf{W}).
\end{aligned}
\end{equation*}

\subsection{Quadratic bounds for paradifferential source terms}
In this section, we compute the leading part of source terms $(\nP\mathcal{G}_0, \nP\mathcal{K}_0)$, $(\nP\mathcal{G}_1, \nP\mathcal{K}_1)$, and use the results to obtain the leading terms of $(T_{D_t}w, T_{D_t }r)$.

We first simplify source terms $(\nP\mathcal{G}_0, \nP\mathcal{K}_0)$, which are given by the following results.
\begin{lemma} \label{t:PGKZeroExp}
The source terms $(\nP\mathcal{G}_0, \nP\mathcal{K}_0)$ can be rewritten as
\begin{align*}
 \nP\mathcal{G}_0 =&  -T_{J^{-1}}T_{\bar{r}_\alpha}\bw +  T_{(1-\bar{Y})^2(1+\bw)}T_{\bar{w}_\alpha} R - T_{J^{-1}} \nP \Pi(\bar{r}_\alpha, \bw) +  T_{(1-\bar{Y})^2(1+\bw)} \nP \Pi(\bar{w}_\alpha, R) + G ,\\
  \nP\mathcal{K}_0 =& -T_{1-\bar{Y}}T_{\bar{r}_\alpha} R  +\f32 i T_{J^{-\f52}} T_{\p^4_{\al}\bar{w}}\bw + iT_{J^{-\f52}}T_{\p^3_{\al}\bar{w}}\bw_\al -iT_{J^{-\f52}} T_{\bar{w}_{\al\al}}\bw_{\al\al}\\
  &- \f32iT_{J^{-\f52}}T_{\bar{w}_{\al}} \p^3_{\al}\bw- T_{1-\bar{Y}}\nP\Pi(\bar{r}_\alpha, R)+\f32 iT_{J^{-\f52}} \nP \Pi(\p^4_{\al}\bar{w},\bw) \\
  &+ iT_{J^{-\f52}} \nP \Pi(\p^3_{\al}\bar{w}, \bw_\al)- iT_{J^{-\f52}}\nP \Pi(\bar{w}_{\al\al},\bw_{\al\al})- \f32 i T_{J^{-\f52}}\nP \Pi(\bar{w}_{\al},\p^3_{\al}\bw) + K,
 \end{align*}
where $(G,K)$ are perturbative source terms that satisfy 
\begin{equation} \label{HalfGK}
 \| (G,K)\|_{\H^0}\lesssim_{\CalAZ}  \ASSharp \| (w,r)\|_{\H^0}. 
\end{equation} 
\end{lemma}

\begin{proof}
We first consider the estimate for the source term $\nP \mathcal{G}_0$.
The auxiliary linearized variables $m$ can be written as
\begin{equation*}
m = |1-Y|^2(r_\alpha +R_\alpha w)+(1-Y)^2\bar{R}w_\alpha,
\end{equation*}
so that by putting perturbative terms into $G$, we obtain
\begin{align*}
&\nP \bar{m} = -\nP[(1-\bar{Y})(\bar{r}_\alpha + \bar{R}_\alpha \bar{w})Y] + \nP[(1-\bar{Y})^2\bar{w}_\alpha R] \\
=& -T_{(1-\bar{Y})(\bar{r}_\alpha + \bar{R}_\alpha \bar{w})}Y -\nP \Pi((1-\bar{Y})(\bar{r}_\alpha + \bar{R}_\alpha \bar{w}), Y) + T_{(1-\bar{Y})^2\bar{w}_\alpha}R + \nP\Pi((1-\bar{Y})^2\bar{w}_\alpha, R)\\
=& - T_{1-\bar{Y}}T_{\bar{r}_\alpha}Y - T_{(1-\bar{Y})^2}T_{\bar{w}_\alpha}R - T_{1-\bar{Y}}\nP\Pi(\bar{w}_\alpha, R)+ T_{(1-\bar{Y})^2}\nP\Pi(\bar{w}_\alpha, R) + G.
\end{align*}
By using the paralinearization result of $Y$ \eqref{YWExpression},
\begin{align*}
    &\nP \mathcal{G}_0 = T_{1+\bw} \nP \bar{m} + G \\
    =& -T_{J^{-1}}T_{\bar{r}_\alpha}\bw +  T_{(1-\bar{Y})^2(1+\bw)}T_{\bar{w}_\alpha} R - T_{J^{-1}} \nP \Pi(\bar{r}_\alpha, \bw) +  T_{(1-\bar{Y})^2(1+\bw)} \nP \Pi(\bar{w}_\alpha, R) + G.
\end{align*}

Next, we consider the source term  $\nP \mathcal{K}_0$.
The auxiliary linearized function is $n = (1-Y)\bar{R}(r_\alpha + R_\alpha w)$, which gives
\begin{align*}
    \nP \bar{n} &= \nP[(1-\bar{Y})\bar{r}_\alpha R] + \nP[(1-\bar{Y})\bar{R}_\alpha \bar{w}R] = \nP[T_{1-\bar{Y}}\bar{r}_\alpha R] + K \\
    &=  T_{1-\bar{Y}}T_{\bar{r}_\alpha} R + T_{1-\bar{Y}}\nP\Pi(\bar{r}_\alpha, R) + K.
\end{align*}
It remains  to simplify the $\nP \bar{p}$ term. 
According to the computation in \eqref{pexpansion},
\begin{align*}
\bar{p}=&(1-\bar{Y})\p_{\al}(J^{-\f12}\p_{\al}(J^{-\f12}\p_{\al}(J^{-\f12}\bar{w}_{\al})))\\
&+i\left[\p_{\al}\left(J^{-\f12}\p_{\al}\left(\tilde{c}J^{-\f12}\bar{w}_{\al}\right)\right)
+\p_{\al}\left(\tilde{c}J^{-\f12}\p_{\al}\left(J^{-\f12}\bar{w}_{\al}\right)\right)\right]\\
&-\f{11}4\p_{\al}\left(\tilde{c}^2J^{-\f12}\bar{w}_{\al}\right)
+ic\bar{w}_{\al} + K\\
=&(1-\bar{Y})J^{-\f32}\p_{\al}^4\bar{w}
-(1-\bar{Y})J^{-\f32}\left[5(1-\bar{Y})\bar{\mathbf{W}}_{\al}+(1-Y)\mathbf{W}_{\al}\right]\p_{\al}^3\bar{w}\\
&+\f{(1-\bar{Y})}{J^{\f32}}\left[-\f{5\bar{\mathbf{W}}_{\al\al}}{1+\bar{\mathbf{W}}}+\f{\mathbf{W}_{\al\al}}{1+\mathbf{W}}+15\f{\bar{\mathbf{W}}^2_{\al}}{(1+\bar{\mathbf{W}})^2}\right]\p_{\al}^2\bar{w}\\
&+(1-\bar{Y})J^{-\f32}\left[-\f52\f{\p^3_{\al}\bar{\mathbf{W}}}{1+\bar{\mathbf{W}}}+\f32\f{\p^3_{\al}\mathbf{W}}{1+\mathbf{W}}\right]\bar{w}_{\al}+K\\
=:&\,\bar{p}_1+\bar{p}_2+\bar{p}_3+\bar{p}_4+K.
\end{align*}
For $\nP\bar{p_1}$, we have $\nP\bar{p}_1= \nP[((1-\bar{Y})J^{-\f32}-1)\p_{\al}^4\bar{w}].$
Using the paralinearization in Lemma \ref{t:Paralinear}, we can write
\begin{equation*}
 (1-\bar{Y})J^{-\f32}-1 = -\f52T_{(1-\bar{Y})^2J^{-\f32}} \bar{\bw} -\f32 T_{J^{-\f52}}\bw + err, \quad \|err\|_{C_*^\f72}  \lesssim \mathcal{A}^2_{\f74}.
\end{equation*}
Hence, we obtain
\begin{align*}
\nP\bar{p}_1=-\f32 T_{J^{-\f52}} \nP(\bw \p^4_{\al}\bar{w})+K = -\f32 T_{J^{-\f52}} T_{\p^4_{\al}\bar{w}}\bw -\f32 T_{J^{-\f52}} \nP \Pi(\p^4_{\al}\bar{w},\bw) +K .
\end{align*}
For $\nP\bar{p}_2$, we have
\begin{equation*}
\begin{split}
\nP\bar{p}_2=&-\nP\left\{(1-\bar{Y})J^{-\f32}\left[5(1-\bar{Y})\bar{\mathbf{W}}_{\al}+(1-Y)\mathbf{W}_{\al}\right]\p_{\al}^3\bar{w} \right\}\\
=&-5\nP[J^{-\f32}(1-\bar{Y})^2\bar{\bw}_{\al}\p^3_{\al}\bar{w}] - T_{J^{-\f52}}\nP(\bw_{\al}\p^3_{\al}\bar{w}).
\end{split}
\end{equation*}
Applying the commutator estimate, the first term of $\nP\bar{p}_2$ is perturbative:
\begin{align*}
&\|\nP[J^{-\f32}(1-\bar{Y})^2\bar{\bw}_{\al}\p^3_{\al}\bar{w}]\|_{L^2}=\|[\nP,J^{-\f32}(1-\bar{Y})^2\bar{\bw}_{\al}]\p^3_{\al}\bar{w}\|_{L^2}\\
\lesssim&\|\nP J^{-\f32}(1-\bar{Y})^2\bar{\bw}_{\al}\|_{C^{\f32}_*}\|w\|_{H^{\f32}}\\
\leq&\left(\|T_{(1-\bar{Y})^2\bar{\bw}_{\al}}(J^{-\f32}-1)\|_{C^{\f32}_*}+\|\nP\Pi(T_{(1-\bar{Y})^2\bar{\bw}_{\al}},J^{-\f32}-1)\|_{C^{\f32}_*}\right)\|w\|_{H^{\f32}}\\
\lesssim&_{\CalAZ}  \ASSharp \| w\|_{H^\f32}.
\end{align*}
Therefore, we obtain
\begin{equation*}
\begin{split}
\nP\bar{p}_2
=&- T_{J^{-\f52}} \nP(\bw_{\al}\p^3_{\al}\bar{w})+K = - T_{J^{-\f52}}T_{\p^3_{\al}\bar{w}}\bw_\al - T_{J^{-\f52}} \nP \Pi(\p^3_{\al}\bar{w}, \bw_\al) + K.
\end{split}
\end{equation*}
Similarly, for the rest two terms, by absorbing perturbative terms into $K$,
\begin{equation*}
\begin{split}
\nP\bar{p}_3=&\nP \left\{\f{(1-\bar{Y})}{J^{\f32}}\left[-\f{5\bar{\mathbf{W}}_{\al\al}}{1+\bar{\mathbf{W}}}+\f{\mathbf{W}_{\al\al}}{1+\mathbf{W}}+15\f{\bar{\mathbf{W}}^2_{\al}}{(1+\bar{\mathbf{W}})^2}\right]\p_{\al}^2\bar{w}\right\}\\
=& T_{J^{-\f52}}\nP(\bw_{\al\al}\bar{w}_{\al\al})+K = T_{J^{-\f52}} T_{\bar{w}_{\al\al}}\bw_{\al\al} + T_{J^{-\f52}}\nP \Pi(\bar{w}_{\al\al},\bw_{\al\al}) + K,\\
\nP\bar{p}_4=&\nP \left\{(1-\bar{Y})J^{-\f32}\left[-\f52\f{\p^3_{\al}\bar{\mathbf{W}}}{1+\bar{\mathbf{W}}}+\f32\f{\p^3_{\al}\mathbf{W}}{1+\mathbf{W}}\right]\bar{w}_{\al}\, \right\}\\
=&\f32T_{J^{-\f52}} \nP (\p^3_{\al}\bw \bar{w}_{\al})+K = \f32T_{J^{-\f52}}T_{\bar{w}_{\al}} \p^3_{\al}\bw + \f32T_{J^{-\f52}}\nP \Pi(\bar{w}_{\al},\p^3_{\al}\bw) +K.
\end{split}
\end{equation*}
Collecting the above, we finally obtain
\begin{align*}
i\nP\bar{p}=& -\f32 i T_{J^{-\f52}} T_{\p^4_{\al}\bar{w}}\bw - iT_{J^{-\f52}}T_{\p^3_{\al}\bar{w}}\bw_\al +iT_{J^{-\f52}} T_{\bar{w}_{\al\al}}\bw_{\al\al}+ \f32iT_{J^{-\f52}}T_{\bar{w}_{\al}} \p^3_{\al}\bw\\
&-\f32 iT_{J^{-\f52}} \nP \Pi(\p^4_{\al}\bar{w},\bw) - iT_{J^{-\f52}} \nP \Pi(\p^3_{\al}\bar{w}, \bw_\al)+ iT_{J^{-\f52}}\nP \Pi(\bar{w}_{\al\al},\bw_{\al\al})\\
&+ \f32 i T_{J^{-\f52}}\nP \Pi(\bar{w}_{\al},\p^3_{\al}\bw) +K.
\end{align*}
Putting each term in $\nP \mathcal{K}_0$ yields the expression for its leading part.
\end{proof}

Then we simplify source terms $(\nP\mathcal{G}_1, \nP\mathcal{K}_1)$.
By putting the perturbative terms into $(G,K)$ and commuting para-coefficients, we obtain the following lemma.
\begin{lemma} \label{t:PGKOneExp}
The source terms $(\nP\mathcal{G}_1, \nP\mathcal{K}_1)$ can be rewritten as
\begin{align*}
 \nP\mathcal{G}_1 &= -T_{1-\bar{Y}}\p_\al T_w R - T_{1-\bar{Y}}\p_\al\Pi(w_\al, R)- T_{1-Y}\nP\Pi(w_\al, \bar{R}) + T_{(1-\bar{Y})^2}\nP \Pi(r_\al, \bar{\bw}) +G,\\
  \nP\mathcal{K}_1 &= -T_{1-\bar{Y}}T_{r_{\al}}R -iT_{(1-Y)^2J^{-\f32}}T_{w}\p^4_{\al}\mathbf{W} -\f52iT_{(1-Y)^2J^{-\f32}}T_{w_{\al}}\p^3_{\al}\mathbf{W} -5iT_{(1-Y)^2J^{-\f32}}T_{w_{\al\al}} \mathbf{W}_{\al\al} \\
  -&5iT_{(1-Y)^2J^{-\f32}}T_{\p_{\al}^3w}\mathbf{W}_{\al} -\f52iT_{(1-Y)^2J^{-\f32}}T_{\p_{\al}^4w} \bw -T_{1-\bar{Y}}\Pi(r_{\al},R) -\f52iT_{(1-Y)^2J^{-\f32}}\Pi(\p_{\al}^4w, \bw) \\
  -&5iT_{(1-Y)^2J^{-\f32}}\Pi(\p^3_{\al}w,\mathbf{W}_{\al}) -5iT_{(1-Y)^2J^{-\f32}}\Pi(w_{\al\al},\mathbf{W}_{\al\al}) -\f52i T_{(1-Y)^2J^{-\f32}}\Pi(w_{\al},\p^3_{\al}\mathbf{W})\\
  -&iT_{(1-Y)^2J^{-\f32}}\Pi(w,\p^4_{\al}\mathbf{W}) -T_{1-Y}\nP\Pi(r_{\al},\bar{R})-\f32i  T_{J^{-\f52}} \nP\Pi(\p^4_{\al}w,\bar{\bw})-iT_{J^{-\f52}}\nP\Pi(\p^3_{\al}w,\bar{\bw}_{\al})\\
  +&iT_{J^{-\f52}}\nP\Pi(w_{\al\al},\bar{\mathbf{W}}_{\al\al})+\f32iT_{J^{-\f52}}\nP\Pi(w_{\al},\p^3_{\al}\bar{\mathbf{W}}) + K,
 \end{align*}
where $(G,K)$ are perturbative source terms that satisfy \eqref{HalfGK}.
\end{lemma}

Finally, after obtaining the source terms $(\nP\mathcal{G}_0, \nP\mathcal{K}_0)$ and $(\nP\mathcal{G}_1, \nP\mathcal{K}_1)$, we compute the leading terms of para-material derivatives of $(w,r)$.
\begin{lemma}
Suppose  $(w,r)$ solve the paradifferential linearized equations \eqref{ParadifferentialLinearEqn}, then we have the expressions for $(T_{D_t}w, T_{D_t}r)$. 
\begin{equation} \label{wrParaMaterial}
\begin{aligned}
T_{D_t}w  =& -T_{1-\bar{Y}}(r_\alpha + T_w R_\alpha + T_{w_\al}R)+ \tilde{G} +G,\\
T_{D_t}r  =& iT_{J^{-\frac{3}{2}}(1-Y)}\big(\p_\al^4w -5T_{1-Y}T_{\mathbf{W}_{\al}}\p_{\al}^3w-T_{1-\bar{Y}}T_{\bar{\bw}_\al}\p_{\al}^3w \\
&-T_{1-Y}T_w \p_\al^4\bw -\f52 T_{1-Y}T_{w_\al} \p_\al^3\bw -\f32 T_{1-\bar{Y}}T_{\bar{w}_\al}\p_\al^3 \bw\big)+\tilde{K}+K, 
\end{aligned}
\end{equation}
where $(G,K)$ are perturbative source terms that satisfy \eqref{HalfGK}, and $(\tilde{G}, \tilde{K})$ satisfy the estimate
\begin{equation*}
    \| (\tilde{G}, \tilde{K})\|_{\mathcal{H}^0} \lesssim_{\CalAZ} \mathcal{A}_{\sharp, \f74}  \| (w,r)\|_{\mathcal{H}^{\f34}}.
\end{equation*}
\end{lemma}

\begin{proof}
A direct computation for each part of the paradifferential equation \eqref{ParadifferentialLinearEqn} yields
\begin{align*}
 &T_{(1-\bar{Y})R_{\al}}w = T_{T_{1-\bar{Y}}R_\al}w + G = \tilde{G} + G, \\
 &\nP \mathcal{G}_0 = \tilde{G} + G, \quad \nP \mathcal{G}_1 = -T_{1-\bar{Y}}\p_\al T_w R + \tilde{G} + G, \\
 & i\mathcal{L}_{para}w = iT_{(1-Y)J^{-\f32}}\p_{\al}^4w-5iT_{(1-Y)^2J^{-\f32}\mathbf{W}_{\al}}\p_{\al}^3w-iT_{J^{-\f52}\bar{\bw}_{\al}}\p_{\al}^3w  + \tilde{K} + K,\\
 &\nP \mathcal{K}_0 = - \f32iT_{J^{-\f52}}T_{\bar{w}_{\al}} \p^3_{\al}\bw+\tilde{K} + K, \\
 & \nP \mathcal{K}_1 = -iT_{(1-Y)^2J^{-\f32}}T_{w}\p^4_{\al}\mathbf{W} -\f52iT_{(1-Y)^2J^{-\f32}}T_{w_{\al}}\p^3_{\al}\mathbf{W}+\tilde{K} + K.
\end{align*}
By adding each non-perturbative part of the paradifferential equation \eqref{ParadifferentialLinearEqn} together and rewriting para-coefficients, we obtain the leading terms of $(T_{D_t}w, T_{D_t}r)$.
\end{proof}

\section{Energy estimates for the linearized hydroelastic waves} \label{s:EstLin}
This section is dedicated to constructing the modified energy and establishing the modified energy estimate for the linearized hydroelastic system \eqref{linearizedeqn}. 
Specifically, we prove the following theorem.
\begin{theorem} \label{t:LinearizedWellposed}
Assume that $\mathcal{A}_\f32 \lesssim 1$ and $\mathcal{A}_{\sharp, \f74} \in L^2_t([0,T])$.
Then the linearized hydroelastic system \eqref{linearizedeqn} is well-posed in $\H^0$ on $[0,T]$.
Furthermore, there exists an energy functional $E^{0}_{lin}(w,r)$ satisfying the following properties on $[0,T]$:
\begin{enumerate}
    \item Norm equivalence:
    \begin{equation*}
        E^{0}_{lin}(w,r) = \left(1+O\left(\mathcal{A}_\f32 \right) \right) \| (w,r)\|^2_{\H^{0}}.
    \end{equation*}
    \item The time derivative of $E_{lin}^{0}(w,r)$ is bounded by:
    \begin{equation*}
        \frac{d}{dt} E_{lin}^{0}(w,r) \lesssim_{\CalAZ}  \ASSharp E^{0}_{lin}(w,r).
    \end{equation*}
\end{enumerate} 
\end{theorem}
The proof of Theorem \ref{t:LinearizedWellposed} is divided into two main steps.
First, we consider the linear part of paradifferential linearized hydroelastic waves \eqref{ParadifferentialLinearEqn} with perturbative source terms, namely, the system: 
\begin{equation}  \label{ParadifferentialFlow}
\begin{cases}
T_{D_t}w+T_{1-\bar{Y}}r_{\al}+T_{T_{1-\bar{Y}}R_{\al}}w=G,\\
T_{D_t}r-i\mathcal{L}_{para}w=K,
\end{cases}
\end{equation}
where the operator $\mathcal{L}_{para}$ is defined in \eqref{LLinDef}, and $(G,K)$ are perturbative source terms that satisfy \eqref{HalfGK}.
We establish the following modified energy estimate for this system.
\begin{proposition} \label{t:wellposedflow}
Assume that $\mathcal{A}_0
\lesssim 1$ and $\mathcal{A}_{\sharp,\f74}  \in L^2_t([0,T])$ for some time $T>0$, then if $(w,r)$ solve the homogeneous paradifferential system \eqref{ParadifferentialFlow} on $[0,T]$, there exists an  energy functional $E^{0,para}_{lin}(w, r)$ such that on $[0,T]$, we have the following two properties:
\begin{enumerate}
\item Norm equivalence:
\begin{equation*}
    E^{0,para}_{lin}(w, r) = (1+O(\CalAZ)) \|(w,r)\|_{\H^0}^2.
\end{equation*}
\item The time derivative of $E^{0,para}_{lin}(w, r)$ is bounded by
\begin{equation*}
    \frac{d}{dt}  E^{0,para}_{lin}(w, r) \lesssim_{\CalAZ} \ASSharp \|(w,r)\|_{\H^0}^2.
\end{equation*}
\end{enumerate}  
\end{proposition}
To prove Proposition \ref{t:wellposedflow}, we first construct a quadratic paradifferential linearized energy $E_{lin}(w,r)$ in \eqref{ElinDef}.
This energy is designed to cancel the leading terms of $(T_{D_t}w, T_{D_t}r)$ in the integral.
The time derivative $\f{d}{dt}E_{lin}$ also contains non-perturbative cubic energy, which we eliminate by constructing a cubic energy correction $E^3_{cor}$.
The remaining non-perturbative terms are subsequently removed by  quartic energy corrections, so that $ E^{0,para}_{lin}$ defined in \eqref{E0ParaLinDef} is exactly the energy that we need in Proposition \ref{t:wellposedflow}.

In the second step, we address the non-perturbative source terms $(\mathcal{G}^\sharp, \mathcal{K}^\sharp)$.
Since $(\mathcal{G}^\sharp, \mathcal{K}^\sharp)$ does not satisfy \eqref{HalfGK}, Theorem \ref{t:LinearizedWellposed} cannot follow  directly from  Proposition \ref{t:wellposedflow}.
We will construct linearized normal form variables 
\begin{equation*}
    (w_{NF}, r_{NF}) :=(w_{NF}^0, r_{NF}^0) + (w_{NF}^1, r_{NF}^1)
\end{equation*}
 to eliminate $(\mathcal{G}^\sharp, \mathcal{K}^\sharp)$.
For each part of the normal form variables, we will perform the normal form analysis for the low-high quadratic portion, balanced quadratic portion and cubic portion of $(\nP\mathcal{G}_0, \nP\mathcal{K}_0)$ or $(\nP\mathcal{G}_1, \nP\mathcal{K}_1)$.
The normal form variables satisfy the bound
\begin{equation} \label{NFVariableBound}
    \| (w_{NF}, r_{NF})\|_{\H^0} \lesssim_\CalAZ \mathcal{A}_\f32 \| (w,r)\|_{\H^0},
\end{equation}
and the pair $(w+w_{NF}, r+r_{NF})$ solves the paradifferential linearized flow \eqref{ParadifferentialFlow} with perturbative source terms.
Choosing the modified energy
\begin{equation*}
 E^{0}_{lin}(w,r) := E^{0,para}_{lin} (w+w_{NF}, r+r_{NF}),
\end{equation*}
and applying Proposition \ref{t:wellposedflow} and \eqref{NFVariableBound},
the modified energy $E^{0}_{lin}(w,r)$ satisfies both the norm equivalence and cubic energy estimate in Theorem \ref{t:LinearizedWellposed}.
The well-posedness of the linearized hydroelastic waves \eqref{linearizedeqn} follows from a standard fixed-point argument using the modified energy estimate.
This concludes the proof of Theorem \ref{t:LinearizedWellposed}.

The remainder of this section focuses on proving the modified energy estimate for \eqref{ParadifferentialFlow} and constructing normal form variables $(w_{NF}, r_{NF})$.
In Section \ref{s:HomoFlow}, we prove Proposition \ref{t:wellposedflow}.
Then, in Section \ref{s:NormalZero}, we compute $(w_{NF}^0, r_{NF}^0)$ to eliminate $(\nP\mathcal{G}_0, \nP\mathcal{K}_0)$.
Finally, in Section \ref{s:NormalOne}, we compute $(w_{NF}^1, r_{NF}^1)$ to eliminate $(\nP\mathcal{G}_1, \nP\mathcal{K}_1)$.

\subsection{$\H^0$ energy estimate of the homogeneous paradifferential flow} \label{s:HomoFlow}
In this subsection, we construct the energy $E^{0,para}_{lin}(w,r)$ in Proposition \ref{t:wellposedflow}, i.e. the $\H^0$ modified energy estimate of the linear paradifferential flow \eqref{ParadifferentialFlow}.
We begin with a paradifferential linearized quadratic energy $E_{lin}(w,r) = (1+O(\CalAZ)) \|(w,r)\|_{\H^0}^2$, and then construct cubic and quartic energy correction $E_{cor}(w,r)$ such that its time derivative eliminates the non-perturbative part of $\frac{d}{dt}E_{lin}(w,r)$.

We define the paradifferential linearized energy $E_{lin}(w,r)$ as
\begin{equation} \label{ElinDef}
    E_{lin}(w,r) = \int \Im \left(T_{J^{-\f54}}w_{\al}\cdot \bar{w}_{\al \al} \right) + \Re \left(r\cdot T_{J^{\f14}}\bar{r}\right) + \Re(w\cdot \bar{w})\,d\alpha.
\end{equation}
The para-coefficients $T_{J^{-\f54}}$, $T_{J^{\f14}}$ are selected for two reasons:
\begin{enumerate}
\item Cancellation of Leading Terms: The identity $J^{-\f54}(1-\bar{Y}) = J^\f14\cdot J^{-\f32}(1-\bar{Y})$ ensures that integral terms of the type $\Re \int ir \cdot \partial_\al^4 \bar{w}\,d\al$ in $\f{d}{dt}E_{lin}(w,r)$ vanish due to the cancellation of para-coefficients.
\item Matching Coefficients for sub-leading terms: The exponents $-\f54$ and $\f14$ are chosen so that in $\f{d}{dt}E_{lin}(w,r)$, integrals of the type $\Re \int ir \cdot T_{\bw_\al} \p_\al^3 \bar{w} \,d\al$ and $\Re \int i\bar{r} \cdot T_{\bw_\al}\p_\al^3 w \,d\al$ have the same coefficient. This leads to a cancellation when computing the symbols of the paradifferential cubic forms for the energy correction $E^3_{cor}(w,r)$, leaving only lower-order integrals.
\end{enumerate}

This choice of energy satisfies the norm equivalence
\begin{equation*}
  E_{lin}(w,r) =  (1+O(\CalAZ)) \|(w,r)\|_{\H^0}^2. 
\end{equation*}
We proceed to compute its time derivative:
\begin{align*}
\frac{d}{dt} E_{lin} =  2\Im \int \bar{w}_{\al \al} \cdot T_{J^{-\f54}}w_{\al t} \,d\alpha  +\Im \int \bar{w}_{\al \al} \cdot T_{\partial_tJ^{-\f54}}w_{\al} \,d\alpha + 2\Re \int r\cdot T_{J^{\f14}}\bar{r}_t \, d\alpha  + \Re \int T_{\partial_t J^{\f14}}r\cdot \bar{r}\,d\alpha.
\end{align*} 
By taking the $\alpha$-derivative on the first equation of \eqref{ParadifferentialFlow} and using the definition of $b_\al$, we obtain:
\begin{equation} \label{wAlphaParaMat}
T_{D_t}w_\al = -T_{1-\bar{Y}}r_{\al\al} + T_{(1-\bar{Y})^2\bar{\bw}_\al}r_\al - 2T_{T_{1-\bar{Y}}R_{\al}}w_\al - T_{T_{1-Y}\bar{R}_\al}w_\al - T_{T_{1-\bar{Y}}R_{\al\al}}w + G_1,
\end{equation}
where the error term $G_1$ satisfies
\begin{equation*}
\|G_1\|_{H^\f12} \lesssim_{\CalAZ} \mathcal{A}_{\sharp, \f74}  \| (w,r)\|_{\mathcal{H}^{\f34}}.
\end{equation*}

To utilize the expressions of $(T_{D_t}w_\al, T_{D_t}r)$, we integrate by parts:
\begin{equation*}
2\Re \int r\cdot T_{J^{\f14}}T_b\bar{r}_\alpha \, d\alpha = -2\Re \int T_{\big(J^{\f14}b\big)_\alpha}r\cdot \bar{r}\, d\alpha - 2\Re \int T_{J^{\f14}}T_{b}r_\alpha \cdot \bar{r}\, d\alpha  + O\left(\ASSharp\right) \|r\|_{L^2}^2,
\end{equation*}
so that using the fact that $J^s$ and $b$ are real-valued functions,
\begin{equation*}
2\Im \int \bar{w}_{\al \al}\cdot T_{J^{-\f54}}T_b w_{\al \al} \,d\alpha  =0,\quad 2\Re \int r\cdot T_{J^{\f14}}T_b\bar{r}_\alpha \, d\alpha = -\Re \int T_{\big(J^{\f14}b\big)_\alpha}r\cdot \bar{r}\, d\alpha  + O\left(\ASSharp\right) \|r\|_{L^2}^2.
\end{equation*}

Using above identities, we plug in para-material derivatives in \eqref{ParadifferentialFlow} and \eqref{wAlphaParaMat} to compute
\begin{align*}
&\frac{d}{dt} E_{lin}(w,r)\\
=& 2\Im \int \bar{w}_{\al \al} \cdot T_{J^{-\f54}}T_{D_t}w_\al \,d\alpha   
+  \Im \int \bar{w}_{\al \al} \cdot T_{\partial_tJ^{-\f54}}w_{\al} \,d\alpha + 2\Re \int r\cdot T_{J^{\f14}}T_{D_t}\bar{r} \, d\alpha \\
&+ \Re \int T_{\partial_t J^{\f14}}r\cdot \bar{r}\,d\alpha 
 + \Re \int T_{\big(J^{\f14}b\big)_\alpha}r\cdot \bar{r}\, d\alpha + \Re \int w\cdot \bar{w}_t \,d\al  + O\left(\ASSharp\right) \|(w,r)\|_{\H^0}^2 \\
=&2\Re \int iT_{J^{-\f54}(1-\bar{Y})}r_{\al \al}\cdot \bar{w}_{\al \al} \,d\alpha -2\Re \int iT_{J^{-\f54}(1-\bar{Y})^2 \bar{\bw}_\al}r_\al \cdot \bar{w}_{\al \al} \,d\al + 4\Re \int i T_{T_{J^{-\f54}(1-\bar{Y})}R_{\al}}w_\al \cdot \bar{w}_{\al \al} \,d\al \\
&+ 2\Re \int i T_{T_{J^{-\f54}(1-Y)}\bar{R}_{\al}}w_\al \cdot \bar{w}_{\al \al} \,d\al + 2\Re \int i T_{T_{J^{-\f54}(1-\bar{Y})}R_{\al\al}}w \cdot \bar{w}_{\al \al} \,d\al - 2\Re \int i r\cdot T_{J^{-\f54}(1-\bar{Y})} \p_\al^4 \bar{w} \,d\al \\
&+10\Re \int ir\cdot T_{J^{-\f54}(1-\bar{Y})^2 \bar{\bw}_\al} \p_\al^3 \bar{w} \,d\al +2\Re \int ir \cdot T_{J^{-\f94}\bw_\al}  \p_\al^3 \bar{w} \,d\al + 10\Re \int ir\cdot T_{J^{-\f54}(1-\bar{Y})^2\bar{\bw}_{\al \al}}\bar{w}_{\al \al}\,d\al \\
&- 2\Re \int ir \cdot T_{J^{-\f94}\bw_{\al \al}} \bar{w}_{\al \al} \,d\al + 5 \Re \int ir \cdot T_{J^{-\f54}(1-\bar{Y})^2 \p_\al^3 \bar{\bw}}\bar{w}_\al \,d\al - 3\Re \int ir \cdot T_{J^{-\f94}\p_\al^3 \bw} \bar{w}_{\al} \,d\al \\
&+2 \Re \int ir \cdot T_{J^{-\f54}(1-\bar{Y})^2 \p_\al^4 \bar{\bw}}\bar{w} \,d\al -15\Re \int ir \cdot T_{J^{-\f54}(1-\bar{Y})^3 \bar{\bw}_\al^2} \bar{w}_{\al \al} \,d\al + \Re \int T_{\partial_t J^{\f14} + J^{\f14} b_\al}r\cdot \bar{r}\,d\alpha\\
&-\Re \int iT_{\p_t J^{-\f54}}w_{\al}\cdot \bar{w}_{\al \al} \,d\alpha-2\Re \int i T_{J^{-\f54}}G_1 \cdot \bar{w}_{\al \al} \,d\al
+2 \Re \int r \cdot T_{J^{\f14}} \bar{K} \,d\al +  O\left(\ASSharp \right) \|(w,r)\|_{\H^0}^2.
\end{align*}
For the first two terms in $\frac{d}{dt} E_{lin}(w,r)$, we integrate by parts to shift $\alpha$-derivatives from $r$ to $\bar{w}$ and the para-coefficients,
\begin{align*}
 &2\Re \int iT_{J^{-\f54}(1-\bar{Y})}r_{\al \al}\cdot \bar{w}_{\al \al} \,d\alpha -2\Re \int iT_{J^{-\f54}(1-\bar{Y})^2 \bar{\bw}_\al}r_\al \cdot \bar{w}_{\al \al} \,d\al \\
 =& \f52\Re  \int iT_{J^{-\f94}\bw_{\al}}r_{\al }\cdot \bar{w}_{\al \al } \,d\al +\f52\Re \int iT_{J^{-\f54}(1-\bar{Y})^2 \bar{\bw}_\al}r_\al \cdot \bar{w}_{\al \al} \,d\al  -2\Re \int iT_{J^{-\f54}(1-\bar{Y})}r_{\al }\cdot \p_\al^3\bar{w} \,d\alpha \\
 =&- \f52\Re  \int iT_{J^{-\f94}\bw_{\al}}r\cdot \p_\al^3\bar{w} \,d\al -\f52\Re \int iT_{J^{-\f54}(1-\bar{Y})^2 \bar{\bw}_\al}r \cdot \p_\al^3\bar{w} \,d\al + 2\Re \int iT_{J^{-\f54}(1-\bar{Y})}r\cdot \p_\al^4\bar{w} \,d\alpha \\
 &-\f52\Re  \int iT_{J^{-\f94}\bw_{\al \al}}r\cdot \bar{w}_{\al \al } \,d\al - \f52\Re \int iT_{J^{-\f54}(1-\bar{Y})^2 \bar{\bw}_{\al\al}}r \cdot \bar{w}_{\al \al} \,d\al -\f52\Re  \int iT_{J^{-\f94}\bw_{\al}}r\cdot \p_\al^3\bar{w} \,d\al\\
 &-\f92\Re \int iT_{J^{-\f54}(1-\bar{Y})^2 \bar{\bw}_\al}r \cdot \p_\al^3\bar{w} \,d\al + \frac{45}{8}\Re \int i T_{J^{-\f94}(1-Y)\bw_\al^2}r \cdot \bar{w}_{\al \al} \,d\al + \frac{45}{8}\Re \int i T_{J^{-\f94}(1-\bar{Y})|\bw_\al|^2}r \cdot \bar{w}_{\al \al} \,d\al\\
& +  \f{25}8\Re \int i T_{J^{-\f94}(1-\bar{Y})|\bw_\al|^2}r \cdot \bar{w}_{\al \al} \,d\al + \frac{65}{8} \Re \int iT_{J^{-\f54}(1-\bar{Y})^3 \bar{\bw}_\al^2}r \cdot \bar{w}_{\al \al} \,d\al \\
  =&  2\Re \int iT_{J^{-\f54}(1-\bar{Y})}r\cdot \p_\al^4\bar{w} \,d\alpha - 5\Re  \int iT_{J^{-\f94}\bw_{\al}}r\cdot \p_\al^3\bar{w} \,d\al -7\Re \int iT_{J^{-\f54}(1-\bar{Y})^2 \bar{\bw}_\al}r \cdot \p_\al^3\bar{w} \,d\al \\
  &- \f52\Re  \int iT_{J^{-\f94}\bw_{\al \al}}r\cdot \bar{w}_{\al \al } \,d\al - \f52\Re \int iT_{J^{-\f54}(1-\bar{Y})^2 \bar{\bw}_{\al\al}}r \cdot \bar{w}_{\al \al} \,d\al
  + \frac{45}{8}\Re \int i T_{J^{-\f94}(1-Y)\bw_\al^2}r \cdot \bar{w}_{\al \al} \,d\al\\
  &+ \frac{65}{8}\Re \int iT_{J^{-\f54}(1-\bar{Y})^3 \bar{\bw}_\al^2}r \cdot \bar{w}_{\al \al} \,d\al 
  + \frac{35}{4}\Re \int i T_{J^{-\f94}(1-\bar{Y})|\bw_\al|^2}r \cdot \bar{w}_{\al \al} \,d\al.
\end{align*}
For the $w\cdot \bar{w}$ type integrals in  $\frac{d}{dt} E_{lin}(w,r)$, we again integrate by parts and use Lemma \ref{t:JsParaMat},
\begin{align*}
&4\Re \int i T_{T_{J^{-\f54}(1-\bar{Y})}R_{\al}}w_\al \cdot \bar{w}_{\al \al} \,d\al 
+2\Re \int i T_{T_{J^{-\f54}(1-\bar{Y})}\bar{R}_{\al}}w_\al \cdot \bar{w}_{\al \al} \,d\al \\
&+ 2\Re \int i T_{T_{J^{-\f54}(1-\bar{Y})}R_{\al\al}}w \cdot \bar{w}_{\al \al} \,d\al-\Re \int iT_{\partial_tJ^{-\f54}}w_{\al}\cdot \bar{w}_{\al \al} \,d\alpha \\
=&\frac{11}{4}\Re \int i T_{T_{J^{-\f54}(1-\bar{Y})}R_{\al}}w_\al \cdot \bar{w}_{\al \al} \,d\al 
+\f34\Re \int i T_{T_{J^{-\f54}(1-\bar{Y})}\bar{R}_{\al}}w_\al \cdot \bar{w}_{\al \al} \,d\al+ 2\Re \int i T_{T_{J^{-\f54}(1-\bar{Y})}R_{\al\al}}w \cdot \bar{w}_{\al \al} \,d\al\\
=& \frac{11}{4}\Re \int i T_{T_{J^{-\f54}(1-\bar{Y})}R_{\al}}w_\al \cdot \bar{w}_{\al \al} \,d\al  + 2\Re \int i T_{T_{J^{-\f54}(1-\bar{Y})}R_{\al\al}}w \cdot \bar{w}_{\al \al} \,d\al \\
&- \f34\Re \int i T_{T_{J^{-\f54}(1-Y)}\bar{R}_{\al}}w_{\al\al} \cdot \bar{w}_{\al} \,d\al - \f34\Re \int i T_{T_{J^{-\f54}(1-Y)}\bar{R}_{\al \al}}w_\al \cdot \bar{w}_{\al } \,d\al+ O\left(\ASSharp \right) \| w\|^2_{H^\f32} \\
=& \frac{7}{2}\Re \int i T_{T_{J^{-\f54}(1-\bar{Y})}R_{\al}}w_\al \cdot \bar{w}_{\al \al} \,d\al + \f54\Re \int i T_{T_{J^{-\f54}(1-\bar{Y})}R_{\al\al}}w \cdot \bar{w}_{\al \al} \,d\al \\
&- \f34\Re \int i T_{T_{J^{-\f54}(1-\bar{Y})}\partial_\al^3 R}w \cdot \bar{w}_{\al } \,d\al+ O\left(\ASSharp \right) \| w\|^2_{H^\f32}.
\end{align*}

Combining these results with Lemma \ref{t:JsParaMat}, we compute:
\begin{align*}
&\frac{d}{dt} E_{lin}(w,r)=\frac{7}{2}\Re \int ir\cdot T_{J^{-\f54}(1-\bar{Y})^2 \bar{\bw}_\al} \p_\al^3 \bar{w} \,d\al-\Re \int ir \cdot T_{J^{-\f94}\bw_\al}  \p_\al^3 \bar{w} \,d\al \\
+& \frac{15}{2}\Re \int ir\cdot T_{J^{-\f54}(1-\bar{Y})^2\bar{\bw}_{\al \al}}\bar{w}_{\al \al}\,d\al-\f92\Re \int ir \cdot T_{J^{-\f94}\bw_{\al \al}} \bar{w}_{\al \al} \,d\al +5 \Re \int ir \cdot T_{J^{-\f54}(1-\bar{Y})^2 \p_\al^3 \bar{\bw}}\bar{w}_\al \,d\al\\
-& 3\Re \int ir \cdot T_{J^{-\f94}\p_\al^3 \bw} \bar{w}_{\al} \,d\al +2 \Re \int ir \cdot T_{J^{-\f54}(1-\bar{Y})^2 \p_\al^4 \bar{\bw}}\bar{w} \,d\al  + 3\Re \int i T_{T_{J^{-\f54}(1-\bar{Y})}R_{\al}}w_\al \cdot \bar{w}_{\al \al} \,d\al \\
+& \f54\Re \int i T_{T_{J^{-\f54}(1-\bar{Y})}R_{\al\al}}w \cdot \bar{w}_{\al \al} \,d\al-\f34\Re \int i T_{T_{J^{-\f54}(1-\bar{Y})}\partial_\al^3 R}w \cdot \bar{w}_{\al } \,d\al  + \f32\Re \int T_{T_{(1-\bar{Y})J^{\f14}}R_\al} r\cdot \bar{r} \,d\al\\
-& \f{55}{8}\Re \int ir \cdot T_{J^{-\f54}(1-\bar{Y})^3 \bar{\bw}_\al^2} \bar{w}_{\al \al} \,d\al  +\f{45}{8}\Re \int ir \cdot T_{J^{-\f94}(1-Y)\bw_\al^2} \bar{w}_{\al \al} \,d\al\\
 +& \frac{35}{4}\Re \int i T_{J^{-\f94}(1-\bar{Y})|\bw_\al|^2}r \cdot \bar{w}_{\al \al}+  O\left(\ASSharp \right) \|(w,r)\|_{\H^0}^2\\
 =&\frac{7}{2}\Re \int i T_{T_{J^{-\f54}(1-\bar{Y})}R_{\al}}w_\al \cdot \bar{w}_{\al \al} \,d\al +\f54 \Re \int i T_{T_{J^{-\f54}(1-\bar{Y})}R_{\al\al}}w \cdot \bar{w}_{\al \al} \,d\al-\f34\Re \int i T_{T_{J^{-\f54}(1-\bar{Y})}\partial_\al^3 R}w \cdot \bar{w}_{\al } \,d\al\\
 -&3\Re \int ir \cdot T_{J^{-\f94}\bw_\al}  \p_\al^3 \bar{w} \,d\al-\f92\Re \int ir \cdot T_{J^{-\f94}\bw_{\al \al}} \bar{w}_{\al \al} \,d\al- 3\Re \int ir \cdot T_{J^{-\f94}\p_\al^3 \bw} \bar{w}_{\al} \,d\al\\
-&3\Re \int i\bar{r}\cdot T_{J^{-\f54}(1-Y)^2 \bw_\al} \p_\al^3 w \,d\al -\frac{15}{2}\Re \int i\bar{r}\cdot T_{J^{-\f54}(1-Y)^2\bw_{\al \al}}w_{\al \al}\,d\al- 5 \Re \int i\bar{r} \cdot T_{J^{-\f54}(1-Y)^2 \p_\al^3 \bw}w_\al \,d\al \\
-&2 \Re \int i\bar{r} \cdot T_{J^{-\f54}(1-Y)^2 \p_\al^4\bw}w\,d\al +\f32\Re \int T_{T_{(1-\bar{Y})J^{\f14}}R_\al} r\cdot \bar{r} \,d\al - \f{55}{8}\Re \int ir \cdot T_{J^{-\f54}(1-\bar{Y})^3 \bar{\bw}_\al^2} \bar{w}_{\al \al} \,d\al  \\
 +&\f{45}{8}\Re \int ir \cdot T_{J^{-\f94}(1-Y)\bw_\al^2} \bar{w}_{\al \al} \,d\al
 + \frac{35}{4}\Re \int i T_{J^{-\f94}(1-\bar{Y})|\bw_\al|^2}r \cdot \bar{w}_{\al \al}\,d\al+  O\left(\ASSharp \right) \|(w,r)\|_{\H^0}^2.
\end{align*}
The expression for $\frac{d}{dt}E_{lin}(w,r)$ contains both non-perturbative cubic and quartic integral terms, which we eliminate by constructing cubic and quartic energy corrections.
For the cubic energy corrections, we seek a cubic energy $E_{cor}^3(w,r) = O(\CalAZ)\|(w,r)\|_{\H^0}^2$ such that its time derivative eliminates the cubic integral terms of $\frac{d}{dt}E_{lin}(w,r)$.
In other words, the time derivative of $E_{cor}^3(w,r)$ needs to satisfy 
\begin{align*}
&\frac{d}{dt} E_{cor}^3(w,r)
=  -\frac{7}{2}\Re \int i T_{T_{J^{-\f54}(1-\bar{Y})}R_{\al}}w_\al \cdot \bar{w}_{\al \al} \,d\al -\f54 \Re \int i T_{T_{J^{-\f54}(1-\bar{Y})}R_{\al\al}}w \cdot \bar{w}_{\al \al} \,d\al\\
+&\f34\Re \int i T_{T_{J^{-\f54}(1-\bar{Y})}\partial_\al^3 R}w \cdot \bar{w}_{\al } \,d\al+3\Re \int ir \cdot T_{J^{-\f94}\bw_\al}  \p_\al^3 \bar{w} \,d\al+\f92\Re \int ir \cdot T_{J^{-\f94}\bw_{\al \al}} \bar{w}_{\al \al} \,d\al\\
+& 3\Re \int ir \cdot T_{J^{-\f94}\p_\al^3 \bw} \bar{w}_{\al} \,d\al+3\Re \int i\bar{r}\cdot T_{J^{-\f54}(1-Y)^2 \bw_\al} \p_\al^3 w \,d\al +\frac{15}{2}\Re \int i\bar{r}\cdot T_{J^{-\f54}(1-Y)^2\bw_{\al \al}}w_{\al \al}\,d\al\\
+& 5 \Re \int i\bar{r} \cdot T_{J^{-\f54}(1-Y)^2 \p_\al^3 \bw}w_\al \,d\al +2 \Re \int i\bar{r} \cdot T_{J^{-\f54}(1-Y)^2 \p_\al^4\bw}w\,d\al -\f32\Re \int T_{T_{(1-\bar{Y})J^{\f14}}R_\al} r\cdot \bar{r} \,d\al\\
+&  O\left(\ASSharp \right) \|(w,r)\|_{\H^0}^2.
\end{align*}

To achieve this, we employ a cubic energy correction of the form 
\begin{align*}
    E^3_{cor}(w,r) = \Re &\int A_{lhh}\left(R, T_{J^{\f14}}w, \bar{r}\right) + B_{lhh}\left(R, T_{J^{\f14}(1-\bar{Y})^2}r, \bar{w}\right) \\
    &+ C_{lhh}\left( \bw, T_{J^{-\f54}(1-Y)}w, \bar{w}\right) + D_{lhh}\left( \bw, T_{J^{\f14}(1-Y)}r, \bar{r}\right) \, d\alpha,
\end{align*}
where $A_{lhh},B_{lhh},C_{lhh},D_{lhh}$ are paradifferential cubic forms where the first variable is at low frequency compared to the other two variables. 
The para-coefficients such as $T_{J^{\f14}}$ are added to the cubic forms in order to match the para-coefficients in cubic integral terms of $\frac{d}{dt}E_{lin}(w,r)$.
We compute the time derivative of $E^3_{cor}(w,r)$.
It is given by
\begin{align*}
&\frac{d}{dt}  E^3_{cor}(w,r) = \text{quartic and higher integral terms} +  O\left(\ASSharp \right) \|(w,r)\|_{\H^0}^2 \\
+ \Re\int& -A\left( T_{J^{-\f54}(1-\bar{Y})}R, w, i\p_\al^4\bar{w}\right) + B\left(T_{J^{-\f54}(1-\bar{Y})}R, i\p_\al^4 w, \bar{w}\right)-C\left(T_{J^{-\f54}(1-\bar{Y})}R_\alpha, w, \bar{w}\right)\,d\al  \\
+\Re \int&  B\left(i \p_\al^4\bw, T_{J^{-\f94}}r, \bar{w}\right) - C\left(\bw, T_{J^{-\f94}}r_{\alpha}, \bar{w}\right)- D\left(\bw, T_{J^{-\f94}}r, i\p_\al^4\bar{w}\right)\,d\al  \\
+\Re \int&  A\left(i \p_\al^4\bw, T_{J^{-\f54}(1-Y)^2}w, \bar{r}\right) - C\left(\bw, T_{J^{-\f54}(1-Y)^2}w, \bar{r}_\alpha\right)+ D\left(\bw, iT_{J^{-\f54}(1-Y)^2}\p_\al^4 w, \bar{r}\right)\,d\al \\
-\Re \int&A\left(T_{J^{\f14}(1-\bar{Y})}R, r_\alpha, \bar{r}\right) + B\left(T_{J^{\f14}(1-\bar{Y})}R, r, \bar{r}_\alpha\right)+ D\left(T_{J^{\f14}(1-\bar{Y})}R_\alpha, r, \bar{r}\right) \,d\alpha.
\end{align*}
Let $\mathfrak{a}(\xi, \eta, \zeta)$ denote the symbol of $A\left(R, T_{J^{\f14}}w, \bar{r}\right)$.
Other three symbols are defined in the same way.
We take the Fourier transform of the integral and compare symbols for each cubic term.
Using integration by parts, a derivative on the third factor is equal to the minus of the sum of  derivatives of the first two factors  in the integral, which shows the symbolic relation $\zeta = \xi + \eta$.
We get that symbols  $\mathfrak{a}, \mathfrak{b}, \mathfrak{c}, \mathfrak{d}$ solve the following algebraic linear system.
\begin{equation*}
\begin{cases}
\zeta^4\mathfrak{a}-\eta^4\mathfrak{b}+\xi\mathfrak{c}=\left(\frac{7}{2}\xi\eta\zeta^2+\f54\xi^2\zeta^2+\f34\xi^3\zeta\right)\chi_{1}(\xi,\eta),\\
\xi^4\mathfrak{b}-\eta\mathfrak{c}-\zeta^4\mathfrak{d}=(-3\xi\zeta^3+\f92\xi^2\zeta^2-3\xi^3\zeta)\chi_{1}(\xi,\eta),\\
\xi^4\mathfrak{a}+\zeta\mathfrak{c}+\eta^4\mathfrak{d}=(3\xi\eta^3+\frac{15}{2}\xi^2\eta^2+5\xi^3\eta+2\xi^4)\chi_{1}(\xi,\eta),\\
\eta \mathfrak{a}-\zeta\mathfrak{b}+\xi\mathfrak{d}=\f32\xi\chi_{1}(\xi,\eta),
\end{cases}
\end{equation*}
where the $\chi_1(\theta_1, \theta_2)$ is a non-negative smooth bump function defined in \eqref{ChiOnelh} that selects low-high frequency portion of a product. 
This algebraic system has the solution
\begin{align*}
\mathfrak{a}=&\f{5(15\xi^6+60\xi^5\eta+125\xi^4\eta^2+157\xi^3\eta^3+120\xi^2\eta^4+51\xi\eta^5+8\eta^6)}{4\eta(25\xi^6+100\xi^5\eta+200\xi^4\eta^2+246\xi^3\eta^3+200\xi^2\eta^4+100\xi\eta^5+25\eta^6)}\chi_{1}(\xi,\eta),\\
\mathfrak{b}=&-\f{5(14\xi^6+49\xi^5\eta+77\xi^4\eta^2+60\xi^3\eta^3+13\xi^2\eta^4-13\xi\eta^5-8\eta^6)}
{4\eta(25\xi^6+100\xi^5\eta+200\xi^4\eta^2+246\xi^3\eta^3+200\xi^2\eta^4+100\xi\eta^5+25\eta^6)}\chi_{1}(\xi,\eta),\\
\mathfrak{c}=&\f{125\xi^9+875\xi^8\eta+2850\xi^7\eta^2+5633\xi^6\eta^3+7407\xi^5\eta^4}
{4(25\xi^6+100\xi^5\eta+200\xi^4\eta^2+246\xi^3\eta^3+200\xi^2\eta^4+100\xi\eta^5+25\eta^6)}\chi_{1}(\xi,\eta)\\
&+\f{6708\xi^4\eta^5+4134\xi^3\eta^6+1600\xi^2\eta^7+300\xi\eta^8}
{4(25\xi^6+100\xi^5\eta+200\xi^4\eta^2+246\xi^3\eta^3+200\xi^2\eta^4+100\xi\eta^5+25\eta^6)}\chi_{1}(\xi,\eta),\\
\mathfrak{d}=&\f{40\xi^6+105\xi^5\eta+135\xi^4\eta^2+83\xi^3\eta^3+25\xi^2\eta^5}
{2(25\xi^6+100\xi^5\eta+200\xi^4\eta^2+246\xi^3\eta^3+200\xi^2\eta^4+100\xi\eta^5+25\eta^6)}\chi_{1}(\xi,\eta).
\end{align*}

By the choice of the frequency cutoff $\chi_1$, $|\xi|\ll |\eta| \approx |\zeta|$.
According to our discussion in Appendix \ref{s:Discuss}, the denominators of above symbols are elliptic such that they cannot be zero.
The leading part of the numerators and denominators are powers of $\eta$.
Hence, at the leading order
\begin{equation*}
\mathfrak{a} \approx \f25\chi_1(\xi, \eta), \quad \mathfrak{b} \approx  \f25\chi_1(\xi, \eta), \quad \mathfrak{c} \approx 3\xi\eta^2 \chi_1(\xi, \eta), \quad \mathfrak{d} \approx \f{1}{2}\xi^2\eta^{-2}\chi_1(\xi, \eta),
\end{equation*}
so that we have
\begin{align*}
 E^3_{cor}(w,r) &= \f25\Re \int  T_R T_{J^\f14}w\cdot\bar{r} + T_R T_{J^{\frac{1}{4}}(1-\bar{Y})^2}r\cdot \bar{w} \,d\al -3\Re \int iT_{\bw_\al}T_{J^{-\f54}(1-Y)}w_\al \cdot \bar{w}_{\al }\,d\al\\
 & + \f12 \Re \int T_{\bw_{\al\al}} T_{J^\f14(1-Y)}\p_\al^{-1}r\cdot \p_\al^{-1}\bar{r} \,d\al + \text{lower order integrals}.
\end{align*}
As a consequence, 
\begin{equation*}
    |E^3_{cor}(w,r)| \lesssim \CalAZ \|(w,r)\|_{\H^0}.
\end{equation*}

According to our computation of the paradifferential quadratic normal forms, $\frac{d}{dt}E^3_{cor}(w,r)$ eliminates the cubic energy produced by $\frac{d}{dt}E_{lin}(w,r)$.
When the time derivative acts on para-coefficients such as $T_{J^\f14}$, it produces perturbative terms.
For instance, 
\begin{equation*}
    \Re \int  T_R T_{\p_t J^\f14}w\cdot\bar{r} \,d\al = -\f14 \Re \int  T_R T_{ J^\f14 b_\al}w\cdot\bar{r} \,d\al + \text{lower order integrals} =  O\left(\ASSharp \right) \|(w,r)\|_{\H^0}^2.
\end{equation*}
The quintic integral terms produced by $\frac{d}{dt}E^3_{cor}(w,r)$ are all perturbative.
However, $\frac{d}{dt}E^3_{cor}(w,r)$ also produces extra non-perturbative quartic integral terms.
We get that
\begin{equation*} 
 \frac{d}{dt}E^3_{cor}(w,r) +\frac{d}{dt}E_{lin}(w,r)= \text{non-perturbative quartic integrals} +  O\left(\ASSharp \right) \|(w,r)\|_{\H^0}^2,
\end{equation*}
where the non-perturbative quartic integral terms can be classified as one of following two types:
\begin{enumerate}
    \item The non-perturbative quartic integral terms in $\frac{d}{dt}E_{lin}(w,r)$, namely 
    \begin{align*}
    &- \f{55}{8}\Re \int ir \cdot T_{J^{-\f54}(1-\bar{Y})^3 \bar{\bw}_\al^2} \bar{w}_{\al \al} \,d\al   +\f{45}{8}\Re \int ir \cdot T_{J^{-\f94}(1-Y)\bw_\al^2} \bar{w}_{\al \al} \,d\al \\
 &+ \frac{35}{4}\Re \int i r \cdot T_{J^{-\f94}(1-\bar{Y})|\bw_\al|^2}\bar{w}_{\al \al}\,d\al.
    \end{align*}
    Note that for first two quartic integral terms of this type, the frequencies of two variables at low frequency have the same sign. 
    \item The non-perturbative quartic integral terms in $ \frac{d}{dt}E^3_{cor}(w,r)$.
    They are produced by sub-leading terms of $(w_t, r_t)$.
    These quartic integrals are given by 
    \begin{equation*}
     3\Re \int ir \cdot T_{J^{-\f94}(1-\bar{Y})|\bw_\al|^2}\bar{w}_{\al\al}\,d\al -3\Re \int ir \cdot T_{J^{-\f94}(1-Y)\bw_\al^2}\bar{w}_{\al \al }\,d\al.
    \end{equation*}
   \end{enumerate}
 These non-perturbative quartic integrals $I_{non}^4$ satisfy 
$|I_{non}^4| = O(\mathcal{A}_1^2) \| (w,r)\|_{\H^\f14}^2$.
The linearized variables have $\f12$ more derivative compared to $\| (w,r)\|_{\H^0}^2$.

To eliminate these non-perturbative quartic integrals, we consider two scenarios, the non-resonant quartic integrals, and the quartic integrals that may have four-wave resonances.
\begin{enumerate}
\item For non-perturbative quartic integrals 
\begin{equation*}
 - \f{55}{8}\Re \int ir \cdot T_{J^{-\f54}(1-\bar{Y})^3 \bar{\bw}_\al^2} \bar{w}_{\al \al} \,d\al   +\f{21}{8}\Re \int ir \cdot T_{J^{-\f94}(1-Y)\bw_\al^2} \bar{w}_{\al \al} \,d\al,
\end{equation*}
two variables at low frequency are either $\bw_\al^2$ or $\bar{\bw}_\al^2$, so that their frequencies have the same sign.
According to the discussion in Appendix \ref{s:Discuss}, four-wave resonances cannot happen.
Therefore, one can construct a quartic integral correction $E^{4,1}_{cor}(w,r) = O(\mathcal{A}_0^2)\|(w,r) \|^2_{\H^0}$ such that 
\begin{align*}
    &\frac{d}{dt} E^{4,1}_{cor}(w,r) = \f{55}{8}\Re \int ir \cdot T_{J^{-\f54}(1-\bar{Y})^3 \bar{\bw}_\al^2} \bar{w}_{\al \al} \,d\al   \\
    -&\f{21}{8}\Re \int ir \cdot T_{J^{-\f94}(1-Y)\bw_\al^2} \bar{w}_{\al \al} \,d\al+\text{quintic and higher integrals}.
\end{align*}
These quintic and higher integrals have one or more lower order compare to these non-perturbative quartic integrals, so that these terms are perturbative.

\item Non-perturbative quartic integrals 
 \begin{equation*}
        \f{47}{4}\Re \int i r \cdot T_{J^{-\f94}(1-\bar{Y})|\bw_\al|^2}\bar{w}_{\al \al}\,d\al
    \end{equation*}
may have resonances.
We cannot construct quartic integral corrections to eliminate these terms without producing extra quartic integrals.
However, we can choose the correction 
\begin{equation*}
    E^{4,2}_{cor}(w,r) = -\frac{47}{8}\Re \int \p_\al^{-1}r \cdot T_{J^{-\f34}|\bw_\al|^2}\p_\al^{-1} \bar{r} \,d\al. 
\end{equation*}
This energy satisfies the norm equivalence $E^{4,2}_{cor}(w,r) = O(\mathcal{A}_0^2)\|(w,r) \|^2_{\H^0}$, and its time derivative equals
\begin{align*}
    E^{4,2}_{cor}(w,r) =& -\f{47}{4}\Re \int i r \cdot T_{J^{-\f94}(1-\bar{Y})|\bw_\al|^2}\bar{w}_{\al \al}\,d\al  \\
    &-\frac{47}{8}\Re \int \p_\al^{-1}r \cdot T_{\partial_t(J^{-\f34}|\bw_\al|^2)}\p_\al^{-1} \bar{r} \,d\al + \text{ quintic and higher integrals} \\
    =& -\f{47}{4}\Re \int i r \cdot T_{J^{-\f94}(1-\bar{Y})|\bw_\al|^2}\bar{w}_{\al \al}\,d\al  +  O\left(\ASSharp \right) \|(w,r)\|_{\H^0}^2.
\end{align*}
\end{enumerate}
To conclude this section, by choosing the modified energy
\begin{equation} \label{E0ParaLinDef}
E^{0,para}_{lin}(w, r) := E_{lin}(w,r) + E^{3}_{cor}(w,r) + E^{4,1}_{cor}(w,r) + E^{4,2}_{cor}(w,r),
\end{equation}
it satisfies the norm equivalence, and its time derivative is perturbative.
Hence, we prove Proposition \ref{t:wellposedflow}.

\subsection{Normal form analysis for $(\nP\mathcal{G}_0, \nP\mathcal{K}_0)$} \label{s:NormalZero}
In this section, we compute normal form variables $(w_{NF}^0, r_{NF}^0)$ such that
\begin{equation} \label{wrNFHZeroBound}
    \|(w_{NF}^0, r_{NF}^0)\|_{\H^0} \lesssim_\CalAZ \mathcal{A}_{\f32} \|(w,r)\|_{\H^0},
\end{equation}
and they solve the system 
\begin{equation}  \label{wZeroNFPGKZero}
\begin{cases}
T_{D_t}w_{NF}^0+T_{1-\bar{Y}}\p_\al r_{NF}^0+T_{T_{1-\bar{Y}}R_{\al}}w_{NF}^0= -\nP \mathcal{G}_0 + G,\\
T_{D_t}r_{NF}^0-i\mathcal{L}_{para}w_{NF}^0=- \nP \mathcal{K}_0 +K,
\end{cases}
\end{equation}
where $(G,K)$ are perturbative source terms that satisfy \eqref{HalfGK}.

Recall that in Lemma \ref{t:PGKZeroExp}, source terms $(\nP\mathcal{G}_0, \nP\mathcal{K}_0)$ are rewritten as the sum of low-high and balanced paraproducts in Lemma \ref{t:PGKZeroExp}.
We will construct $(w_{NF}^0, r_{NF}^0)$ as the sum of low-high quadratic normal form variables,  balanced quadratic normal form variables, and cubic normal form variables
\begin{equation} \label{wrNFZeroDef}
 (w_{NF}^0, r_{NF}^0) := (w^0_{lh}, r^0_{lh}) + (w_{bal}^0, r_{bal}^0) + (w^0_c, r^0_{c}).   
\end{equation}
Normal form variables $(w^0_{lh}, r^0_{lh})$ and $(w_{bal}^0, r_{bal}^0)$ eliminate the low-high and balanced part of $(\nP\mathcal{G}_0, \nP\mathcal{K}_0)$,  and $(w^0_c, r^0_{c})$ eliminates the extra non-perturbative cubic part of source terms produced by quadratic normal form variables.

\subsubsection{Low-high quadratic normal form analysis for $(\nP\mathcal{G}_0, \nP\mathcal{K}_0)$} 
We begin by computing low-high quadratic normal form variables $(w^0_{lh}, r^0_{lh})$ to eliminate the low-high portion of the $(\nP\mathcal{G}_0, \nP\mathcal{K}_0)$.
That is, we seek $(w^0_{lh}, r^0_{lh})$ such that
\begin{align*}
 &\p_tw^0_{lh}+T_{1-\bar{Y}}\p_\al r^0_{lh} + \text{cubic and higher terms}\\
 =&  T_{J^{-1}}T_{\bar{r}_\alpha}\bw -  T_{(1-\bar{Y})^2(1+\bw)}T_{\bar{w}_\alpha} R  + G,\\
&\p_tr^0_{lh}-iT_{J^{-\frac{3}{2}}(1-Y)}\p_\al^4w^0_{lh} + \text{cubic and higher terms}\\
=&T_{1-\bar{Y}}T_{\bar{r}_\alpha} R  -\f32 i T_{J^{-\f52}} T_{\p^4_{\al}\bar{w}}\bw - iT_{J^{-\f52}}T_{\p^3_{\al}\bar{w}}\bw_\al +iT_{J^{-\f52}} T_{\bar{w}_{\al\al}}\bw_{\al\al}+ \f32iT_{J^{-\f52}}T_{\bar{w}_{\al}} \p^3_{\al}\bw +K.   
\end{align*}
We consider low-high normal form transformations as the sum of low-high paradifferential bilinear forms of the following type:
\begin{equation*} 
\begin{aligned}
w^0_{lh} &= B^0_{lh}\left(\bar{w}, T_{1-\bar{Y}}\bw \right) + C^0_{lh}\left(\bar{r}, T_{J^{\frac{1}{2}}(1+\bw)^2}R \right),\\
r^0_{lh} &= A^0_{lh} \left(\bar{r}, T_{1-Y}\bw \right) + D^0_{lh}\left(\bar{w}, T_{(1-\bar{Y})(1+\bw)}R\right).
\end{aligned}
\end{equation*}
Using these bilinear forms, we compute
\begin{align*}
&\partial_t w^0_{lh}+T_{1-\bar{Y}} \partial_\alpha r^0_{lh} + \text{cubic and higher terms} \\
=& T_{J^{-1}}\partial_\alpha A^0_{lh}(\bar{r}, \bw) -T_{J^{-1}}B^0_{lh}(\bar{r}_\alpha, \bw) +  T_{J^{-1}}C^0_{lh}(\bar{r}, i\p_\al^4\bw) \\
 &-T_{(1-\bar{Y})^2(1+\bw)}B^0_{lh}(\bar{w}, R_\alpha)- T_{(1-\bar{Y})^2(1+\bw)}C^0_{lh}(i\p_\al^4\bar{w}, R) + T_{(1-\bar{Y})^2(1+\bw)}\partial_\alpha D^0_{lh}(\bar{w}, R),\\
&\partial_t r^0_{lh}  -iT_{J^{-\frac{3}{2}}(1-Y)}\p_\al^4 w^0_{lh}+ \text{cubic and higher  terms} \\
=& -T_{1-\bar{Y}}A^0_{lh}(\bar{r}, R_\alpha) - iT_{1-\bar{Y}}\partial_\alpha^4 C^0_{lh}(\bar{r}, R) - T_{1-\bar{Y}}D^0_{lh}(\bar{r}_\alpha, R)\\
&- iT_{J^{-\frac{5}{2}}}A^0_{lh}(\p_\al^4\bar{w}, \bw) -iT_{J^{-\frac{5}{2}}}\partial_\alpha^4 B^0_{lh}(\bar{w}, \bw) +iT_{J^{-\frac{5}{2}}}D^0_{lh}(\bar{w}, \p_\al^4\bw).
\end{align*}
We write $\mathfrak{a}^0_{lh}(\eta, \zeta)$ for the symbol of $A^0_{lh}(\bar{r}, T_{1-Y}\bw)$, and similarly for other low-high bilinear forms.
To match the low-high part of paradifferential source terms in $(\nP\mathcal{G}_0, \nP\mathcal{K}_0)$, the paradifferential symbols must satisfy the following algebraic system:
\begin{equation*}
\begin{cases}
(\zeta-\eta)\mathfrak{a}^0_{lh}+\eta\mathfrak{b}^0_{lh}+\zeta^4\mathfrak{c}^0_{lh}=-\eta\chi_{1}(\eta, \zeta),\\
\zeta\mathfrak{b}^0_{lh}+\eta^4\mathfrak{c}^0_{lh}-(\zeta-\eta)\mathfrak{d}^0_{lh}=-\eta\chi_{1}(\eta, \zeta),\\
\zeta\mathfrak{a}^0_{lh}+(\zeta-\eta)^4\mathfrak{c}^0_{lh}-\eta\mathfrak{d}^0_{lh}=\eta\chi_{1}(\eta, \zeta),\\
\eta^4 \mathfrak{a}^0_{lh}+(\zeta-\eta)^4\mathfrak{b}^0_{lh}-\zeta^4\mathfrak{d}^0_{lh}=\left(\f32\eta^4-\eta^3\zeta-\eta^2\zeta^2+\f32\eta\zeta^3\right)\chi_{1}(\eta, \zeta),
\end{cases}
\end{equation*}
where  the symbol $\chi_{1}(\eta, \zeta)$ is defined in \eqref{ChiOnelh} to select the low-high frequencies.
The expressions for low-high paradifferential symbols are given by
\begin{align*}
\mathfrak{a}^0_{lh}=&\f{5(2\eta^6-5\eta^5\zeta+12\eta^4\zeta-13\eta^3\zeta^3+12\eta^2\zeta^4-5\eta\zeta^5+2\zeta^6)}{2(4\eta^6-12\eta^5\zeta+37\eta^4\zeta-54\eta^3\zeta^3+75\eta^2\zeta^4-50\eta\zeta^5+25\zeta^6)}\chi_{1}(\eta, \zeta),\\
\mathfrak{b}^0_{lh}=&\f{2\eta^6-11\eta^5\zeta+11\eta^4\zeta-17\eta^3\zeta^3-25\eta^2\zeta^4+25\eta\zeta^5-25\zeta^6}{2(4\eta^6-12\eta^5\zeta+37\eta^4\zeta-54\eta^3\zeta^3+75\eta^2\zeta^4-50\eta\zeta^5+25\zeta^6)}\chi_{1}(\eta, \zeta),\\
\mathfrak{c}^0_{lh}=&-\f{5\zeta(\eta^2-\eta\zeta+\zeta^2)}{4\eta^6-12\eta^5\zeta+37\eta^4\zeta-54\eta^3\zeta^3+75\eta^2\zeta^4-50\eta\zeta^5+25\zeta^6}\chi_{1}(\eta, \zeta),\\
\mathfrak{d}^0_{lh}=&-\f{8\eta^6-24\eta^5\zeta+49\eta^4\zeta-58\eta^3\zeta^3+75\eta^2\zeta^4-50\eta\zeta^5+25\zeta^6}{2(4\eta^6-12\eta^5\zeta+37\eta^4\zeta-54\eta^3\zeta^3+75\eta^2\zeta^4-50\eta\zeta^5+25\zeta^6)}\chi_{1}(\eta, \zeta).
\end{align*}
The leading terms of these symbols are
\begin{equation*}
\mathfrak{a}^0_{lh} \approx\f15\chi_{1}(\eta, \zeta), \quad \mathfrak{b}^0_{lh} \approx -\f12\chi_{1}(\eta, \zeta),  \quad \mathfrak{c}^0_{lh} \approx -\f15\zeta^{-3}\chi_{1}(\eta, \zeta), \quad \mathfrak{d}^0_{lh} \approx -\f12\chi_{1}(\eta, \zeta).   
\end{equation*}
Consequently, we obtain
\begin{align*}
    &w^0_{lh} = -\f12T_{\bar{w}}T_{1-Y}\bw +\f{i}{5} T_{\bar{r}}  T_{J^{\frac{1}{2}}(1+\bw)^2} \p_\al^{-3} R + \text{lower order terms},\\
    &r^0_{lh} = \f15 T_{\bar{r}}T_{1-Y}\bw -\f12T_{\bar{w}_{\al}}  T_{(1-\bar{Y})(1+\bw)}  \p^{-1}_{\al}R + \text{lower order terms},
\end{align*}
which satisfy the estimate
\begin{equation*}
    \|(w^0_{lh}, r^0_{lh})\|_{\H^0} \lesssim \mathcal{A}_{\f32}\|(w,r)\|_{\H^0}.
\end{equation*}

\subsubsection{Balanced quadratic normal form analysis for $(\nP\mathcal{G}_0, \nP\mathcal{K}_0)$} 
Next, we compute balanced quadratic normal form variables $(w^0_{bal}, r^0_{bal})$ to eliminate the balanced portion of the $(\nP\mathcal{G}_0, \nP\mathcal{K}_0)$.
In other words, we seek $(w^0_{bal}, r^0_{bal})$ such that
\begin{align*}
 &\p_tw^0_{bal}+T_{1-\bar{Y}}\p_\al r^0_{bal} + \text{cubic and higher terms}\\
 =&  T_{J^{-1}} \nP \Pi(\bar{r}_\alpha, \bw) -  T_{(1-\bar{Y})^2(1+\bw)} \nP \Pi(\bar{w}_\alpha, R)  + G,\\
&\p_tr^0_{bal}-iT_{J^{-\frac{3}{2}}(1-Y)}\p_\al^4w^0_{bal} + \text{cubic and higher terms}\\
=&  T_{1-\bar{Y}}\nP\Pi(\bar{r}_\alpha, R)-\f32 iT_{J^{-\f52}} \nP \Pi(\p^4_{\al}\bar{w},\bw) - iT_{J^{-\f52}} \nP \Pi(\p^3_{\al}\bar{w}, \bw_\al) \\
  &+ iT_{J^{-\f52}}\nP \Pi(\bar{w}_{\al\al},\bw_{\al\al})+ \f32 i T_{J^{-\f52}}\nP \Pi(\bar{w}_{\al},\p^3_{\al}\bw)+K.   
\end{align*}
We consider balanced normal form transformations as the sum of balanced paradifferential bilinear forms of the following type:
\begin{equation*} 
\begin{aligned}
w^0_{bal} &= B^0_{bal}\left(\bar{w}, T_{1-\bar{Y}}\bw \right) + C^0_{bal}\left(\bar{r}, T_{J^{\frac{1}{2}}(1+\bw)^2}R \right),\\
r^0_{bal} &= A^0_{bal} \left(\bar{r}, T_{1-Y}\bw \right) + D^0_{bal}\left(\bar{w}, T_{(1-\bar{Y})(1+\bw)}R\right).
\end{aligned}
\end{equation*}
For these bilinear forms, we compute
\begin{align*}
&\partial_t w^0_{bal}+T_{1-\bar{Y}} \partial_\alpha r^0_{bal} + \text{cubic and higher terms} \\
=& T_{J^{-1}}\partial_\alpha A^0_{bal}(\bar{r}, \bw) -T_{J^{-1}}B^0_{bal}(\bar{r}_\alpha, \bw) +  T_{J^{-1}}C^0_{bal}(\bar{r}, i\p_\al^4\bw) \\
 &-T_{(1-\bar{Y})^2(1+\bw)}B^0_{bal}(\bar{w}, R_\alpha)- T_{(1-\bar{Y})^2(1+\bw)}C^0_{bal}(i\p_\al^4\bar{w}, R) + T_{(1-\bar{Y})^2(1+\bw)}\partial_\alpha D^0_{bal}(\bar{w}, R),\\
&\partial_t r^0_{bal}  -iT_{J^{-\frac{3}{2}}(1-Y)}\p_\al^4 w^0_{bal}+ \text{cubic and higher  terms} \\
=& -T_{1-\bar{Y}}A^0_{bal}(\bar{r}, R_\alpha) - iT_{1-\bar{Y}}\partial_\alpha^4 C^0_{bal}(\bar{r}, R) - T_{1-\bar{Y}}D^0_{bal}(\bar{r}_\alpha, R)\\
&- iT_{J^{-\frac{5}{2}}}A^0_{bal}(\p_\al^4\bar{w}, \bw) -iT_{J^{-\frac{5}{2}}}\partial_\alpha^4 B^0_{bal}(\bar{w}, \bw) +iT_{J^{-\frac{5}{2}}}D^0_{bal}(\bar{w}, \p_\al^4\bw).
\end{align*}
We write $\mathfrak{a}^0_{bal}(\eta, \zeta)$ for the symbol of $A^0_{bal}(\bar{r}, T_{1-Y}\bw)$, and similarly for other balanced bilinear forms.
To match the balanced part of the paradifferential source terms in $(\nP\mathcal{G}_0, \nP\mathcal{K}_0)$, the paradifferential symbols must solve the following algebraic system: 
\begin{equation*}
\begin{cases}
(\zeta-\eta)\mathfrak{a}^{0}_{bal}+\eta\mathfrak{b}^{0}_{bal}+\zeta^4\mathfrak{c}^{0}_{bal}=-\eta\chi_{2}(\eta,\zeta)1_{\zeta<\eta},\\
\zeta\mathfrak{b}^{0}_{bal}+\eta^4\mathfrak{c}^{0}_{bal}-(\zeta-\eta)\mathfrak{d}^{0}_{bal}=-\eta\chi_{2}(\eta,\zeta)1_{\zeta<\eta},\\
\zeta\mathfrak{a}^{0}_{bal}+(\zeta-\eta)^4\mathfrak{c}^{0}_{bal}-\eta\mathfrak{d}^{0}_{bal}=\eta\chi_{2}(\eta,\zeta)1_{\zeta<\eta},\\
\eta^4 \mathfrak{a}^{0}_{bal}+(\zeta-\eta)^4\mathfrak{b}^{0}_{bal}-\zeta^4\mathfrak{d}^{0}_{bal}=\left(\f32\eta^4-\eta^3\zeta-\eta^2\zeta^2+\f32\eta\zeta^3\right)\chi_{2}(\eta,\zeta)1_{\zeta<\eta},
\end{cases}
\end{equation*}
where  the symbol $\chi_{2}(\eta, \zeta)$ is defined in \eqref{ChiTwohh} to select the balanced frequencies, and the indicator function $1_{\zeta<\eta}$ represents the holomorphic projection.
The expressions for the balanced paradifferential symbols are:
\begin{align*}
\mathfrak{a}^0_{bal}=&\f{5(2\eta^6-5\eta^5\zeta+12\eta^4\zeta^2-13\eta^3\zeta^3+12\eta^2\zeta^4-5\eta\zeta^5+2\zeta^6)}{2(4\eta^6-12\eta^5\zeta+37\eta^4\zeta-54\eta^3\zeta^3+75\eta^2\zeta^4-50\eta\zeta^5+25\zeta^6)}\chi_{2}(\eta, \zeta)1_{\zeta<\eta},\\
\mathfrak{b}^0_{bal}=&\f{2\eta^6-11\eta^5\zeta+11\eta^4\zeta^2-17\eta^3\zeta^3-25\eta^2\zeta^4+25\eta\zeta^5-25\zeta^6}{2(4\eta^6-12\eta^5\zeta+37\eta^4\zeta-54\eta^3\zeta^3+75\eta^2\zeta^4-50\eta\zeta^5+25\zeta^6)}\chi_{2}(\eta, \zeta)1_{\zeta<\eta},\\
\mathfrak{c}^0_{bal}=&-\f{5\zeta(\eta^2-\eta\zeta+\zeta^2)}{4\eta^6-12\eta^5\zeta+37\eta^4\zeta-54\eta^3\zeta^3+75\eta^2\zeta^4-50\eta\zeta^5+25\zeta^6}\chi_{2}(\eta, \zeta)1_{\zeta<\eta},\\
\mathfrak{d}^0_{bal}=&-\f{8\eta^6-24\eta^5\zeta+49\eta^4\zeta^2-58\eta^3\zeta^3+75\eta^2\zeta^4-50\eta\zeta^5+25\zeta^6}{2(4\eta^6-12\eta^5\zeta+37\eta^4\zeta-54\eta^3\zeta^3+75\eta^2\zeta^4-50\eta\zeta^5+25\zeta^6)}\chi_{2}(\eta, \zeta)1_{\zeta<\eta}.
\end{align*}
The leading terms of these symbols are
\begin{align*}
&\mathfrak{a}^0_{bal} \approx \f15\chi_{2}(\eta, \zeta)1_{\zeta<\eta}, \quad \mathfrak{b}^0_{bal} \approx -\f45\eta\zeta^{-1}\chi_{2}(\eta, \zeta)1_{\zeta<\eta},\\ &\mathfrak{c}^0_{bal} \approx -\f15\zeta^{-3}\chi_{2}(\eta, \zeta)1_{\zeta<\eta}, \quad \mathfrak{d}^0_{bal} \approx -\f12\eta\zeta^{-1}\chi_{2}(\eta, \zeta)1_{\zeta<\eta}.   
\end{align*}
Hence, we obtain the estimate
\begin{equation*}
    \|(w^0_{bal}, r^0_{bal})\|_{\H^0} \lesssim \mathcal{A}_{0}\|(w,r)\|_{\H^0}.
\end{equation*}
Moreover, the cubic and higher-order terms produced by the balanced normal form variables $(w^0{bal}, r^0_{bal})$ are perturbative.

\subsubsection{Cubic normal form analysis for $(\nP\mathcal{G}_0, \nP\mathcal{K}_0)$}
Finally, we construct cubic normal form variables $(w^0_c, r^0_c)$ so that $(w{NF}^0, r{NF}^0)$ satisfy \eqref{wZeroNFPGKZero}. Recall that the quadratic normal form variables $(w^0_{lh}+w^0_{bal}, r^0_{lh}+ r^0_{bal})$ satisfy the system
\begin{equation*} 
\begin{cases}
T_{D_t}(w^0_{lh}+w^0_{bal})+T_{1-\bar{Y}}\p_\al (r^0_{lh}+r^0_{bal})+T_{T_{1-\bar{Y}}R_{\al}}(w^0_{lh}+w^0_{bal})= -\nP \mathcal{G}_0 + \mathcal{G}_0^{[3]} + G,\\
T_{D_t}(r^0_{lh}+r^0_{bal})-i\mathcal{L}_{para}(w^0_{lh}+w^0_{bal})=- \nP \mathcal{K}_0 + \mathcal{K}_0^{[3]} +K,
\end{cases}
\end{equation*}
where the non-perturbative cubic terms $(\mathcal{G}0^{[3]}, \mathcal{K}0^{[3]})$ are given by
\begin{align*}
 &\mathcal{G}_0^{[3]} = \f12 T_{\bar{w}}T_{\bw_\al}T_{J^{-1}}R - \f12 T_{\bar{w}}T_{\bar{\bw}_\al}T_{(1-\bar{Y})^2}R + \f12 T_{\bar{w}} T_R T_{J^{-1}}\bw_\al + \f12 T_{\bar{w}} T_{\bar{R}} T_{(1-Y)^2}\bw_\al,\\
  &\mathcal{K}_0^{[3]} = -\f52i T_{\bar{w}} T_{\bw_\al} T_{J^{-\f32}(1-Y)^3} \p_\al^3 \bw -\f12i T_{\bar{w}} T_{\bar{\bw}_\al} T_{J^{-\f52}(1-Y)}  \p_\al^3 \bw.
\end{align*}
To eliminate non-resonant terms in $(\mathcal{G}_0^{[3]}, \mathcal{K}_0^{[3]})$
\begin{equation*}
\Big(-\f12 T_{\bar{w}}T_{\bar{\bw}_\al}T_{(1-\bar{Y})^2}R + \f12 T_{\bar{w}} T_{\bar{R}} T_{(1-Y)^2}\bw_\al, -\f12i T_{\bar{w}} T_{\bar{\bw}_\al} T_{J^{-\f52}(1-Y)}  \p_\al^3 \bw \Big), 
\end{equation*} 
we consider the system of equations for auxiliary normal form transformations $(\bw^{0,1}_c, R^{0,1}_c)$:
\begin{equation*}  
\begin{cases}
\p_t \bw^{0,1}_c+T_{1-\bar{Y}}\p_\al R^{0,1}_c=\f12 T_{\bar{\bw}_\al}T_{(1-\bar{Y})^2}R -\f12  T_{\bar{R}} T_{(1-Y)^2}\bw_\al + \text{cubic and higher terms},\\
\p_t R^{0,1}_c-iT_{J^{-\f32}(1-Y)}\p_\al^4\bw^{0,1}_c= \f12iT_{\bar{\bw}_\al} T_{J^{-\f52}(1-Y)}  \p_\al^3 \bw  + \text{cubic and higher terms}.
\end{cases}
\end{equation*}
 Auxiliary normal form transformations $(\bw^{0,1}_c, R^{0,1}_c)$ can be chosen as the sum of low-high paradifferential bilinear forms of the following type:
\begin{align*}
\bw^{0,1}_{c} =& B^{0,1}_{c}\left(\bar{\bw}, T_{J^{-1}}\bw \right) + C^{0,1}_{c}\left(\bar{R}, T_{J^{\f32}} R\right),\\
R^{0,1}_{c} =& A^{0,1}_{c} \left(\bar{R}, T_{(1-Y)^2(1+\bar{\bw})}\bw \right) + D^{0,1}_{c}\left(\bar{\bw}, T_{(1-\bar{Y})}R\right).    
\end{align*} 
A direct computation for $(\bw^{0,1}_c, R^{0,1}_c)$ yields
\begin{align*}
 &\p_t \bw^{0,1}_c+T_{1-\bar{Y}}\p_\al R^{0,1}_c  + \text{cubic and higher terms} \\
=& T_{(1-Y)^2}\partial_\alpha A^{0,1}_c(\bar{R}, \bw) -T_{(1-Y)^2}B^{0,1}_{c}(\bar{R}_\alpha, \bw) +  T_{(1-Y)^2}C^{0,1}_{c}(\bar{R}, i\p_\al^4\bw) \\
 &-T_{(1-\bar{Y})^2}B^{0,1}_{c}(\bar{\bw}, R_\alpha)-T_{(1-\bar{Y})^2}C^{0,1}_{c}(i\p_\al^4\bar{\bw}, R) + T_{(1-\bar{Y})^2}\partial_\alpha D^{0,1}_{c}(\bar{\bw}, R),\\
&\p_t R^0_c-iT_{J^{-\f32}(1-Y)}\p_\al^4\bw^0_c+ \text{cubic and higher  terms} \\
=& -T_{1-Y}A^0_{c}(\bar{R}, R_\alpha) - iT_{1-Y}\partial_\alpha^4 C^0_{c}(\bar{R}, R) - T_{1-Y}D^0_{c}(\bar{R}_\alpha, R)\\
&- iT_{J^{-\f52}(1-Y)}A^0_{c}(\p_\al^4\bar{\bw}, \bw) -iT_{J^{-\f52}(1-Y)}\partial_\alpha^4 B^0_{c}(\bar{\bw}, \bw) +iT_{J^{-\f52}(1-Y)}D^0_{c}(\bar{\bw}, \p_\al^4\bw).
\end{align*}
Let $\mathfrak{a}^{0,1}_{c}(\eta, \zeta)$ denote the symbol of $A^{0,1}_{c}(\bar{R}, T_{(1-Y)^2(1+\bar{\bw})}\bw)$, and similarly for other low-high bilinear forms. To match the low-high part of the paradifferential source terms, the paradifferential symbols must satisfy the following algebraic system:
\begin{equation*}
\begin{cases}
(\zeta - \eta)\mathfrak{a}^{0,1}_{c}+\eta\mathfrak{b}^{0,1}_{c}+\zeta^4\mathfrak{c}^{0,1}_{c}= -\f12 \zeta\chi_1(\eta, \zeta),\\
\zeta\mathfrak{b}^{0,1}_{c}+\eta^4\mathfrak{c}^{0,1}_{c}-(\zeta-\eta)\mathfrak{d}^{0,1}_{c}=\f12 \eta\chi_{1}(\eta,\zeta),\\
\zeta\mathfrak{a}^0_{c}+(\zeta-\eta)^4\mathfrak{c}^0_{c}-\eta\mathfrak{d}^0_{c}=0,\\
\eta^4 \mathfrak{a}^0_{c}+(\zeta-\eta)^4\mathfrak{b}^0_{c}-\zeta^4\mathfrak{d}^0_{c}=\f12\eta\zeta^3\chi_{1}(\xi,\eta),
\end{cases}
\end{equation*}
where  the symbol $\chi_{1}(\xi, \eta)$ is defined in \eqref{ChiOnelh} to select the low-high frequencies.
The solution to this system is
\begin{align*}
\mathfrak{a}^{0,1}_{c}=&\f{\zeta(-2\eta^7+11\eta^6\zeta-33\eta^5\zeta^2+58\eta^4\zeta^3-71\eta^3\zeta^4+52\eta^2\zeta^5-25\eta \zeta^6+5\zeta^7)}{2\eta(\zeta-\eta)(4\eta^6-12\eta^5\zeta+37\eta^4\zeta^2-54\eta^3\eta^3+75\eta^2\zeta^4-50\eta\zeta^5+25\zeta^6)}\chi_{1}(\eta,\zeta),\\
\mathfrak{b}^{0,1}_{c}=&\f{\eta(\eta^2-\eta\zeta+\zeta^2)(2\eta^3\zeta-\eta^2\zeta^2+2\eta\zeta^3+5\zeta^4)}{2(\zeta-\eta)(4\eta^6-12\eta^5\zeta+37\eta^4\zeta^2-54\eta^3\eta^3+75\eta^2\zeta^4-50\eta\zeta^5+25\zeta^6)}\chi_{1}(\eta,\zeta),\\
\mathfrak{c}^{0,1}_{c}=&\f{-2\eta^5+3\eta^4\zeta-7\eta^3\zeta^2-2\eta^2\zeta^3+5\eta\zeta^4+5\zeta^5}{2\eta(\zeta-\eta)(4\eta^6-12\eta^5\zeta+37\eta^4\zeta^2-54\eta^3\eta^3+75\eta^2\zeta^4-50\eta\zeta^5+25\zeta^6)}\chi_{1}(\eta,\zeta),\\
\mathfrak{d}^{0,1}_{c}=&\f{-2\eta^7+11\eta^6\zeta-33\eta^5\zeta^2+63\eta^4\zeta^3-76\eta^3\zeta^4+52\eta^2\zeta^5-20\eta \zeta^6}{2(\zeta-\eta)(4\eta^6-12\eta^5\zeta+37\eta^4\zeta^2-54\eta^3\eta^3+75\eta^2\zeta^4-50\eta\zeta^5+25\zeta^6)}\chi_{1}(\eta,\zeta).
\end{align*}
Therefore, we obtain
\begin{align*}
&w^{0,1}_c = \f1{10}T_{\bar{\bw}_{\al}}T_{J^{-1}}\p^{-1}_{\al}\bw+\f{i}{10} T_{\p^{-1}_{\al}\bar{R}}  T_{J^{\frac{3}{2}}}  \p_\al^{-2}R + \text{lower order terms},\\
&r^{0,1}_c = -\f1{10} T_{\p^{-1}_{\al}\bar{R}}T_{(1-Y)^2(1+\bar{\bw})}\bw_{\al}+\f25 T_{\bar{\bw}_{\al}}T_{1-\bar{Y}}  \p^{-1}_{\al}R + \text{lower order terms}.
\end{align*}
We set $(w^{0,1}_c, r^{0,1}_c) = \big(T_{\bar{w}}\bw^0_c, T_{\bar{w}}R^0_c \big)$. Then
\begin{equation*}
\|(w^{0,1}_c, r^{0,1}_c)\|_{\H^0}  \lesssim\mathcal{A}_0 \mathcal{A}_\f32\|(w,r)\|^2_{\H^0},
\end{equation*}
and this pair satisfies the system 
\begin{align*} 
&T_{D_t}w^{0,1}_c+T_{1-\bar{Y}}\p_\al r^{0,1}_c+T_{T_{1-\bar{Y}}R_{\al}}w^{0,1}_c\\
=& \f12 T_{\bar{w}}T_{\bar{\bw}_\al}T_{(1-\bar{Y})^2}R -\f12  T_{\bar{w}}T_{\bar{R}} T_{(1-Y)^2}\bw_\al+\f1{10}\bar{w}_tT_{\bar{\bw}_{\al}}T_{\bar{\bw}_{\al}}T_{J^{-1}}\p^{-1}_{\al}\bw\\
&+\f{i}{10} T_{\p^{-1}_{\al}\bar{R}}  T_{J^{\frac{3}{2}}}  \p_\al^{-2}R + G + \text{quartic and higher terms}\\
=& \f12 T_{\bar{w}}T_{\bar{\bw}_\al}T_{(1-\bar{Y})^2}R -\f12  T_{\bar{w}}T_{\bar{R}} T_{(1-Y)^2}\bw_\al + G,\\
&T_{D_t}r^{0,1}_c-i\mathcal{L}_{para}w^{0,1}_c \\
=&\f12iT_{\bar{w}}T_{\bar{\bw}_\al} T_{J^{-\f52}(1-Y)}  \p_\al^3 \bw-\f1{10} T_{\bar{w}_t}T_{\p^{-1}_{\al}\bar{R}}T_{(1-Y)^2(1+\bar{\bw})}\bw_{\al}\\
&+\f25 T_{\bar{w}_t}T_{\bar{\bw}_{\al}}T_{1-\bar{Y}}  \p^{-1}_{\al}R+K+ \text{quartic and higher terms}\\
=& \f12iT_{\bar{w}}T_{\bar{\bw}_\al} T_{J^{-\f52}(1-Y)}  \p_\al^3 \bw  + \f12  T_wT_{\bar{R}} T_{1-Y}R_\al + K.
\end{align*}

The remaining terms in $(\mathcal{G}_0^{[3]}, \mathcal{K}_0^{[3]})$ may have resonances. 
To eliminate these terms, as above, we consider the system  of equations for auxiliary normal form transformations $(\bw^{0,2}_c, R^{0,2}_c)$,
\begin{equation*}  
\begin{cases}
\p_t \bw^{0,2}_c+T_{1-\bar{Y}}\p_\al R^{0,2}_c= -\f12 T_{\bw_\al}T_{J^{-1}}R -\f12  T_R T_{J^{-1}}\bw_\al  + \text{cubic and higher terms},\\
\p_t R^{0,2}_c-iT_{J^{-\f32}(1-Y)}\p_\al^4\bw^{0,2}_c= \f52i  T_{\bw_\al} T_{J^{-\f32}(1-Y)^2} \p_\al^3 \bw  + \text{cubic and higher terms}.
\end{cases}
\end{equation*}
We consider auxiliary normal form transformations as the sum of low-high paradifferential bilinear forms of the following type:
\begin{equation*} 
\bw^{0,2}_{c} = B^{0,2}_{c}\left(\bw, T_{(1-Y)^2}\bw \right) + C^{0,2}_{c}\left(R, T_{J^\f12(1+\bw)^2}R \right),\quad R^{0,2}_{c} = A^{0,2}_{c} \left(R, T_{1-Y}\bw \right) + D^{0,2}_{c}\left(\bw, T_{1-Y}R\right).
\end{equation*}
Then, we find $(\bw^{0,2}_{c}, R^{0,2}_{c})$ such that
\begin{align*}
 &\p_t \bw^{0,2}_c+T_{1-\bar{Y}}\p_\al R^{0,2}_c  + \text{cubic and higher terms} \\
=& T_{J^{-1}}\partial_\alpha A^{0,2}_c(R, \bw) -T_{J^{-1}}B^{0,2}_{c}(R_\alpha, \bw) +  T_{J^{-1}}C^{0,2}_{c}(R, i\p_\al^4\bw) \\
&-T_{J^{-1}}B^{0,2}_{c}(\bw, R_\alpha)+ T_{J^{-1}}C^{0,2}_{c}(i\p_\al^4\bw, R) + T_{J^{-1}}\partial_\alpha D^{0,2}_{c}(\bw, R),\\
&\p_t R^{0,2}_c-iT_{J^{-\f32}(1-Y)}\p_\al^4\bw^{0,2}_c+ \text{cubic and higher  terms} \\
=& -T_{1-\bar{Y}}A^{0,2}_{c}(R, R_\alpha) - iT_{1-\bar{Y}}\partial_\alpha^4 C^{0,2}_{c}(R, R) - T_{1-\bar{Y}}D^{0,2}_{c}(R_\alpha, R)\\
&+ iT_{J^{-\f32}(1-Y)^2}A^{0,2}_{c}(\p_\al^4\bw, \bw) -iT_{J^{-\f32}(1-Y)^2}\partial_\alpha^4 B^{0,2}_{c}(\bw, \bw) +iT_{J^{-\f32}(1-Y)^2}D^{0,2}_{c}(\bw, \p_\al^4\bw).
\end{align*}
Here, we write $\mathfrak{a}^{0,2}_{c}(\xi, \eta)$ for the symbol of $A^{0,2}_{c}(R, T_{1-Y}\bw)$, and similarly for other low-high bilinear forms.
To match the low-high part of paradifferential source terms, paradifferential symbols solve the following algebraic system: 
\begin{equation*}
\begin{cases}
(\xi + \eta)\mathfrak{a}^{0,2}_{c}-\xi\mathfrak{b}^{0,2}_{c}+\eta^4\mathfrak{c}^{0,2}_{c}= -\f12 \eta\chi_1(\xi, \eta),\\
\eta\mathfrak{b}^{0,2}_{c}-\xi^4\mathfrak{c}^{0,2}_{c}-(\xi+\eta)\mathfrak{d}^{0,2}_{c}=\f12 \xi\chi_{1}(\xi,\eta),\\
\eta\mathfrak{a}^{0,2}_{c}+(\xi+\eta)^4\mathfrak{c}^{0,2}_{c}+\xi\mathfrak{d}^{0,2}_{c}=0,\\
\xi^4 \mathfrak{a}^{0,2}_{c}-(\xi+\eta)^4\mathfrak{b}^{0,2}_{c}+\eta^4\mathfrak{d}^{0,2}_{c}=\f52\xi\eta^3\chi_{1}(\xi,\eta),
\end{cases}
\end{equation*}
where  the symbol $\chi_{1}(\xi, \eta)$ is defined in \eqref{ChiOnelh} to select the low-high frequencies.
The solution of above system is
\begin{align*}
\mathfrak{a}^{0,2}_{c}=&\f{-5\xi^7-30\xi^6\eta-107\xi^5\eta^2-183\xi^4\eta^3-187\xi^3\eta^4-133\xi^2\eta^5-42\xi \eta^6-5\eta^7}{2\xi(25\xi^6+100\xi^5\eta+200\xi^4\eta^2+246\xi^3\eta^3+200\xi^2\eta^4+100\xi\eta^5+25\eta^6)}\chi_{1}(\xi,\eta),\\
\mathfrak{b}^{0,2}_{c}=&\f{-5\xi^6-10\xi^5\eta -37\xi^4\eta^2-85\xi^3\eta^3-112\xi^2\eta^4-85\xi\eta^5-30\eta^6}
{2(25\xi^6+100\xi^5\eta+200\xi^4\eta^2+246\xi^3\eta^3+200\xi^2\eta^4+100\xi\eta^5+25\eta^6)}\chi_{1}(\xi,\eta),\\
\mathfrak{c}^{0,2}_{c}=&\f{5\xi^5+12\xi^4\eta+15\xi^3\eta^2+25\xi^2\eta^3+22\xi\eta^4+5\eta^5}
{2\xi\eta(25\xi^6+100\xi^5\eta+200\xi^4\eta^2+246\xi^3\eta^3+200\xi^2\eta^4+100\xi\eta^5+25\eta^6)}\chi_{1}(\xi,\eta),\\
\mathfrak{d}^{0,2}_{c}=&-\f{5\xi^7+32\xi^6\eta + 88\xi^5\eta^2+147\xi^4\eta^3+158\xi^3\eta^4+132\xi^2\eta^5+80\xi\eta^6+30\eta^7}
{2\eta(25\xi^6+100\xi^5\eta+200\xi^4\eta^2+246\xi^3\eta^3+200\xi^2\eta^4+100\xi\eta^5+25\eta^6)}\chi_{1}(\xi,\eta).
\end{align*}
Therefore, we obtain
\begin{align*}
    &w^{0,2}_c = -\f35 T_{\bw}T_{(1-Y)^2}\bw - \f{i}{10} T_{\p_\al^{-1}R}  T_{J^{\frac{1}{2}}(1+\bw)^2}  \p_\al^{-2}R + \text{lower order terms},\\
    &r^{0,2}_c = -\f{1}{10} T_{\p_\al^{-1}R}T_{1-Y}\p_\al\bw-\f35 T_{\bw}T_{1-Y}  R + \text{lower order terms}.
\end{align*}
Choose $(w^{0,2}_c, r^{0,2}_c) = \big(T_{\bar{w}}\bw^0_c, T_{\bar{w}}T_{1-Y}R^0_c \big)$. It then follows that
\begin{equation*}
\|(w^{0,2}_c, r^{0,2}_c)\|_{\H^0}  \lesssim\mathcal{A}_0 \mathcal{A}_\f32\|(w,r)\|^2_{\H^0}.
\end{equation*}
This pair satisfies 
\begin{align*} 
&T_{D_t}w^{0,2}_c+T_{1-\bar{Y}}\p_\al r^{0,2}_c+T_{T_{1-\bar{Y}}R_{\al}}w^{0,2}_c\\
=& -\f12 T_{\bar{w}}T_{\bw_\al}T_{J^{-1}}R -\f12  T_{\bar{w}}T_R T_{J^{-1}}\bw_\al -\f35 T_{\bar{w}_t}T_{\bw}T_{(1-Y)^2}\bw \\
&- \f{i}{10} T_{\bar{w}_t}T_{\p_\al^{-1}R}  T_{J^{\frac{1}{2}}(1+\bw)^2}  \p_\al^{-2}R  + \text{quartic and higher terms} + G\\
=& -\f12 T_{\bar{w}}T_{\bw_\al}T_{J^{-1}}R -\f12  T_{\bar{w}}T_R T_{J^{-1}}\bw_\al + G,\\
&T_{D_t}r^{0,2}_c-i\mathcal{L}_{para}w^{0,2}_c  \\
=& \f52i T_{\bar{w}} T_{\bw_\al} T_{J^{-\f32}(1-Y)^3} \p_\al^3 \bw  -\f{1}{10} T_{\bar{w}_t}T_{\p_\al^{-1}R}T_{(1-Y)^2}\p_\al\bw-\f35 T_{\bar{w}_t}T_{\bw}T_{(1-Y)^2}  R \\
&+\text{quartic and higher terms} + K\\
=& \f52i T_{\bar{w}} T_{\bw_\al} T_{J^{-\f32}(1-Y)^3} \p_\al^3 \bw + K.
\end{align*}
As a consequence of above computations, the cubic normal form variables $(w^0_c, r^0_c): = (w^{0,1}_c, r^{0,1}_c) + (w^{0,2}_c, r^{0,2}_c)$ solve the system 
\begin{equation*} 
\begin{cases}
T_{D_t}w^0_c+T_{1-\bar{Y}}\p_\al r^0_c+T_{T_{1-\bar{Y}}R_{\al}}w^0_c= - \mathcal{G}_0^{[3]} + G,\\
T_{D_t}r^0_c-i\mathcal{L}_{para}w^0_c=- \mathcal{K}_0^{[3]} +K.
\end{cases}
\end{equation*}

To conclude this subsection, by choosing normal form variables $(w_{NF}^0, r_{NF}^0)$ as in \eqref{wrNFZeroDef}, they satisfy \eqref{wrNFHZeroBound} and solve the system \eqref{wZeroNFPGKZero}.

\subsection{Normal form analysis for $(\nP\mathcal{G}_1, \nP\mathcal{K}_1)$} \label{s:NormalOne}
In this section, we construct normal form variables $(w_{NF}^1, r_{NF}^1)$ satisfying
\begin{equation} \label{wrNFHOneBound}
    \|(w_{NF}^1, r_{NF}^1)\|_{\H^0} \lesssim_\CalAZ \mathcal{A}_{\f32} \|(w,r)\|_{\H^0},
\end{equation}
such that they solve the system
\begin{equation}  \label{wOneNFPGKOne}
\begin{cases}
T_{D_t}w_{NF}^1+T_{1-\bar{Y}}\p_\al r_{NF}^1+T_{T_{1-\bar{Y}}R_{\al}}w_{NF}^1= -\nP \mathcal{G}_1 + G,\\
T_{D_t}r_{NF}^1-i\mathcal{L}_{para}w_{NF}^1=- \nP \mathcal{K}_1 +K,
\end{cases}
\end{equation}
where $(G,K)$ are perturbative source terms that satisfy \eqref{HalfGK}.

Recall that source terms $(\nP\mathcal{G}_1, \nP\mathcal{K}_1)$ are rewritten as the sum of low-high and balanced paraproducts in Lemma \ref{t:PGKOneExp}.
We will construct $(w_{NF}^1, r_{NF}^1)$ as the sum of low-high quadratic normal form variables, balanced quadratic normal form variables, and cubic normal form variables
\begin{equation} \label{wrNFOneDef}
 (w_{NF}^1, r_{NF}^1) := (w^1_{lh}, r^1_{lh}) + (w_{bal}^1, r_{bal}^1) + (w^1_c, r^1_{c}).   
\end{equation}
Normal form variables $(w^1_{lh}, r^1_{lh})$ and $(w_{bal}^1, r_{bal}^1)$ eliminate the low-high and balanced part of $(\nP\mathcal{G}_1, \nP\mathcal{K}_1)$, and $(w^1_c, r^1_{c})$ eliminates the extra non-perturbative cubic part of source terms produced by quadratic normal form variables.

\subsubsection{Low-high quadratic normal form analysis for $(\nP\mathcal{G}_1, \nP\mathcal{K}_1)$} 
We first construct the low‑high quadratic normal form variables $(w^1_{lh}, r^1_{lh})$ to eliminate the low-high portion of the $(\nP\mathcal{G}_1, \nP\mathcal{K}_1)$.
That is, we seek $(w^1_{lh}, r^1_{lh})$ satisfying
\begin{align*}
 &\p_tw^1_{lh}+T_{1-\bar{Y}}\p_\al r^1_{lh} + \text{cubic and higher terms}\\
 =&  -T_{1-\bar{Y}} T_w R_{\al}-T_{w_{\al}}T_{1-\bar{Y}}R + G,\\
&\p_tr^1_{lh}-iT_{J^{-\frac{3}{2}}(1-Y)}\p_\al^4w^1_{lh} + \text{cubic and higher terms}\\
=& -T_{1-\bar{Y}}T_{r_{\al}}R -iT_{(1-Y)^2J^{-\f32}}T_{w}\p^4_{\al}\mathbf{W} -\f52iT_{(1-Y)^2J^{-\f32}}T_{w_{\al}}\p^3_{\al}\mathbf{W} -5iT_{(1-Y)^2J^{-\f32}}T_{w_{\al\al}} \mathbf{W}_{\al\al} \\
  &-5iT_{(1-Y)^2J^{-\f32}}T_{\p_{\al}^3w}\mathbf{W}_{\al} -\f52iT_{(1-Y)^2J^{-\f32}}T_{\p_{\al}^4w} \bw+K.   
\end{align*}
We consider low-high normal form transformations as the sum of low-high paradifferential bilinear forms of the following type:
\begin{equation*} 
\begin{aligned}
w^1_{lh} &= B^1_{lh}\left(w, T_{1-Y}\bw \right) + C^1_{lh}\left(r, T_{J^{\frac{1}{2}}(1+\bw)^2}R \right),\\
r^1_{lh} &= A^1_{lh} \left(r, T_{1-Y}\bw \right) + D^1_{lh}\left(w, R\right).
\end{aligned}
\end{equation*}
We then compute using the above bilinear forms,
\begin{align*}
&\partial_t w^1_{lh}+T_{1-\bar{Y}} \partial_\alpha r^1_{lh} + \text{cubic and higher terms} \\
=& T_{J^{-1}}\partial_\alpha A^1_{lh}(r, \bw) -T_{J^{-1}}B^1_{lh}(r_\alpha, \bw) +  T_{J^{-1}}C^1_{lh}(r, i\p_\al^4\bw) \\
 &-T_{1-\bar{Y}}B^1_{lh}(w, R_\alpha)+T_{1-\bar{Y}}C^1_{lh}(i\p_\al^4w, R) + T_{1-\bar{Y}}\partial_\alpha D^1_{lh}(w, R),\\
&\partial_t r^1_{lh}  -iT_{J^{-\frac{3}{2}}(1-Y)}\p_\al^4 w^1_{lh}+ \text{cubic and higher  terms} \\
=& -T_{1-\bar{Y}}A^1_{lh}(r, R_\alpha) - iT_{1-\bar{Y}}\partial_\alpha^4 C^1_{lh}(r, R) - T_{1-\bar{Y}}D^1_{lh}(r_\alpha, R)\\
&+ iT_{J^{-\f32}(1-Y)^2}A^1_{lh}(\p_\al^4w, \bw) -iT_{J^{-\f32}(1-Y)^2}\partial_\alpha^4 B^1_{lh}(w, \bw) +iT_{J^{-\f32}(1-Y)^2}D^1_{lh}(w, \p_\al^4\bw).
\end{align*}
We write $\mathfrak{a}^1_{lh}(\xi, \eta)$ for the symbol of $A^1_{lh}(r, T_{1-Y}\bw)$, and similarly for other low-high bilinear forms.
To match the low-high part of paradifferential source terms in $(\nP\mathcal{G}_1, \nP\mathcal{K}_1)$, paradifferential symbols solve the following algebraic system: 
\begin{equation*}
\begin{cases}
(\xi + \eta)\mathfrak{a}^1_{lh}-\xi\mathfrak{b}^1_{lh}+\eta^4\mathfrak{c}^1_{lh}=0,\\
\eta\mathfrak{b}^1_{lh}-\xi^4\mathfrak{c}^1_{lh}-(\xi+\eta)\mathfrak{d}^1_{lh}=(\xi+\eta)\chi_{1}(\xi,\eta),\\
\eta\mathfrak{a}^1_{lh}+(\xi+\eta)^4\mathfrak{c}^1_{lh}+\xi\mathfrak{d}^1_{lh}=\xi\chi_{1}(\xi,\eta),\\
\xi^4 \mathfrak{a}^1_{lh}-(\xi+\eta)^4\mathfrak{b}^1_{lh}+\eta^4\mathfrak{d}^1_{lh}=-\left(\eta^4+\f52\xi\eta^3+5\xi^2\eta^2+5\xi^3\eta+\xi^4\right)\chi_{1}(\xi,\eta),
\end{cases}
\end{equation*}
where  the symbol $\chi_{1}(\xi, \eta)$ is defined in \eqref{ChiOnelh} to select the low-high frequencies. The expressions for the low-high paradifferential symbols are provided below.
\begin{align*}
\mathfrak{a}^1_{lh}=&\f{10\xi^7+70\xi^6\eta+170\xi^5\eta^2+227\xi^4\eta^3+184\xi^3\eta^4+80\xi^2\eta^5+5\xi\eta^6+10\eta^7}{2\eta(25\xi^6+100\xi^5\eta+200\xi^4\eta^2+246\xi^3\eta^3+200\xi^2\eta^4+100\xi\eta^5+25\eta^6)}\chi_{1}(\xi,\eta),\\
\mathfrak{b}^1_{lh}=&\f{10\xi^7+80\xi^6\eta+240\xi^5\eta^2+397\xi^4\eta^3+427\xi^3\eta^4+300\xi^2\eta^5+125\xi\eta^6+25\eta^7}
{2\eta(25\xi^6+100\xi^5\eta+200\xi^4\eta^2+246\xi^3\eta^3+200\xi^2\eta^4+100\xi\eta^5+25\eta^6)}\chi_{1}(\xi,\eta),\\
\mathfrak{c}^1_{lh}=&\f{8\xi^4+18\xi^3\eta+20\xi^2\eta^2+15\xi\eta^3+5\eta^4}
{\eta(25\xi^6+100\xi^5\eta+200\xi^4\eta^2+246\xi^3\eta^3+200\xi^2\eta^4+100\xi\eta^5+25\eta^6)}\chi_{1}(\xi,\eta),\\
\mathfrak{d}^1_{lh}=&-\f{(8\xi^3+10\xi^2\eta+10\xi\eta^2+5\eta^3)(2\xi^4+5\xi^3\eta+10\xi^2\eta^2+10\xi\eta^3+5\eta^4)}
{2\eta(25\xi^6+100\xi^5\eta+200\xi^4\eta^2+246\xi^3\eta^3+200\xi^2\eta^4+100\xi\eta^5+25\eta^6)}\chi_{1}(\xi,\eta).
\end{align*}
The leading terms of these symbols are
\begin{equation*}
\mathfrak{a}^1_{lh} \approx\f15\chi_{1}(\xi, \eta), \quad \mathfrak{b}^1_{lh} \approx \f12\chi_{1}(\xi, \eta),  \quad \mathfrak{c}^1_{lh} \approx \f15\eta^{-3}\chi_{1}(\xi, \eta), \quad \mathfrak{d}^1_{lh} \approx -\f12\chi_{1}(\xi, \eta).   
\end{equation*}
Hence, we obtain
\begin{align*}
    &w^1_{lh} = \f12 T_{w}T_{1-Y}\bw - \f{i}{5} T_{r}  T_{J^{\frac{1}{2}}(1+\bw)^2} \p_\al^{-3} R + \text{lower order terms},\\
    &r^1_{lh} = \f15 T_{r}T_{1-Y}\bw-\f12 T_{w}  R + \text{lower order terms},
\end{align*}
so that they satisfy the estimate
\begin{equation*}
    \|(w^1_{lh}, r^1_{lh})\|_{\H^0} \lesssim \mathcal{A}_{\f32}\|(w,r)\|_{\H^0}.
\end{equation*}

\subsubsection{Balanced quadratic normal form analysis for $(\nP\mathcal{G}_1, \nP\mathcal{K}_1)$} 
Next, we compute balanced quadratic normal form variables $(w^1_{bal}, r^1_{bal})$ to eliminate the balanced portion of the $(\nP\mathcal{G}_1, \nP\mathcal{K}_1)$. That is, we seek $(w^1_{bal}, r^1_{bal})$ such that
\begin{align*}
 &\p_tw^1_{bal}+T_{1-\bar{Y}}\p_\al r^1_{bal} + \text{cubic and higher terms}\\
 =&  - T_{1-\bar{Y}}\p_\al\Pi(w_\al, R)- T_{1-Y}\nP\Pi(w_\al, \bar{R}) + T_{(1-\bar{Y})^2}\nP \Pi(r_\al, \bar{\bw}) +G,\\
&\p_tr^1_{bal}-iT_{J^{-\frac{3}{2}}(1-Y)}\p_\al^4w^1_{bal} + \text{cubic and higher terms}\\
=& -T_{1-\bar{Y}}\Pi(r_{\al},R) -\f52iT_{(1-Y)^2J^{-\f32}}\Pi(\p_{\al}^4w, \bw) -5iT_{(1-Y)^2J^{-\f32}}\Pi(\p^3_{\al}w,\mathbf{W}_{\al}) \\
&-5iT_{(1-Y)^2J^{-\f32}}\Pi(w_{\al\al},\mathbf{W}_{\al\al}) -\f52i T_{(1-Y)^2J^{-\f32}}\Pi(w_{\al},\p^3_{\al}\mathbf{W}) -iT_{(1-Y)^2J^{-\f32}}\Pi(w,\p^4_{\al}\mathbf{W})\\
&-T_{1-Y}\nP\Pi(r_{\al},\bar{R})-\f32i  T_{J^{-\f52}} \nP\Pi(\p^4_{\al}w,\bar{\bw})-iT_{J^{-\f52}}\nP\Pi(\p^3_{\al}w,\bar{\bw}_{\al})\\
&+iT_{J^{-\f52}}\nP\Pi(w_{\al\al},\bar{\mathbf{W}}_{\al\al}) +\f32iT_{J^{-\f52}}\nP\Pi(w_{\al},\p^3_{\al}\bar{\mathbf{W}}) + K.   
\end{align*}

We consider balanced normal form transformations as the sum of balanced paradifferential bilinear forms of the following type:
\begin{equation*} 
\begin{aligned}
w^1_{bal} &=B^{1,h}_{bal}\left(T_{1-Y}\bw ,w\right) + C^{1,h}_{bal}\left(T_{J^{\frac{1}{2}}(1+\bw)^2}R ,r\right)+ B^{1,a}_{bal}\left(T_{1-\bar{Y}}\bar{\bw},w\right) + C^{1,a}_{bal}\left(T_{J^{\frac{3}{2}}}\bar{R} ,r\right),\\
r^1_{bal} &=  A^{1,h}_{bal} \left( T_{1-Y}\bw,r\right) + D^{1,h}_{bal}\left(R,w\right)+A^{1,a}_{bal} \left(T_{1-\bar{Y}}\bar{\bw},r\right) + D^{1,a}_{bal}\left( T_{(1-Y)(1+\bar{\bw})}\bar{R},w\right).
\end{aligned}
\end{equation*}
For above bilinear forms, we compute
\begin{align*}
&\partial_t w^1_{bal}+T_{1-\bar{Y}} \partial_\alpha r^1_{bal} + \text{cubic and higher terms} \\
=& T_{J^{-1}}\partial_\alpha A^{1,h}_{bal}(\bw,r) -T_{J^{-1}}B^{1,h}_{bal}(\bw,r_{\al}) +  T_{J^{-1}}C^{1,h}_{bal}(i\p^4_{\al}\bw,r) \\
&-T_{1-\bar{Y}}B^{1,h}_{bal}( R_\alpha,w)+T_{1-\bar{Y}}C^{1,h}_{bal}(R,i\p_\al^4w) + T_{1-\bar{Y}}\partial_\alpha D^{1,h}_{bal}(R,w)\\
&+T_{(1-\bar{Y})^2}\p_{\al}A^{1,a}_{bal}(\bar{\bw},r)-T_{(1-\bar{Y})^2}B^{1,a}_{bal}(\bar{\bw},r_{\al})-T_{(1-\bar{Y})^2}C^{1,a}_{bal}(i\p^4_{\al}\bar{\bw},r)\\
 &-T_{1-Y}B^{1,a}_{bal}(\bar{R}_\alpha,w)+T_{1-Y}C^{1,a}_{bal}(\bar{R},i\p_\al^4w) + T_{1-Y}\partial_\alpha D^{1,a}_{bal}(\bar{R},w),\\
&\partial_t r^1_{bal}  -iT_{J^{-\frac{3}{2}}(1-Y)}\p_\al^4 w^1_{bal}+ \text{cubic and higher  terms} \\
=& -T_{1-\bar{Y}}A^{1,h}_{bal}( R_\alpha,r) - iT_{1-\bar{Y}}\partial_\alpha^4 C^{1,h}_{bal}(R,r) - T_{1-\bar{Y}}D^{1,h}_{bal}(R,r_\alpha)\\
&+ iT_{J^{-\frac{3}{2}}(1-Y)^2}A^{1,h}_{bal}(\bw,\p_\al^4w) -iT_{J^{-\frac{3}{2}}(1-Y)^2}\partial_\alpha^4 B^{1,h}_{bal}(\bw,w) +iT_{J^{-\frac{3}{2}}(1-Y)^2}D^{1,h}_{bal}(\p_\al^4\bw,w)\\
 &-T_{1-Y}A^{1,a}_{bal}(\bar{R}_\alpha,r) - iT_{1-Y}\partial_\alpha^4 C^{1,a}_{bal}( \bar{R},r) - T_{1-Y}D^{1,a}_{bal}(\bar{R},r_\alpha)\\
&+ iT_{J^{-\f52}}A^{1,a}_{bal}(\bar{\bw},\p_\al^4w) -iT_{J^{-\f52}}\partial_\alpha^4 B^{1,a}_{bal}(\bar{\bw},w) -iT_{J^{-\f52}}D^{1,h}_{bal}(\p_\al^4\bar{\bw},w).
\end{align*}
Here, we write $\mathfrak{a}^{1,h}_{bal}(\xi, \eta)$ for the symbol of $A^{1,h}_{bal}(T_{1-Y}\bw,r)$, $\mathfrak{a}^{1,a}_{bal}(\eta, \zeta)$ for the symbol of $A^{1,a}_{bal}(T_{1-\bar{Y}}\bar{\bw},r)$ and similarly for other balanced bilinear forms.
To match the balanced part of paradifferential source terms in $(\nP\mathcal{G}_1, \nP\mathcal{K}_1)$, paradifferential symbols of the holomorphic type solve the following algebraic systems
\begin{equation*}
\begin{cases}
(\xi + \eta)\mathfrak{a}^{1,h}_{bal}-\eta\mathfrak{b}^{1,h}_{bal}+\xi^4\mathfrak{c}^{1,h}_{bal}=0,\\
\xi\mathfrak{b}^{1,h}_{bal}-\eta^4\mathfrak{c}^{1,h}_{bal}-(\xi+\eta)\mathfrak{d}^{1,h}_{bal}=(\xi+\eta)\chi_{2}(\xi,\eta),\\
\xi\mathfrak{a}^{1,h}_{bal}+(\xi+\eta)^4\mathfrak{c}^{1,h}_{bal}+\eta\mathfrak{d}^{1,h}_{bal}=\eta\chi_{2}(\xi,\eta),\\
\eta^4 \mathfrak{a}^{1,h}_{bal}-(\xi+\eta)^4\mathfrak{b}^{1,h}_{bal}+\xi^4\mathfrak{d}^{1,h}_{bal}=-\left(\f52\eta^4+5\eta^3\xi+5\eta^2\xi^2+\f52\eta\xi^3+\xi^4\right)\chi_{2}(\xi,\eta).
\end{cases}
\end{equation*}
And for paradifferential symbols of the mixed type,
\begin{equation*}
\begin{cases}
(\zeta-\eta)\mathfrak{a}^{1,a}_{bal}-\zeta\mathfrak{b}^{1,a}_{bal}-\eta^4\mathfrak{c}^{1,a}_{bal}=\zeta\chi_{2}(\eta,\zeta)1_{\zeta<\eta},\\
\eta\mathfrak{b}^{1,a}_{bal}+\zeta^4\mathfrak{c}^{1,a}_{bal}+(\zeta-\eta)\mathfrak{d}^{1,a}_{bal}=-\zeta\chi_{2}(\eta,\zeta)1_{\zeta<\eta},\\
\eta\mathfrak{a}^{1,a}_{bal}-(\zeta-\eta)^4\mathfrak{c}^{1,a}_{bal}-\zeta\mathfrak{d}^{1,a}_{bal}=-\zeta\chi_{2}(\eta,\zeta)1_{\zeta<\eta},\\
\zeta^4 \mathfrak{a}^{1,a}_{bal}-(\zeta-\eta)^4\mathfrak{b}^{1,a}_{bal}-\eta^4\mathfrak{d}^{1,a}_{bal}=\left(-\f32\zeta^4+\zeta^3\eta+\eta^2\zeta^2-\f32\zeta\eta^3\right)\chi_{2}(\eta,\zeta)1_{\zeta<\eta}.
\end{cases}
\end{equation*}
The expressions for the balanced symbols of holomorphic type are given by
\begin{align*}
\mathfrak{a}^{1,h}_{bal}=&\f{25\eta^7+100\eta^6\xi+200\eta^5\xi^2+242\eta^4\xi^3+190\eta^3\xi^4+80\eta^2\xi^5+5\xi^6+10\xi^7}{2\xi(25\xi^6+100\xi^5\eta+200\xi^4\eta^2+246\xi^3\eta^3+200\xi^2\eta^4+100\xi\eta^5+25\eta^6)}\chi_{2}(\xi,\eta),\\
\mathfrak{b}^{1,h}_{bal}=&\f{25\eta^7+125\eta^6\xi+300\eta^5\xi^2+442\eta^4\xi^3+442\eta^3\xi^4+300\eta^2\xi^5+125\eta\xi^6+25\xi^7}
{2\xi(25\xi^6+100\xi^5\eta+200\xi^4\eta^2+246\xi^3\eta^3+200\xi^2\eta^4+100\xi\eta^5+25\eta^6)}\chi_{2}(\xi,\eta),\\
\mathfrak{c}^{1,h}_{bal}=&\f{5(\eta^4+3\eta^3\xi+4\eta^2\xi^2+3\xi^3\eta+\xi^4)}
{\xi(25\xi^6+100\xi^5\eta+200\xi^4\eta^2+246\xi^3\eta^3+200\xi^2\eta^4+100\xi\eta^5+25\eta^6)}\chi_{2}(\xi,\eta),\\
\mathfrak{d}^{1,h}_{bal}=&-\f{5(\eta^3+2\eta^2\xi+2\eta\xi^2+\xi^3)(2\eta^4+5\eta^3\xi+10\eta^2\xi^2+10\eta\xi^3+5\xi^4)}
{2\xi(25\xi^6+100\xi^5\eta+200\xi^4\eta^2+246\xi^3\eta^3+200\xi^2\eta^4+100\xi\eta^5+25\eta^6)}\chi_{2}(\xi,\eta).
\end{align*}
The leading terms of these balanced symbols of holomorphic type are
\begin{equation*}
\mathfrak{a}^{1,h}_{bal} \approx\f{213}{448}\chi_{2}(\xi, \eta), \quad \mathfrak{b}^{1,h}_{bal} \approx \f{223}{224}\chi_{2}(\xi, \eta),  \quad \mathfrak{c}^{1,h}_{bal} \approx \f{15}{224}\eta^{-3}\chi_{2}(\xi, \eta), \quad \mathfrak{d}^{1,h}_{bal} \approx -\f{15}{28}\chi_{2}(\xi, \eta).   
\end{equation*}
The expressions for the balanced symbols of mixed type are given by
\begin{align*}
\mathfrak{a}^{1,a}_{bal}=&-\f{8\eta^6\zeta-24\eta^5\zeta^2+49\eta^4\zeta^3-58\eta^3\zeta^4+75\eta^2\zeta^5-50\eta\zeta^6+25\zeta^7}{2\eta(4\eta^6-12\eta^5\zeta+37\eta^4\zeta-54\eta^3\zeta^3+75\eta^2\zeta^4-50\eta\zeta^5+25\zeta^6)}\chi_{2}(\eta, \zeta)1_{\zeta<\eta},\\
\mathfrak{b}^{1,a}_{bal}=&-\f{-2\eta^6\zeta+11\eta^5\zeta^2-11\eta^4\zeta^3+17\eta^3\zeta^4+25\eta^2\zeta^5-25\eta\zeta^6+25\zeta^7}{2\eta(4\eta^6-12\eta^5\zeta+37\eta^4\zeta-54\eta^3\zeta^3+75\eta^2\zeta^4-50\eta\zeta^5+25\zeta^6)}\chi_{2}(\eta, \zeta)1_{\zeta<\eta},\\
\mathfrak{c}^{1,a}_{bal}=&-\f{5(\eta^2\zeta^2-\eta\zeta^3+\zeta^4)}{\eta(4\eta^6-12\eta^5\zeta+37\eta^4\zeta-54\eta^3\zeta^3+75\eta^2\zeta^4-50\eta\zeta^5+25\zeta^6)}\chi_{2}(\eta, \zeta)1_{\zeta<\eta},\\
\mathfrak{d}^{1,a}_{bal}=&\f{5(\eta^2\zeta-\eta\zeta^2+\zeta^3)(2\eta^4-3\eta^3\zeta+7\zeta^2\zeta^2-3\eta\zeta^3+2\zeta^4)}{2\eta(4\eta^6-12\eta^5\zeta+37\eta^4\zeta-54\eta^3\zeta^3+75\eta^2\zeta^4-50\eta\zeta^5+25\zeta^6)}\chi_{2}(\eta, \zeta)1_{\zeta<\eta}.   
\end{align*}
The leading terms of these balanced symbols of mixed type are
\begin{equation*}
\mathfrak{a}^{1,a}_{bal} \approx -\f12\chi_{2}(\eta, \zeta)1_{\zeta<\eta}, \quad \mathfrak{b}^{1,a}_{bal} \approx -\f45\chi_{2}(\eta, \zeta)1_{\zeta<\eta},  \quad \mathfrak{c}^{1,a}_{bal} \approx -\f{\eta^{-3}}{5}\chi_{2}(\eta, \zeta)1_{\zeta<\eta}, \quad \mathfrak{d}^{1,a}_{bal} \approx \f12\chi_{2}(\eta, \zeta)1_{\zeta<\eta}.   
\end{equation*}

Hence, we obtain that balanced normal form corrections $(w^1_{bal}, r^1_{bal})$ satisfy the estimate
\begin{equation*}
    \|(w^1_{bal}, r^1_{bal})\|_{\H^0} \lesssim \mathcal{A}_{0}\|(w,r)\|_{\H^0}.
\end{equation*}
Moreover, the cubic and higher terms produced by balanced normal form variables  $(w^1_{bal}, r^1_{bal})$ are perturbative.

\subsubsection{Cubic normal form analysis for $(\nP\mathcal{G}_1, \nP\mathcal{K}_1)$}
Finally, we construct cubic normal form variables $(w^1_c, r^1_c)$ so that $(w_{NF}^1, r_{NF}^1)$ solve \eqref{wOneNFPGKOne}.
Recall that quadratic normal form variables $(w^1_{lh}+w^1_{bal}, r^1_{lh}+ r^1_{bal})$ solve the system
\begin{equation*} 
\begin{cases}
T_{D_t}(w^1_{lh}+w^1_{bal})+T_{1-\bar{Y}}\p_\al (r^1_{lh}+r^1_{bal})+T_{T_{1-\bar{Y}}R_{\al}}(w^1_{lh}+w^1_{bal})= -\nP \mathcal{G}_1 + \mathcal{G}_1^{[3]} + G,\\
T_{D_t}(r^1_{lh}+r^1_{bal})-i\mathcal{L}_{para}(w^1_{lh}+w^1_{bal})=- \nP \mathcal{K}_1 + \mathcal{K}_1^{[3]} +K,
\end{cases}
\end{equation*}
where non-perturbative cubic terms $( \mathcal{G}_1^{[3]},  \mathcal{K}_1^{[3]})$ are given by
\begin{align*}
 &\mathcal{G}_1^{[3]} = - \f12 T_{w}T_{\bw_\al}T_{J^{-1}}R + \f12 T_{w}T_{\bar{\bw}_\al}T_{(1-\bar{Y})^2}R - \f12 T_w T_R T_{J^{-1}}\bw_\al - \f12 T_w T_{\bar{R}} T_{(1-Y)^2}\bw_\al,\\
  &\mathcal{K}_1^{[3]} = \f{25}{4}i T_{w} T_{\bw_\al} T_{J^{-\f32}(1-Y)^3} \p_\al^3 \bw +3i T_{w} T_{\bar{\bw}_\al} T_{J^{-\f52}(1-Y)}  \p_\al^3 \bw - \f12 T_w T_R T_{1-\bar{Y}}R_\al - \f12 T_w T_{\bar{R}} T_{1-Y}R_\al.
\end{align*}
The terms
\begin{equation*}
    \left(- \f12 T_{w}T_{\bw_\al}T_{J^{-1}}R- \f12 T_w T_R T_{J^{-1}}\bw_\al,\,  \f{25}{4}i T_{w} T_{\bw_\al} T_{J^{-\f32}(1-Y)^3} \p_\al^3 \bw - \f12 T_w T_R T_{1-\bar{Y}}R_\al  \right)
\end{equation*}
are non-resonant. To eliminate these terms, we consider the system of of equations for auxiliary normal form transformations $(\bw^{1,1}_c, R^{1,1}_c)$,
\begin{equation*}  
\begin{cases}
\p_t \bw^{1,1}_c+T_{1-\bar{Y}}\p_\al R^{1,1}_c=\f12 T_{\bw_\al}T_{J^{-1}}R+\f12 T_R T_{J^{-1}}\bw_\al  + \text{cubic and higher terms},\\
\p_t R^{1,1}_c-iT_{J^{-\f32}(1-Y)}\p_\al^4\bw^{1,1}_c= -\f{25}{4}i  T_{\bw_\al} T_{J^{-\f32}(1-Y)^3} \p_\al^3 \bw +\f12  T_R T_{1-\bar{Y}}R_\al + \text{cubic and higher terms}.
\end{cases}
\end{equation*}
Auxiliary normal form transformations $(\bw^{1,1}_c, R^{1,1}_c)$ are chosen as the sum of low-high paradifferential bilinear forms of the following type:
\begin{align*} 
\bw^{1,1}_{c} = &B^{1,1}_{c}\left(\bw, T_{(1-Y)^2}\bw \right) + C^{1,1}_{c}\left(R, T_{J^\f12(1+\bw)^2}R \right),\\
R^{1,1}_{c} =& A^{1,1}_{c} \left(R, T_{1-Y}\bw \right) + D^{1,1}_{c}\left(\bw, T_{1-Y}R\right).
\end{align*}

One can verify that
$(\bw^{1,1}_{c}, R^{1,1}_{c})$ satisfies
\begin{align*}
 &\p_t \bw^{1,1}_c+T_{1-\bar{Y}}\p_\al R^{1,1}_c  + \text{cubic and higher terms} \\
=& T_{J^{-1}}\partial_\alpha A^{1,1}_c(R, \bw) -T_{J^{-1}}B^{1,1}_{c}(R_\alpha, \bw) +  T_{J^{-1}}C^{1,1}_{c}(R, i\p_\al^4\bw) \\
 &-T_{J^{-1}}B^{1,1}_{c}(\bw, R_\alpha)+ T_{J^{-1}}C^{1,1}_{c}(i\p_\al^4\bw, R) + T_{J^{-1}}\partial_\alpha D^{1,1}_{c}(\bw, R),\\
&\p_t R^{1,1}_c-iT_{J^{-\f32}(1-Y)}\p_\al^4\bw^{1,1}_c+ \text{cubic and higher  terms} \\
=& -T_{1-\bar{Y}}A^{1,1}_{c}(R, R_\alpha) - iT_{1-\bar{Y}}\partial_\alpha^4 C^{1,1}_{c}(R, R) - T_{1-\bar{Y}}D^{1,1}_{c}(R_\alpha, R)\\
&+iT_{J^{-\f32}(1-Y)^2}A^{1,1}_{c}(\p_\al^4\bw, \bw) -iT_{J^{-\f32}(1-Y)^2}\partial_\alpha^4 B^{1,1}_{c}(\bw, \bw) +iT_{J^{-\f32}(1-Y)^2}D^{1,1}_{c}(\bw, \p_\al^4\bw).
\end{align*}
Here, we write $\mathfrak{a}^{1,1}_{c}(\xi, \eta)$ for the symbol of $A^{1,1}_{c}(R, T_{1-Y}\bw)$, and similarly for other low-high bilinear forms.
To match the low-high part of paradifferential source terms, paradifferential symbols solve the following algebraic system: 
\begin{equation*}
\begin{cases}
(\xi + \eta)\mathfrak{a}^{1,1}_{c}-\xi\mathfrak{b}^{1,1}_{c}+\eta^4\mathfrak{c}^{1,1}_{c}= \f12 \eta\chi_1(\xi, \eta),\\
\eta\mathfrak{b}^{1,1}_{c}-\xi^4\mathfrak{c}^{1,1}_{c}-(\xi+\eta)\mathfrak{d}^{1,1}_{c}=-\f12 \xi\chi_{1}(\xi,\eta),\\
\eta\mathfrak{a}^{1,1}_{c}+(\xi+\eta)^4\mathfrak{c}^{1,1}_{c}+\xi\mathfrak{d}^{1,1}_{c}=-\f12\eta\chi_{1}(\xi,\eta),\\
\xi^4 \mathfrak{a}^{1,1}_{c}-(\xi+\eta)^4\mathfrak{b}^{1,1}_{c}+\eta^4\mathfrak{d}^{1,1}_{c}=-\f{25}{4}\xi\eta^3\chi_{1}(\xi,\eta),
\end{cases}
\end{equation*}
where  the symbol $\chi_{1}(\xi, \eta)$ is defined in \eqref{ChiOnelh} to select the low-high frequencies.
The solution of above system is 
\begin{align*}
\mathfrak{a}^{1,1}_{c}=&\f{10\xi^7+60\xi^6\eta+289\xi^5\eta^2+520\xi^4\eta^3+534\xi^3\eta^4+321\xi^2\eta^5+134\xi \eta^6+20\eta^7}{4\xi(25\xi^6+100\xi^5\eta+200\xi^4\eta^2+246\xi^3\eta^3+200\xi^2\eta^4+100\xi\eta^5+25\eta^6)}\chi_{1}(\xi,\eta),\\
\mathfrak{b}^{1,1}_{c}=&\f{10\xi^6+20\xi^5\eta+149\xi^4\eta^2+399\xi^3\eta^3+528\xi^2\eta^4+395\xi\eta^5+135\eta^6}
{4(25\xi^6+100\xi^5\eta+200\xi^4\eta^2+246\xi^3\eta^3+200\xi^2\eta^4+100\xi\eta^5+25\eta^6)}\chi_{1}(\xi,\eta),\\
\mathfrak{c}^{1,1}_{c}=&-\f{5\xi^5+17\xi^4\eta+30\xi^3\eta^2+60\xi^2\eta^3+52\xi\eta^4+10\eta^5}
{2\xi\eta(25\xi^6+100\xi^5\eta+200\xi^4\eta^2+246\xi^3\eta^3+200\xi^2\eta^4+100\xi\eta^5+25\eta^6)}\chi_{1}(\xi,\eta),\\
\mathfrak{d}^{1,1}_{c}=&\f{10\xi^7+74\xi^6\eta + 196\xi^5\eta^2+344\xi^4\eta^3+401\xi^3\eta^4+418\xi^2\eta^5+310\xi\eta^6+135\eta^7}
{4\eta(25\xi^6+100\xi^5\eta+200\xi^4\eta^2+246\xi^3\eta^3+200\xi^2\eta^4+100\xi\eta^5+25\eta^6)}\chi_{1}(\xi,\eta).
\end{align*}
Therefore, we obtain
\begin{align*}
    &w^{1,1}_c = \f{27}{20}T_{\bw}T_{(1-Y)^2}\bw +\f{i}{5} T_{\p_\al^{-1}R}  T_{J^{\frac{1}{2}}(1+\bw)^2}  \p_\al^{-2}R + \text{lower order terms},\\
    &r^{1,1}_c = \f{1}{5} T_{\p_\al^{-1}R}T_{1-Y}\p_\al\bw+\f{27}{20} T_{\bw}T_{1-Y}  R + \text{lower order terms}.
\end{align*}
We set $(w^{1,1}_c, r^{1,1}_c) = \big(T_{w}\bw^{1,1}_c, T_{w}T_{1-Y}R^{1,1}_c \big)$. Then
\begin{equation*}
\|(w^{1,1}_c, r^{1,1}_c)\|_{\H^0}  \lesssim\mathcal{A}_0 \mathcal{A}_\f32\|(w,r)\|^2_{\H^0},
\end{equation*}
and this pair solves the system
\begin{align*} 
&T_{D_t}w^{1,1}_c+T_{1-\bar{Y}}\p_\al r^{1,1}_c+T_{T_{1-\bar{Y}}R_{\al}}w^{1,1}_c\\
=& \f12 T_{\bw_\al}T_{J^{-1}}R+\f12 T_R T_{J^{-1}}\bw_\al+\f{27}{20}T_{w_t}T_{\bw}T_{(1-Y)^2}\bw \\
&+\f{i}{5} T_{w_t}T_{\p_\al^{-1}R}  T_{J^{\frac{1}{2}}(1+\bw)^2}  \p_\al^{-2}R  + \text{quartic and higher terms} + G\\
=& \f12 T_{\bw_\al}T_{J^{-1}}R+\f12 T_R T_{J^{-1}}\bw_\al + G,\\
&T_{D_t}r^{0,2}_c-i\mathcal{L}_{para}w^{0,2}_c  \\
=& -\f{25}{4}i  T_{\bw_\al} T_{J^{-\f32}(1-Y)^3} \p_\al^3 \bw +\f12  T_R T_{1-\bar{Y}}R_\al+\f{1}{5} T_{w_t}T_{\p_\al^{-1}R}T_{(1-Y)^2}\p_\al\bw+\f{27}{20} T_{w_t}T_{\bw}T_{(1-Y)^2}  R \\
&+\text{quartic and higher terms} + K\\
=& -\f{25}{4}i  T_{\bw_\al} T_{J^{-\f32}(1-Y)^3} \p_\al^3 \bw +\f12  T_R T_{1-\bar{Y}}R_\al + K.
\end{align*}

To eliminate other terms in $(\mathcal{G}_1^{[3]}, \mathcal{K}_1^{[3]})$ that may have resonances, we consider the system of of equations for auxiliary normal form transformations $(\bw^{1,2}_c, R^{1,2}_c)$,
\begin{equation*}  
\begin{cases}
\p_t \bw^{1,2}_c+T_{1-\bar{Y}}\p_\al R^{1,2}_c=  -\f12 T_{\bar{\bw}_\al}T_{(1-\bar{Y})^2}R+ \f12  T_{\bar{R}} T_{(1-Y)^2}\bw_\al + \text{cubic and higher terms},\\
\p_t R^{1,2}_c-iT_{J^{-\f32}(1-Y)}\p_\al^4\bw^{1,2}_c= -3i  T_{\bar{\bw}_\al} T_{J^{-\f52}(1-Y)}  \p_\al^3 \bw  + \f12  T_{\bar{R}} T_{1-Y}R_\al  + \text{cubic and higher terms}.
\end{cases}
\end{equation*}
As above, the auxiliary normal form transformations $(\bw^1_c, R^1_c)$  are chosen as the sum of low-high paradifferential bilinear forms of the following type:
\begin{equation*} 
\bw^{1,2}_{c} = B^{1,2}_{c}\left(\bar{\bw}, T_{J^{-1}}\bw \right) + C^{1,2}_{c}\left(\bar{R}, T_{J^{\f32}}R \right),\quad R^{1,2}_{c} = A^{1,2}_{c} \left(\bar{R}, T_{(1-Y)^2(1+\bar{\bw})}\bw \right) + D^{1,2}_{c}\left(\bar{\bw}, T_{1-\bar{Y}}R\right).
\end{equation*}

For there bilinear forms, we compute
\begin{align*}
&\partial_t \bw^{1,2}_c+T_{1-\bar{Y}} \partial_\alpha R^{1,2}_c + \text{cubic and higher terms} \\
=& T_{(1-Y)^2}\partial_\alpha A^{1,2}_c(\bar{R}, \bw) -T_{(1-Y)^2}B^{1,2}_c(\bar{R}_\alpha, \bw) +  T_{(1-Y)^2}C^1_c(\bar{R}, i\p_\al^4\bw) \\
 &-T_{(1-\bar{Y})^2}B^{1,2}_c(\bar{\bw}, R_\alpha)- T_{(1-\bar{Y})^2}C^{1,2}_c(i\p_\al^4\bar{\bw}, R) + T_{(1-\bar{Y})^2}\partial_\alpha D^{1,2}_c(\bar{\bw}, R),\\
&\partial_t R^{1,2}_c  -iT_{J^{-\frac{3}{2}}(1-Y)}\p_\al^4 \bw^1_c+ \text{cubic and higher  terms} \\
=& -T_{1-Y}A^{1,2}_c(\bar{R}, R_\alpha) - iT_{1-Y}\partial_\alpha^4 C^{1,2}_c(\bar{R}, R) - T_{1-Y}D^{1,2}_c(\bar{R}_\alpha, R)\\
&- iT_{J^{-\frac{5}{2}}(1-Y)}A^{1,2}_c(\p_\al^4\bar{\bw}, \bw) -iT_{J^{-\frac{5}{2}}(1-Y)}\partial_\alpha^4 B^{1,2}_c(\bar{\bw}, \bw) +iT_{J^{-\frac{5}{2}}(1-Y)}D^{1,2}_c(\bar{\bw}, \p_\al^4\bw).
\end{align*}
Here, as above, we write $\mathfrak{a}^{1,2}_c(\eta, \zeta)$ for the symbol of  $A^{1,2}_{c} \left(\bar{R}, T_{(1-Y)^2(1+\bar{\bw})}\bw \right)$ and similarly for other low-high bilinear forms.
To match the low-high part of paradifferential source terms, paradifferential symbols solve the following algebraic system: 
\begin{equation*}
\begin{cases}
(\zeta-\eta)\mathfrak{a}^{1,2}_c+\eta\mathfrak{b}^1_c+\zeta^4\mathfrak{c}^{1,2}_c=\f12\zeta\chi_{1}(\eta,\zeta),\\
\zeta\mathfrak{b}^{1,2}_c+\eta^4\mathfrak{c}^{1,2}_c-(\zeta-\eta)\mathfrak{d}^{1,2}_c=-\f12\eta\chi_{1}(\eta,\zeta),\\
\zeta\mathfrak{a}^{1,2}_c+(\zeta-\eta)^4\mathfrak{c}^{1,2}_c-\eta\mathfrak{d}^{1,2}_c= -\f12\zeta\chi_{1}(\eta,\zeta),\\
\eta^4 \mathfrak{a}^{1,2}_c+(\zeta-\eta)^4\mathfrak{b}^{1,2}_c-\zeta^4\mathfrak{d}^{1,2}_c= -3\eta \zeta^3\chi_{1}(\eta,\zeta).
\end{cases}
\end{equation*}
The solutions of the above system are given by
\begin{align*}
\mathfrak{a}^{1,2}_c=&\f{-2\eta^7\zeta + 11\eta^6\zeta^2-43\eta^5\zeta^3+73\eta^4\zeta^4-106\eta^3\zeta^5+72\eta^2\zeta^6-40\eta\zeta^7+10\zeta^8}{\eta(\eta - \zeta) (8\eta^6-24\eta^5\zeta+74\eta^4\zeta^2-108\eta^3\zeta^3+150\eta^2\zeta^4-100\eta\zeta^5+50\zeta^6)}\chi_{1}(\eta, \zeta),\\
\mathfrak{b}^{1,2}_c=&\f{2\eta^5\zeta -\eta^4\zeta^2-6\eta^3\zeta^3+21\eta^2\zeta^4-30\eta\zeta^5+25\zeta^6}{8\eta^6-24\eta^5\zeta+74\eta^4\zeta^2-108\eta^3\zeta^3+150\eta^2\zeta^4-100\eta\zeta^5+50\zeta^6}\chi_{1}(\eta, \zeta),\\
\mathfrak{c}^{1,2}_c=&\f{-2\eta^4 - \eta^3\zeta-3\eta^2\zeta^2-15\eta\zeta^3+10\zeta^4}{\eta(8\eta^6-24\eta^5\zeta+74\eta^4\zeta^2-108\eta^3\zeta^3+150\eta^2\zeta^4-100\eta\zeta^5+50\zeta^6)}\chi_{1}(\eta, \zeta),\\
\mathfrak{d}^{1,2}_c=&-\f{2\eta^7-13\eta^6\zeta+36\eta^5\zeta^2-70\eta^4\zeta^3+79\eta^3\zeta^4-29\eta^2\zeta^5-5\eta\zeta^6+25\zeta^7}{(\eta - \zeta)(8\eta^6-24\eta^5\zeta+74\eta^4\zeta^2-108\eta^3\zeta^3+150\eta^2\zeta^4-100\eta\zeta^5+50\zeta^6)}\chi_{1}(\eta, \zeta).
\end{align*}
Therefore, we have
\begin{align*}
    &w^{1,2}_c = \f12 T_{\bar{\bw}}T_{J^{-1}}\bw + \f{i}{5} T_{\p_\al^{-1}\bar{R}}  T_{J^{\frac{3}{2}}}  \p_\al^{-2}R + \text{lower order terms},\\
    &r^{1,2}_c = -\f15 T_{\p_\al^{-1}\bar{R}}T_{(1-Y)^2(1+\bar{\bw})}\p_\al\bw+\f12 T_{\bar{\bw}}T_{1-\bar{Y}}  R + \text{lower order terms}.
\end{align*}
Setting cubic  normal form variables $(w^{1,2}_c, r^{1,2}_c) =(T_ww^1_c, T_w r^1_c)$, we obtain
\begin{equation*}
\|(w^{1,2}_c, r^{1,2}_c)\|_{\H^0}  \lesssim\mathcal{A}_0 \mathcal{A}_\f32\|(w,r)\|^2_{\H^0}.
\end{equation*}
This pair solves the system 
\begin{align*} 
&T_{D_t}w^{1,2}_c+T_{1-\bar{Y}}\p_\al r^{1,2}_c+T_{T_{1-\bar{Y}}R_{\al}}w^{1,2}_c\\
=& -\f12 T_wT_{\bar{\bw}_\al}T_{(1-\bar{Y})^2}R+ \f12  T_wT_{\bar{R}} T_{(1-Y)^2}\bw_\al + \f12 T_{w_t}T_{\bar{\bw}}T_{J^{-1}}\bw \\
&+ \f{i}{5} T_{w_t}T_{\p_\al^{-1}\bar{R}}  T_{J^{\frac{3}{2}}}  \p_\al^{-2}R + G + \text{quartic and higher terms}\\
=& -\f12 T_wT_{\bar{\bw}_\al}T_{(1-\bar{Y})^2}R+ \f12  T_wT_{\bar{R}} T_{(1-Y)^2}\bw_\al + G,\\
&T_{D_t}r^{1,2}_c-i\mathcal{L}_{para}w^{1,2}_c \\
=&-3i  T_wT_{\bar{\bw}_\al} T_{J^{-\f52}(1-Y)}  \p_\al^3 \bw  + \f12  T_wT_{\bar{R}} T_{1-Y}R_\al -\f15 T_{w_t}T_{\p_\al^{-1}\bar{R}}T_{(1-Y)^2(1+\bar{\bw})}\p_\al\bw\\
&+T_{w_t}\f12 T_{\bar{\bw}}T_{1-\bar{Y}}  R +K+ \text{quartic and higher terms}\\
=& -3i  T_wT_{\bar{\bw}_\al} T_{J^{-\f52}(1-Y)}  \p_\al^3 \bw  + \f12  T_wT_{\bar{R}} T_{1-Y}R_\al + K.
\end{align*}

As a consequence of above computations,  the cubic normal form variables $(w^1_c, r^1_c): = (w^{1,1}_c, r^{1,1}_c) + (w^{1,2}_c, r^{1,2}_c)$ solve the system 
\begin{equation*} 
\begin{cases}
T_{D_t}w^1_c+T_{1-\bar{Y}}\p_\al r^1_c+T_{T_{1-\bar{Y}}R_{\al}}w^1_c= - \mathcal{G}_1^{[3]} + G,\\
T_{D_t}r^1_c-i\mathcal{L}_{para}w^1_c=- \mathcal{K}_1^{[3]} +K.
\end{cases}
\end{equation*}

To conclude this subsection, by choosing normal form variables $(w_{NF}^1, r_{NF}^1)$ as in \eqref{wrNFOneDef}, they satisfy \eqref{wrNFHOneBound} and solve the system \eqref{wOneNFPGKOne}.

\section{Energy estimate for the hydroelastic waves}  \label{s:Energy}
In this section, we derive the modified energy estimate for the differentiated two-dimensional hydroelastic wave system \eqref{HF19}.
Rather than constructing the modified energy $E_s(\mathbf{W},R)$ directly, we proceed in several steps.
First, we rewrite \eqref{HF19} in the paradifferential form in Section \ref{s:ParaReductionHydro}, moving non-perturbative balanced source terms to the right-hand side.
Then, in Section \ref{s:NormalHydroWaves}, we construct balanced normal form variables $(\bw_{NF}, R_{NF})$ so that the paradifferential equations of $(\bw_{NF}, R_{NF})$ no longer have balanced perturbative source terms.
Next, we demonstrate how to reduce Theorem \ref{t:MainEnergyEst} to  the modified energy estimate of the homogeneous paradifferential system \eqref{ParadifferentialFlow} in Section \ref{s:ParaFlowReduct}.
Finally, in Section \ref{s:HsHomogeneousSys}, we  prove the modified energy estimate of  \eqref{ParadifferentialFlow} by partially exploiting Proposition \ref{t:wellposedflow}, which ultimately yields the modified energy estimate for $(\bw, R)$. 

Throughout this section, we denote by $(G, K)$ perturbative source terms satisfying
\begin{equation}  \label{GKHsPerturb}
    \|(G,K)\|_{\H^s} \lesssim_{\mathcal{A}_0} \ASSharp \|(\bw, R)\|_{\H^s}, \quad s > 0,
\end{equation}
or $(G_s, K_s)$ that satisfy \eqref{GsHsSource}.

\subsection{Paradifferential reduction of hydroelastic waves} \label{s:ParaReductionHydro}
First, we rewrite the differentiated two-dimensional hydroelastic wave system \eqref{HF19} as a system of paradifferential equations, placing balanced paraproducts and perturbative terms on the right-hand side.

Recall that $(\bw, R)$ solve the system 
 \begin{equation*}
 \begin{cases}
&D_t\mathbf{W}+(1-\bar{Y})(1+\mathbf{W})R_{\al}=(1+\mathbf{W})M,\\
&D_tR+ia(1-Y) -i(1-Y) \nP \p_\al\left\{J^{-\f12}\p_\al\left[J^{-\f12}\p_\al\left(\mathbf{W}_{\al}J^{-\f12}(1-Y)\right)\right]\right\}\\
&
-\f12i(1-Y) \nP \p_\al\left[\mathbf{W}^3_{\al}J^{-\f32}(1-Y)^3-3\mathbf{W}_{\al}|\mathbf{W}_{\al}|^2 J^{-\f52}(1-Y)\right]\\
&=-i(1-Y)\nP \p_{\al}\left\{J^{-\f12}\p_\al\left[J^{-\f12}\p_\al\left(\bar{\mathbf{W}}_{\al}
J^{-\f12}(1-\bar{Y})\right)\right]\right\}\\
&-\f12i(1-Y) \nP \p_{\al}\left[\bar{\mathbf{W}}^3_{\al}J^{-\f32}(1-\bar{Y})^3-3\bar{\mathbf{W}}_{\al}|\mathbf{W}_{\al}|^2 J^{-\f52}(1-\bar{Y})\right].
\end{cases}
\end{equation*}

Following Section 6 in \cite{AIT}, $\bw$ solves the paradifferential equation
\begin{equation} \label{WParaEqn}
\begin{aligned}
&T_{D_t}\bw + T_{b_\alpha}\bw + \partial_\alpha T_{1-\bar{Y}}T_{1+\bw} R \\
=& - T_{1- \bar{Y}}\p_\al \Pi(\bw, R)- T_{1-Y}\partial_\alpha \nP \Pi( \bar{R}, \bw) -T_{(1-\bar{Y})^2(1+\bw)}\partial_\alpha \nP\Pi(\bar{\bw}, R) + G.
\end{aligned}    
\end{equation}
For the first two terms in the second equation of the hydroelastic waves, we apply the projection $\nP$ and write
\begin{align*}
 &\nP D_tR + \nP(ia(1-Y))\\
 =&   T_{D_t} R + T_{R_\alpha} \nP b + \nP\Pi(R_\alpha, b)+ iT_{1-Y} \nP a -iT_a Y -i\nP \Pi(a,Y)\\
          =&    T_{D_t} R + T_{R_\alpha} T_{1-\bar{Y}}R +\nP\Pi(R_\alpha, T_{1-\bar{Y}}R + T_{1-Y}\bar{R})+ T_{1-Y}T_{\bar{R}_\alpha}R +T_{1-Y}\nP\Pi(\bar{R}_\alpha, R) +K\\
 =&  T_{D_t} R + T_{b_\al} R + T_{1-\bar{Y}}\Pi(R_\alpha, R) + T_{1-Y}\nP \Pi( R_\al, \bar{R}) +T_{1-Y}\nP\Pi(\bar{R}_\alpha, R) + K.
\end{align*}

It remains to consider the elastic terms in the equation of $R$.
For a fixed real constant $\beta$,
\begin{align*}
\p_{\al}J^{-\beta}
=&-\beta J^{-\beta}(1-Y)\bw_{\al}-\beta J^{-\beta}(1-\bar{Y})\bar{\bw}_{\al},\quad \p_{\al}Y=(1-Y)^2\bw_{\al},\\
\p_{\al}(J^{-\beta}(1-Y))
=&-(\beta+1)J^{-\beta}(1-Y)^2\bw_{\al}-\beta J^{-\beta-1}
\bar{\bw}_{\al},\\
\p^2_{\al}(J^{-\beta}(1-Y))
=&-(\beta+1)J^{-\beta}(1-Y)^2\bw_{\al\al}-\beta J^{-\beta-1}
\bar{\bw}_{\al\al}\\
&+(\beta+2)(\beta+1)J^{-\beta}(1-Y)^3\bw_{\al}^2+2\beta(\beta+1)J^{-\beta-1}(1-Y)|\bw_{\al}|^2\\
&+\beta(\beta+1)J^{-\beta-1}(1-\bar{Y})\bar{\bw}_{\al}^2.
\end{align*}
From these, we derive the following identities:
\begin{align*}
&J^{-1}\p_{\al}(J^{-\f12}(1-Y))+J^{-\f12}\p_{\al}(J^{-1}(1-Y))+\p_{\al}(J^{-\f32}(1-Y))\\ =&-6J^{-\f32}(1-Y)^2\bw_{\al} -3J^{-\f52}
\bar{\bw}_{\al}, \\
&J^{-\f12}\p_{\al}\left(J^{-\f12}\p_{\al}\left(J^{-\f12}(1-Y)\right)\right)\\
=&-\f32 J^{-\f32}(1-Y)^2\bw_{\al\al}-\f12 J^{-\f52}
\bar{\bw}_{\al\al}+\f{9}{2}J^{-\f32}(1-Y)^3\bw_{\al}^2\\
&+J^{-\f52}(1-\bar{Y})\bar{\bw}_{\al}^2
+\f52J^{-\f52}(1-Y)|\bar{\bw}_{\al}|^2,\\
&\p_{\al}(J^{-1}\p_{\al}(J^{-\f12}(1-Y)))\\
=&-\f32 J^{-\f32}(1-Y)^2\bw_{\al\al}-\f12 J^{-\f52}
\bar{\bw}_{\al\al}+\f{21}{4}J^{-\f32}(1-Y)^3\bw_{\al}^2\\
&+\f54J^{-\f52}(1-\bar{Y})\bar{\bw}_{\al}^2
+\f{7}{2}J^{-\f52}(1-Y)|\bar{\bw}_{\al}|^2,\\
&\p_{\al}(J^{-\f12}\p_{\al}(J^{-1}(1-Y)))\\
=&-2 J^{-\f32}(1-Y)^2\bw_{\al\al}- J^{-\f52}
\bar{\bw}_{\al\al}+7J^{-\f32}(1-Y)^3\bw_{\al}^2\\
&+\f52J^{-\f52}(1-\bar{Y})\bar{\bw}_{\al}^2
+\f{11}{2}J^{-\f52}(1-Y)|\bar{\bw}_{\al}|^2.
\end{align*}
Then, we have
\begin{align*}
&J^{-\f12}\p_{\al}(J^{-\f12}\p_{\al}(J^{-\f12}(1-Y)))+\p_{\al}(J^{-1}\p_{\al}(J^{-\f12}(1-Y)))+\p_{\al}(J^{-\f12}\p_{\al}(J^{-1}(1-Y)))\\ 
=&-5J^{-\f32}(1-Y)^2\bw_{\al\al}-2 J^{-\f52}
\bar{\bw}_{\al\al}+\f{67}{4}J^{-\f32}(1-Y)^3\bw_{\al}^2\\
&+\f{19}{4}J^{-\f52}(1-\bar{Y})\bar{\bw}_{\al}^2
+\f{23}{2}J^{-\f52}(1-Y)|\bar{\bw}_{\al}|^2.
\end{align*}
Direct computation yields
\begin{align*}
&J^{-\f12}\p_{\al}\left(J^{-\f12}\p_{\al}\left(J^{-\f12}(1-Y)\right)\right)\\
=&-\f32 J^{-\f32}(1-Y)^2\bw_{\al\al}-\f12 J^{-\f52}
\bar{\bw}_{\al\al}+\f92J^{-\f32}(1-Y)^3\bw_{\al}^2\\
&+J^{-\f52}(1-\bar{Y})\bar{\bw}_{\al}^2
+\f52J^{-\f52}(1-Y)|\bar{\bw}_{\al}|^2,
\end{align*}
we then have
\begin{align*}
&\p_{\al}\left(J^{-\f12}\p_{\al}\left(J^{-\f12}\p_{\al}\left(J^{-\f12}(1-Y)\right)\right)\right)\\
=&-\f32 J^{-\f32}(1-Y)^2\p^3_{\al}\bw-\f12 J^{-\f52}
\p^3_{\al}\bar{\bw}+\f{57}{4}J^{-\f32}(1-Y)^3\bw_{\al}\bw_{\al\al} +\f{13}{4}J^{-\f52}(1-\bar{Y})\bar{\bw}_{\al}\bar{\bw}_{\al\al}\\
&+\f{19}{4}J^{-\f52}(1-Y)\bar{\bw}_{\al}\bw_{\al\al}+\f{15}{4}\bw_\al\bar{\bw}_{\al\al}-\f{81}{4}J^{-\f32}(1-Y)^4\bw^3_{\al}\\
&-\f{31}{2}J^{-\f52}(1-Y)^2\bw_{\al}|\bw_{\al}|^2-\f{35}{4}J^{-\f72}\bar{\bw}_{\al}|\bw_{\al}|^2-\f{7}{2}J^{-\f52}(1-\bar{Y})^2\bar{\bw}^3_{\al}.
\end{align*} 
As a consequence,
\begin{align*}
&\p_\al\left\{J^{-\f12}\p_\al\left[J^{-\f12}\p_\al\left(\mathbf{W}_{\al}J^{-\f12}(1-Y)\right)\right]\right\}\\
=&J^{-\f32}(1-Y)\p^4_{\al}\bw+\left(J^{-1}\p_{\al}(J^{-\f12}(1-Y))+J^{-\f12}\p_{\al}(J^{-1}(1-Y))+\p_{\al}(J^{-\f32}(1-Y))\right)\p^3_{\al}\bw\\
&+\left(J^{-\f12}\p_{\al}(J^{-\f12}\p_{\al}(J^{-\f12}(1-Y)))+\p_{\al}(J^{-1}\p_{\al}(J^{-\f12}(1-Y)))+\p_{\al}(J^{-\f12}\p_{\al}(J^{-1}(1-Y)))\right)\bw_{\al\al}\\
&+\p_{\al}\left(J^{-\f12}\p_{\al}\left(J^{-\f12}\p_{\al}\left(J^{-\f12}(1-Y)\right)\right)\right)\bw_{\al}\\
=&J^{-\f32}(1-Y)\p^4_{\al}\bw-6J^{-\f32}(1-Y)^2\bw_{\al}\p^3_{\al}\bw-3 J^{-\f52}\bar{\bw}_{\al}\p^3_{\al}\bw\\
&-5J^{-\f32}(1-Y)^2\bw^2_{\al\al}-2 J^{-\f52}
|\bw_{\al\al}|^2+\f{67}{4}J^{-\f32}(1-Y)^3\bw_{\al}^2\bw_{\al\al}\\
&+\f{19}{4}J^{-\f52}(1-\bar{Y})\bar{\bw}_{\al}^2\bw_{\al\al}
+\f{23}{2}J^{-\f52}(1-Y)|\bar{\bw}_{\al}|^2\bw_{\al\al}-\f32 J^{-\f32}(1-Y)^2\bw_{\al}\p^3_{\al}\bw\\
&-\f12 J^{-\f52}\bw_{\al}\p^3_{\al}\bar{\bw}+\f{57}{4}J^{-\f32}(1-Y)^3\bw^2_{\al}\bw_{\al\al}+\f{13}{4}J^{-\f52}(1-\bar{Y})|\bw_{\al}|^2\bar{\bw}_{\al\al}\\
&+\f{19}{4}J^{-\f52}(1-Y)|\bw_{\al}|^2\bw_{\al\al}+\f{15}{4}J^{-\f52}(1-Y)\bw^2_\al\bar{\bw}_{\al\al}+K\\
=&J^{-\f32}(1-Y)\p^4_{\al}\bw-\f{15}{2}J^{-\f32}(1-Y)^2\bw_{\al}\p^3_{\al}\bw-3 J^{-\f52}\bar{\bw}_{\al}\p^3_{\al}\bw-\f12 J^{-\f52}
\bw_{\al}\p^3_{\al}\bar{\bw}\\
&-5J^{-\f32}(1-Y)^2\bw^2_{\al\al}-2 J^{-\f52}
|\bw_{\al\al}|^2+31J^{-\f32}(1-Y)^3\bw_{\al}^2\bw_{\al\al}+\f{13}{4}J^{-\f52}(1-\bar{Y})|\bw_{\al}|^2\bar{\bw}_{\al\al}\\
&+\f{65}{4}J^{-\f52}(1-Y)|\bw_{\al}|^2\bw_{\al\al}+\f{15}{4}J^{-\f52}(1-Y)\bw^2_\al\bar{\bw}_{\al\al}+\f{19}{4}J^{-\f52}(1-\bar{Y})\bar{\bw}_{\al}^2\bw_{\al\al}+K.
\end{align*}
We derive
\begin{align*}
&\p_\al\left[\mathbf{W}^3_{\al}J^{-\f32}(1-Y)^3-3\mathbf{W}_{\al}|\mathbf{W}_{\al}|^2 J^{-\f52}(1-Y)\right]\\
=&3J^{-\f32}(1-Y)^3\bw_{\al}^2\bw_{\al\al}-6J^{-\f52}(1-Y)|\bw_{\al}|^2\bw_{\al\al}-3J^{-\f52}(1-Y)\bw_{\al}^2\bar{\bw}_{\al\al}+K.
\end{align*}
Therefore, we obtain the expression for quantity $A$ defined below
\begin{align*}
A:=&\p_\al\left\{J^{-\f12}\p_\al\left[J^{-\f12}\p_\al\left(\mathbf{W}_{\al}J^{-\f12}(1-Y)\right)\right]\right\}+\f12\p_\al\left[\mathbf{W}^3_{\al}J^{-\f32}(1-Y)^3-3\mathbf{W}_{\al}|\mathbf{W}_{\al}|^2 J^{-\f52}(1-Y)\right]\\
=&J^{-\f32}(1-Y)\p^4_{\al}\bw-\f{15}{2}J^{-\f32}(1-Y)^2\bw_{\al}\p^3_{\al}\bw-3 J^{-\f52}\bar{\bw}_{\al}\p^3_{\al}\bw-\f12 J^{-\f52}
\bw_{\al}\p^3_{\al}\bar{\bw}\\
&-5J^{-\f32}(1-Y)^2\bw^2_{\al\al}-2 J^{-\f52}
|\bw_{\al\al}|^2+\f{65}{2}J^{-\f32}(1-Y)^3\bw_{\al}^2\bw_{\al\al}+\f{13}{4}J^{-\f52}(1-\bar{Y})|\bw_{\al}|^2\bar{\bw}_{\al\al}\\
&+\f{53}{4}J^{-\f52}(1-Y)|\bw_{\al}|^2\bw_{\al\al}+\f{9}{4}J^{-\f52}(1-Y)\bw^2_\al\bar{\bw}_{\al\al}+\f{19}{4}J^{-\f52}(1-\bar{Y})\bar{\bw}_{\al}^2\bw_{\al\al}+K.
\end{align*}

Using the notation of $A$, the elastic terms are simply
\begin{equation*}
    -i(1-Y)\nP A + i(1-Y)\nP\bar{A}.
\end{equation*}

We first compute $\nP A$. We use the paralinearization in Lemma \ref{t:Paralinear} to write
\begin{align*}
\nP [J^{-\f32}(1-Y)\p^4_{\al}\bw]
    &= T_{J^{-\f32}(1-Y)}\p^4_{\al}\bw+\nP T_{\p^4_{\al}\bw}(J^{-\f32}(1-Y)-1)+\nP\Pi(\p^4_{\al}\bw,J^{-\f32}(1-Y)-1) \\
    &= T_{J^{-\f32}(1-Y)}\p^4_{\al}\bw- \f52 T_{(1-Y)^2J^{-\f32}}T_{\p^4_{\al}\bw} \bw -\f52 T_{(1-Y)^2J^{-\f32}}\Pi(\p^4_{\al}\bw,\bw) \\
    &\quad -\f32 T_{J^{-\f52}}\nP\Pi(\p^4_{\al}\bw, \bar{\bw})+K,\\
-\f{15}{2}\nP[J^{-\f32}(1-Y)^2\bw_{\al}\p^3_{\al}\bw] 
    &= -\f{15}{2} T_{J^{-\f32}(1-Y)^2\bw_{\al}}\p^3_{\al}\bw - \f{15}{2} T_{\p^3_{\al}\bw} \nP [J^{-\f32}(1-Y)^2\bw_{\al}] \\
    &\quad - \f{15}{2} \nP \Pi(\p^3_{\al}\bw, J^{-\f32}(1-Y)^2\bw_{\al}) \\
    &= -\f{15}{2} T_{J^{-\f32}(1-Y)^2}T_{\bw_{\al}}\p^3_{\al}\bw -\f{15}{2} T_{J^{-\f32}(1-Y)^2}T_{\p^3_{\al}\bw}\bw_{\al} \\
    &\quad - \f{15}{2} T_{J^{-\f32}(1-Y)^2}\nP \Pi(\p^3_{\al}\bw, \bw_{\al}) + K, \\
-3\nP[ J^{-\f52}\bar{\bw}_{\al}\p^3_{\al}\bw]
    &= -3 T_{J^{-\f52}\bar{\bw}_{\al}}\p^3_{\al}\bw - 3 T_{\p^3_{\al}\bw} \nP (J^{-\f52}\bar{\bw}_{\al}) - 3\nP\Pi(\p^3_{\al}\bw, J^{-\f52}\bar{\bw}_{\al}) \\
    &= -3 T_{J^{-\f52}\bar{\bw}_{\al}}\p^3_{\al}\bw - 3T_{J^{-\f52}}\nP\Pi(\p^3_{\al}\bw, \bar{\bw}_{\al}) +K, \\
-\f12 \nP[J^{-\f52}\bw_{\al}\p^3_{\al}\bar{\bw}] 
    &= -\f12 T_{J^{-\f52}\p^3_{\al}\bar{\bw}} \bw_{\al} -\f12 \nP \Pi(J^{-\f52}\p^3_{\al}\bar{\bw}, \bw_{\al}) \\
    &= -\f12 T_{J^{-\f52}\p^3_{\al}\bar{\bw}} \bw_{\al} -\f12 T_{J^{-\f52}}\nP \Pi(\p^3_{\al}\bar{\bw}, \bw_{\al}) + K,\\
-5\nP [J^{-\f32}(1-Y)^2\bw^2_{\al\al}] 
    &= -5 T_{J^{-\f32}(1-Y)^2\bw_{\al\al}} \bw_{\al \al} - 5 T_{\bw_{\al \al}} \nP(J^{-\f32}(1-Y)^2\bw_{\al\al}) \\
    &\quad -5 \nP \Pi(\bw_{\al \al}, J^{-\f32}(1-Y)^2\bw_{\al\al}) \\
    &= -10 T_{J^{-\f32}(1-Y)^2\bw_{\al\al}} \bw_{\al \al} -5  T_{J^{-\f32}(1-Y)^2}\Pi(\bw_{\al \al}, \bw_{\al\al}) + K,\\
-2 \nP [J^{-\f52}|\bw_{\al\al}|^2] 
    &= -2 T_{J^{-\f52}\bar{\bw}_{\al \al}}\bw_{\al \al} -2 T_{J^{-\f52}} \nP \Pi(\bar{\bw}_{\al \al}, \bw_{\al \al}) + K ,\\
\f{65}{2}\nP[J^{-\f32}(1-Y)^3\bw_{\al}^2\bw_{\al\al}] 
    &= \f{65}{2}T_{J^{-\f32}(1-Y)^3\bw_{\al}^2}\bw_{\al\al} + K,\\
\f{13}{4}\nP[J^{-\f52}(1-\bar{Y})|\bw_{\al}|^2\bar{\bw}_{\al\al}] 
    &= K ,\\
\f{53}{4} \nP[J^{-\f52}(1-Y)|\bw_{\al}|^2\bw_{\al\al}] 
    &= \f{53}{4}T_{J^{-\f52}(1-Y)|\bw_{\al}|^2}\bw_{\al\al} + K,\\
\f{9}{4}\nP [J^{-\f52}(1-Y)\bw^2_\al\bar{\bw}_{\al\al}] 
    &= K,\\
\f{19}{4} \nP [J^{-\f52}(1-\bar{Y})\bar{\bw}_{\al}^2\bw_{\al\al}] 
    &= \f{19}{4}T_{J^{-\f52}(1-\bar{Y})\bar{\bw}_{\al}^2}\bw_{\al\al} +K.
\end{align*}

Collecting all above terms, we obtain the expression for $PA$:
\begin{align*}
\nP A
&=T_{J^{-\f32}(1-Y)}\p^4_{\al}\bw-\f{15}{2}T_{J^{-\f32}(1-Y)^2}T_{\bw_{\al}}\p^{3}_{\al}\bw -10T_{J^{-\f32}(1-Y)^2}T_{\bw_{\al\al}}\bw_{\al\al}\\
&-\f{15}{2}T_{J^{-\f32}(1-Y)^2}T_{\p^3_{\al}\bw}\bw_{\al} -\f52 T_{(1-Y)^2J^{-\f32}}T_{\p^4_{\al}\bw} \bw-3 T_{J^{-\f52}}T_{\bar{\bw}_{\al}}\p^3_{\al}\bw \\
&-2 T_{J^{-\f52}}T_{\bar{\bw}_{\al\al}}\bw_{\al\al}-\f12 T_{J^{-\f52}}T_{\p^3_{\al}\bar{\bw}}
\bw_{\al} +\f{65}{2}T_{J^{-\f32}(1-Y)^3\bw_{\al}^2}\bw_{\al\al}\\
&+\f{53}{4}T_{J^{-\f52}(1-Y)|\bw_{\al}|^2}\bw_{\al\al}+\f{19}{4}T_{J^{-\f52}(1-\bar{Y})\bar{\bw}_{\al}^2}\bw_{\al\al} -\f52 T_{(1-Y)^2J^{-\f32}}\Pi(\p^4_{\al}\bw,\bw)\\
& -\f{15}{2}T_{J^{-\f32}(1-Y)^2}\Pi(\p^3_{\al}\bw,\bw_{\al})-5 T_{J^{-\f32}(1-Y)^2}\Pi(\bw_{\al\al},\bw_{\al\al}) -\f12 T_{J^{-\f52}}\nP\Pi(\p^3_{\al}\bar{\bw},
\bw_{\al}) \\
&-2 T_{J^{-\f52}}\nP\Pi(\bar{\bw}_{\al\al},
\bw_{\al\al}) -3 T_{J^{-\f52}}\nP\Pi(\bar{\bw}_{\al}, \p^3_{\al}\bw)-\f32 T_{J^{-\f52}}\nP\Pi(\bar{\bw}, \p^4_{\al}\bw)+K.
\end{align*}
Using the expression of $\nP A$, we derive the expression for $(1-Y)\nP A$:
\begin{align*}
&(1-Y)\nP A=T_{1-Y}\nP A-T_{\nP A}Y-\Pi(\nP A,Y)\\
=&T_{1-Y}\nP A-T_{\nP A}T_{(1-Y)^2}\bw-\Pi(\nP A,T_{(1-Y)^2}\bw)+K\\
=&T_{J^{-\f32}(1-Y)^2}\p^4_{\al}\bw-\f{15}{2}T_{J^{-\f32}(1-Y)^3\bw_{\al}}\p^{3}_{\al}\bw-3 T_{J^{-\f52}(1-Y)\bar{\bw}_{\al}}\p^3_{\al}\bw\\
&-10T_{J^{-\f32}(1-Y)^3\bw_{\al\al}}\bw_{\al\al}-2 T_{J^{-\f52}(1-Y)\bar{\bw}_{\al\al}}
\bw_{\al\al}+\f{65}{2}T_{J^{-\f32}(1-Y)^4\bw_{\al}^2}\bw_{\al\al}\\
&+\f{53}{4}T_{J^{-\f52}(1-Y)^2|\bw_{\al}|^2}\bw_{\al\al}+\f{19}{4}T_{J^{-\f72})\bar{\bw}_{\al}^2}\bw_{\al\al}-\f52T_{(1-Y)^3J^{-\f32}\p^4_{\al}\bw}\bw\\
&-\f{15}{2}T_{\p^3_{\al}\bw}T_{J^{-\f32}(1-Y)^3}\bw_{\al}-\f12 T_{J^{-\f52}(1-Y)\p^3_{\al}\bar{\bw}}
\bw_{\al}\\
&-\f52\nP\Pi(\p^4_{\al}\bw,T_{(1-Y)^3J^{-\f32}}\bw)-\f32\nP\Pi(\p^4_{\al}\bw, T_{J^{-\f52}(1-Y)}\bar{\bw}) -\f{15}{2}\nP\Pi(\p^3_{\al}\bw,T_{J^{-\f32}(1-Y)^3}\bw_{\al})\\
&-3 \nP\Pi(\p^3_{\al}\bw,T_{J^{-\f52}(1-Y)}\bar{\bw}_{\al})-\f12 \nP\Pi(\p^3_{\al}\bar{\bw},T_{J^{-\f52}(1-Y)}
\bw_{\al})-5\nP\Pi(\bw_{\al\al},T_{J^{-\f32}(1-Y)^3}\bw_{\al\al})\\
&-2 \nP\Pi(\bar{\bw}_{\al\al},T_{J^{-\f52}(1-Y)}
\bw_{\al\al})-T_{J^{-\f32}(1-Y)^3\p^4_{\al}\bw}\bw-\nP\Pi(\p^4_{\al}\bw,T_{(1-Y)^3J^{-\f32}}\bw)+K\\
=&T_{J^{-\f32}(1-Y)^2}\p^4_{\al}\bw-\f{15}{2}T_{J^{-\f32}(1-Y)^3}T_{\bw_{\al}}\p^{3}_{\al}\bw-3 T_{J^{-\f52}(1-Y)}T_{\bar{\bw}_{\al}}\p^3_{\al}\bw\\
&-10T_{J^{-\f32}(1-Y)^3}T_{\bw_{\al\al}}\bw_{\al\al}-2 T_{J^{-\f52}(1-Y)}T_{\bar{\bw}_{\al\al}}
\bw_{\al\al}+\f{65}{2}T_{J^{-\f32}(1-Y)^4\bw_{\al}^2}\bw_{\al\al}\\
&+\f{53}{4}T_{J^{-\f52}(1-Y)^2|\bw_{\al}|^2}\bw_{\al\al}+\f{19}{4}T_{J^{-\f72}\bar{\bw}_{\al}^2}\bw_{\al\al}-\f72T_{(1-Y)^3J^{-\f32}}T_{\p^4_{\al}\bw}\bw\\
&-\f{15}{2}T_{J^{-\f32}(1-Y)^3}T_{\p^3_{\al}\bw}\bw_{\al}-\f12 T_{J^{-\f52}(1-Y)}T_{\p^3_{\al}\bar{\bw}}
\bw_{\al}\\
&-\f72 T_{(1-Y)^3J^{-\f32}}\Pi(\p^4_{\al}\bw,\bw) -\f{15}{2}T_{J^{-\f32}(1-Y)^3}\Pi(\p^3_{\al}\bw,\bw_{\al}) -5T_{J^{-\f32}(1-Y)^3}\Pi(\bw_{\al\al},\bw_{\al\al})\\
&-\f32 T_{J^{-\f52}(1-Y)}\nP\Pi(\bar{\bw}, \p^4_{\al}\bw)-3 T_{J^{-\f52}(1-Y)}\nP\Pi(\bar{\bw}_{\al}, \p^3_{\al}\bw)-2 T_{J^{-\f52}(1-Y)}\nP\Pi(\bar{\bw}_{\al\al}, \bw_{\al\al})\\
&-\f12 T_{J^{-\f52}(1-Y)}\nP\Pi(\p^3_{\al}\bar{\bw}, \bw_{\al}) + K.
\end{align*}

Similarly, we compute $\nP \bar{A}$.
\begin{align*}
\nP [J^{-\f{3}{2}}(1-\bar{Y})\p^4_{\al}\bar{\bw}] 
    &= T_{\p^4_{\al}\bar{\bw}}\nP(J^{-\f{3}{2}}(1-\bar{Y})-1)+\nP\Pi(\p^4_{\al}\bar{\bw},J^{-\f{3}{2}}(1-\bar{Y})-1) \\
    &= - \f{3}{2} T_{J^{-\f{5}{2}}}T_{\p^4_{\al}\bar{\bw}} \bw -\f{3}{2} T_{J^{-\f{5}{2}}}\nP\Pi(\p^4_{\al}\bar{\bw}, \bw)+K,\\
-3\nP[ J^{-\f{5}{2}}\bw_{\al}\p^3_{\al}\bar{\bw}] 
    &= - 3 T_{\p^3_{\al}\bar{\bw}} \nP (J^{-\f{5}{2}}\bw_{\al}) - 3\nP\Pi(\p^3_{\al}\bar{\bw}, J^{-\f{5}{2}}\bw_{\al}) \\
    &= -3 T_{J^{-\f{5}{2}}}T_{\p^3_{\al}\bar{\bw}}\bw_{\al} - 3T_{J^{-\f{5}{2}}}\nP\Pi(\p^3_{\al}\bar{\bw}, \bw_{\al}) +K, \\
-\f{1}{2} \nP[J^{-\f{5}{2}}\bar{\bw}_{\al}\p^3_{\al}\bw] 
    &= -\f{1}{2} T_{J^{-\f{5}{2}}\bar{\bw}_{\al}} \p^3_{\al}\bw -\f{1}{2} \nP \Pi(J^{-\f{5}{2}}\p^3_{\al}\bw, \bar{\bw}_{\al}) \\
    &= -\f{1}{2} T_{J^{-\f{5}{2}}}T_{\bar{\bw}_{\al}} \p^3_{\al}\bw -\f{1}{2} T_{J^{-\f{5}{2}}}\nP \Pi(\p^3_{\al}\bw, \bar{\bw}_{\al}) + K, \\
-2 \nP [J^{-\f{5}{2}}|\bw_{\al\al}|^2] 
    &= -2 T_{J^{-\f{5}{2}}\bar{\bw}_{\al \al}}\bw_{\al \al} -2 T_{J^{-\f{5}{2}}} \nP \Pi(\bar{\bw}_{\al \al}, \bw_{\al \al}) + K , \\
\f{13}{4}\nP[J^{-\f{5}{2}}(1-Y)|\bw_{\al}|^2\bw_{\al\al}] 
    &= \f{13}{4}T_{J^{-\f{5}{2}}(1-Y)|\bw_{\al}|^2}\bw_{\al\al}+K , \\
\f{53}{4} \nP[J^{-\f{5}{2}}(1-\bar{Y})|\bw_{\al}|^2\bar{\bw}_{\al\al}] 
    &= K, \\
\f{9}{4}\nP [J^{-\f{5}{2}}(1-\bar{Y})\bar{\bw}^2_{\al}\bw_{\al\al}] 
    &= \f{9}{4} T_{J^{-\f{5}{2}}(1-\bar{Y})\bar{\bw}^2_{\al}}\bw_{\al\al}+K, \\
\f{19}{4} \nP [J^{-\f{5}{2}}(1-Y)\bw_{\al}^2\bar{\bw}_{\al\al}] 
    &= K.
\end{align*}
Therefore, we obtain
\begin{align*}
\nP\bar{A}=&- \f32 T_{J^{-\f52}}T_{\p^4_{\al}\bar{\bw}} \bw 
-\f12 T_{J^{-\f52}}T_{\bar{\bw}_{\al}} \p^3_{\al}\bw 
-2 T_{J^{-\f52}}T_{\bar{\bw}_{\al \al}}\bw_{\al \al} -3 T_{J^{-\f52}\p^3_{\al}\bar{\bw}}\bw_{\al}
\\
&+\f{13}{4}T_{J^{-\f52}(1-Y)|\bw_{\al}|^2}\bw_{\al\al}+\f94 T_{J^{-\f52}(1-\bar{Y})\bar{\bw}^2_{\al}}\bw_{\al\al}-\f32 T_{J^{-\f52}}\nP\Pi(\p^4_{\al}\bar{\bw}, \bw)\\
&- 3T_{J^{-\f52}}\nP\Pi(\p^3_{\al}\bar{\bw}, \bw_{\al}) -2 T_{J^{-\f52}} \nP \Pi(\bar{\bw}_{\al \al}, \bw_{\al \al}) -\f12 T_{J^{-\f52}}\nP \Pi(\bar{\bw}_{\al},
\p^3_{\al}\bw) + K.
\end{align*}
Using the expression of $\nP \bar{A}$, we get the equation for $(1-Y)\nP \bar{A}$:
\begin{align*}
(1-Y)\nP \bar{A}
=&T_{1-Y}\nP \bar{A}-T_{\nP \bar{A}}T_{(1-Y)^2}\bw-\Pi(\nP \bar{A},T_{(1-Y)^2}\bw)+K \\
=& - \f32 T_{J^{-\f52}(1-Y)}T_{\p^4_{\al}\bar{\bw}} \bw 
-\f12 T_{J^{-\f52}(1-Y)}T_{\bar{\bw}_{\al}} \p^3_{\al}\bw 
-2 T_{J^{-\f52}(1-Y)}T_{\bar{\bw}_{\al \al}}\bw_{\al \al} \\
&-3 T_{J^{-\f52}(1-Y)}T_{\p^3_{\al}\bar{\bw}}\bw_{\al}+\f{13}{4}T_{J^{-\f52}(1-Y)^2|\bw_{\al}|^2}\bw_{\al\al}+\f94 T_{J^{-\f72}\bar{\bw}^2_{\al}}\bw_{\al\al} \\
&-\f32 T_{J^{-\f52}(1-Y)}\nP\Pi(\p^4_{\al}\bar{\bw}, \bw)- 3T_{J^{-\f52}(1-Y)}\nP\Pi(\p^3_{\al}\bar{\bw}, \bw_{\al}) \\
&-2 T_{J^{-\f52}(1-Y)} \nP \Pi(\bar{\bw}_{\al \al}, \bw_{\al \al}) -\f12 T_{J^{-\f52}(1-Y)}\nP \Pi(\bar{\bw}_{\al},
\p^3_{\al}\bw) + K.
\end{align*}
Combining the expressions for $(1-Y)\nP A$ and $(1-Y)\nP\bar{A}$, the elastic terms can be written as:
\begin{align*}
    &\text{Elastic terms} =  -i(1-Y)\nP A + i(1-Y)\nP\bar{A}\\
    =&-iT_{J^{-\f32}(1-Y)^2}\p^4_{\al}\bw+\f{15}{2}iT_{J^{-\f32}(1-Y)^3}T_{\bw_{\al}}\p^{3}_{\al}\bw+3 iT_{J^{-\f52}(1-Y)}T_{\bar{\bw}_{\al}}\p^3_{\al}\bw\\
&+10iT_{J^{-\f32}(1-Y)^3}T_{\bw_{\al\al}}\bw_{\al\al}+2i T_{J^{-\f52}(1-Y)}T_{\bar{\bw}_{\al\al}}
\bw_{\al\al}-\f{65}{2}iT_{J^{-\f32}(1-Y)^4\bw_{\al}^2}\bw_{\al\al}\\
&-\f{53}{4}iT_{J^{-\f52}(1-Y)^2|\bw_{\al}|^2}\bw_{\al\al}-\f{19}{4}iT_{J^{-\f72}\bar{\bw}_{\al}^2}\bw_{\al\al}+\f72iT_{(1-Y)^3J^{-\f32}}T_{\p^4_{\al}\bw}\bw\\
&+\f{15}{2}iT_{J^{-\f32}(1-Y)^3}T_{\p^3_{\al}\bw}\bw_{\al}+\f12 iT_{J^{-\f52}(1-Y)}T_{\p^3_{\al}\bar{\bw}}
\bw_{\al}\\
&+\f72i T_{(1-Y)^3J^{-\f32}}\Pi(\p^4_{\al}\bw,\bw) +\f{15}{2}iT_{J^{-\f32}(1-Y)^3}\Pi(\p^3_{\al}\bw,\bw_{\al}) +5iT_{J^{-\f32}(1-Y)^3}\Pi(\bw_{\al\al},\bw_{\al\al})\\
&+\f32i T_{J^{-\f52}(1-Y)}\nP\Pi(\bar{\bw}, \p^4_{\al}\bw)+3i T_{J^{-\f52}(1-Y)}\nP\Pi(\bar{\bw}_{\al}, \p^3_{\al}\bw)+2i T_{J^{-\f52}(1-Y)}\nP\Pi(\bar{\bw}_{\al\al}, \bw_{\al\al})\\
&+\f12i T_{J^{-\f52}(1-Y)}\nP\Pi(\p^3_{\al}\bar{\bw}, \bw_{\al})  - \f32i T_{J^{-\f52}(1-Y)}T_{\p^4_{\al}\bar{\bw}} \bw 
-\f12i T_{J^{-\f52}(1-Y)}T_{\bar{\bw}_{\al}} \p^3_{\al}\bw 
\\
&-2i T_{J^{-\f52}(1-Y)}T_{\bar{\bw}_{\al \al}}\bw_{\al \al} 
-3i T_{J^{-\f52}(1-Y)}T_{\p^3_{\al}\bar{\bw}}\bw_{\al}+\f{13}{4}iT_{J^{-\f52}(1-Y)^2|\bw_{\al}|^2}\bw_{\al\al}\\
&+\f94i T_{J^{-\f72}\bar{\bw}^2_{\al}}\bw_{\al\al} -\f32i T_{J^{-\f52}(1-Y)}\nP\Pi(\p^4_{\al}\bar{\bw}, \bw)- 3iT_{J^{-\f52}(1-Y)}\nP\Pi(\p^3_{\al}\bar{\bw}, \bw_{\al}) \\
&-2i T_{J^{-\f52}(1-Y)} \nP \Pi(\bar{\bw}_{\al \al}, \bw_{\al \al}) -\f12 iT_{J^{-\f52}(1-Y)}\nP \Pi(\bar{\bw}_{\al},
\p^3_{\al}\bw) +K\\
=&-iT_{J^{-\f32}(1-Y)^2}\p^4_{\al}\bw+\f{15}{2}iT_{J^{-\f32}(1-Y)^3}T_{\bw_{\al}}\p^{3}_{\al}\bw+\f52iT_{J^{-\f52}(1-Y)}T_{\bar{\bw}_{\al}}\p^3_{\al}\bw\\
&+10iT_{J^{-\f32}(1-Y)^3}T_{\bw_{\al\al}}\bw_{\al\al}+\f{15}{2}iT_{J^{-\f32}(1-Y)^3}T_{\p^3_{\al}\bw}\bw_{\al}-\f52 iT_{J^{-\f52}(1-Y)}T_{\p^3_{\al}\bar{\bw}}
\bw_{\al} \\
&+\f72iT_{(1-Y)^3J^{-\f32}}T_{\p^4_{\al}\bw}\bw- \f32i T_{J^{-\f52}(1-Y)}T_{\p^4_{\al}\bar{\bw}} \bw-\f{65}{2}iT_{J^{-\f32}(1-Y)^4\bw_{\al}^2}\bw_{\al\al}\\
&-10iT_{J^{-\f52}(1-Y)^2|\bw_{\al}|^2}\bw_{\al\al}-\f{5}{2}iT_{J^{-\f72}\bar{\bw}_{\al}^2}\bw_{\al\al}+\f72i T_{(1-Y)^3J^{-\f32}}\Pi(\p^4_{\al}\bw,\bw)\\
&+\f{15}{2}iT_{J^{-\f32}(1-Y)^3}\Pi(\p^3_{\al}\bw,\bw_{\al}) +5iT_{J^{-\f32}(1-Y)^3}\Pi(\bw_{\al\al},\bw_{\al\al})\\
&+\f32i T_{J^{-\f52}(1-Y)}\nP\Pi(\bar{\bw}, \p^4_{\al}\bw)+\f52i T_{J^{-\f52}(1-Y)}\nP\Pi(\bar{\bw}_{\al}, \p^3_{\al}\bw)\\
&-\f52i T_{J^{-\f52}(1-Y)}\nP\Pi(\p^3_{\al}\bar{\bw}, \bw_{\al}) -\f32i T_{J^{-\f52}(1-Y)}\nP\Pi(\p^4_{\al}\bar{\bw}, \bw) +K.
\end{align*}

Therefore, by moving balanced and perturbative terms to the right-hand side, $R$ solves the paradifferential equation
\begin{align*}
&T_{D_t}R + T_{b_\alpha}R -iT_{J^{-\f32}(1-Y)^2}\p^4_{\al}\bw+\f{15}{2}iT_{J^{-\f32}(1-Y)^3}T_{\bw_{\al}}\p^{3}_{\al}\bw\\
&+\f52iT_{J^{-\f52}(1-Y)}T_{\bar{\bw}_{\al}}\p^3_{\al}\bw+10iT_{J^{-\f32}(1-Y)^3}T_{\bw_{\al\al}}\bw_{\al\al}+\f{15}{2}iT_{J^{-\f32}(1-Y)^3}T_{\p^3_{\al}\bw}\bw_{\al}\\
&-\f52 iT_{J^{-\f52}(1-Y)}T_{\p^3_{\al}\bar{\bw}}
\bw_{\al} +\f72iT_{(1-Y)^3J^{-\f32}}T_{\p^4_{\al}\bw}\bw - \f32i T_{J^{-\f52}(1-Y)}T_{\p^4_{\al}\bar{\bw}} \bw\\
&-\f{65}{2}iT_{J^{-\f32}(1-Y)^4\bw_{\al}^2}\bw_{\al\al}-10iT_{J^{-\f52}(1-Y)^2|\bw_{\al}|^2}\bw_{\al\al}-\f{5}{2}iT_{J^{-\f72}\bar{\bw}_{\al}^2}\bw_{\al\al} \\
=&-T_{1-\bar{Y}}\Pi(R_\alpha, R) - T_{1-Y}\nP \p_\al\Pi(\bar{R}, R) -\f72i T_{(1-Y)^3J^{-\f32}}\Pi(\p^4_{\al}\bw,\bw) \\
&-\f{15}{2}iT_{J^{-\f32}(1-Y)^3}\Pi(\p^3_{\al}\bw,\bw_{\al}) -5iT_{J^{-\f32}(1-Y)^3}\Pi(\bw_{\al\al},\bw_{\al\al})\\
&-\f32i T_{J^{-\f52}(1-Y)}\nP\Pi(\bar{\bw}, \p^4_{\al}\bw)-\f52i T_{J^{-\f52}(1-Y)}\nP\Pi(\bar{\bw}_{\al}, \p^3_{\al}\bw)\\
&+\f52i T_{J^{-\f52}(1-Y)}\nP\Pi(\p^3_{\al}\bar{\bw}, \bw_{\al}) +\f32i T_{J^{-\f52}(1-Y)}\nP\Pi(\p^4_{\al}\bar{\bw}, \bw) + K.
\end{align*}

\subsection{Normal form transformation of the hydroelastic waves} \label{s:NormalHydroWaves}
In this section, we demonstrate the existence of balanced normal form corrections $(\tilde{\bw}, \tilde{R})$ capable of removing the non-perturbative balanced source terms in the paradifferential equations for $\bw$ and $R$.
Specifically, we seek quadratic normal form corrections of balanced type $(\tilde{\bw}, \tilde{R})$ such that
 \begin{align*}
&\p_t\tilde{\mathbf{W}} +  T_{1-\bar{Y}}T_{1+\bw} \tilde{R}_\al   + \text{cubic and higher terms}\\
=& T_{1- \bar{Y}}\p_\al \Pi(\bw, R) +T_{1-Y}\partial_\alpha \nP \Pi( \bar{R}, \bw) +T_{(1-\bar{Y})^2(1+\bw)}\partial_\alpha \nP\Pi(\bar{\bw}, R) , \\
&\p_t \tilde{R} -iT_{J^{-\f32}(1-Y)^2}\p^4_{\al}\tilde{\bw}  + \text{cubic and higher terms} \\
=&T_{1-\bar{Y}}\Pi(R_\alpha, R) + T_{1-Y}\nP \p_\al\Pi(\bar{R}, R) +\f72i T_{(1-Y)^3J^{-\f32}}\Pi(\p^4_{\al}\bw,\bw) \\
&+\f{15}{2}iT_{J^{-\f32}(1-Y)^3}\Pi(\p^3_{\al}\bw,\bw_{\al}) +5iT_{J^{-\f32}(1-Y)^3}\Pi(\bw_{\al\al},\bw_{\al\al})\\
&+\f32i T_{J^{-\f52}(1-Y)}\nP\Pi(\bar{\bw}, \p^4_{\al}\bw)+\f52i T_{J^{-\f52}(1-Y)}\nP\Pi(\bar{\bw}_{\al}, \p^3_{\al}\bw)\\
&-\f52i T_{J^{-\f52}(1-Y)}\nP\Pi(\p^3_{\al}\bar{\bw}, \bw_{\al}) -\f32i T_{J^{-\f52}(1-Y)}\nP\Pi(\p^4_{\al}\bar{\bw}, \bw) .
\end{align*}
We consider balanced quadratic normal form corrections as the sum of the holomorphic type and the mixed type:
\begin{align*}
\tilde{\mathbf{W}}   &= T_{1-Y}B^h_{bal}(\bw, \bw) +T_{J^\f12(1+\bw)^3} C^h_{bal}(R, R) + B^a_{bal}(\bar{\bw}, T_{1-\bar{Y}} \bw) + C^a_{bal}\Big(\bar{R}, T_{J^\f32(1+\bw)}R\Big) \\
\tilde{R} &= A^h_{bal}(\bw, T_{1-Y}R)  + A^a_{bal}\Big(\bar{R}, T_{(1-Y)^2(1+\bar{\bw})}\bw\Big) + D^a_{bal}(\bar{\bw}, T_{1-\bar{Y}} R).
\end{align*}

For above bilinear forms, we compute
\begin{align*}
&\partial_t \tilde{\mathbf{W}}+T_{1-\bar{Y}} \partial_\alpha \tilde{\mathbf{R}} + \text{cubic and higher terms} \\
=& T_{1-\bar{Y}}\partial_\alpha A^{h}_{bal}(\bw, R) -2T_{1-\bar{Y}}B^{h}_{bal}(\bw,R_{\al}) + 2 T_{1-\bar{Y}}C^{h}_{bal}(i\p^4_{\al}\bw,R) \\
&+T_{(1-\bar{Y})^2(1+\bw)}\p_{\al}A^{a}_{bal}(\bar{\bw},R)-T_{(1-\bar{Y})^2(1+\bw)}B^{a}_{bal}(\bar{\bw},R_{\al})-T_{(1-\bar{Y})^2(1+\bw)}C^{a}_{bal}(i\p^4_{\al}\bar{\bw},R)\\
 &-T_{1-Y}B^{a}_{bal}(\bar{R}_\alpha,\bw)+T_{1-Y}C^{a}_{bal}(\bar{R},i\p_\al^4\bw) + T_{1-Y}\partial_\alpha D^{a}_{bal}(\bar{R}, \bw),\\
&\partial_t \tilde{\mathbf{R}}  -iT_{J^{-\frac{3}{2}}(1-Y)}\p_\al^4 \tilde{\mathbf{W}}+ \text{cubic and higher  terms} \\
=& -T_{1-\bar{Y}}A^{h}_{bal}(R_\alpha ,R) - iT_{1-\bar{Y}}\partial_\alpha^4 C^{h}_{bal}(R,R)\\
&+iT_{J^{-\frac{3}{2}}(1-Y)^3}A^{h}_{bal}(\bw, \p_\al^4\bw) -iT_{J^{-\frac{3}{2}}(1-Y)^3}\partial_\alpha^4 B^{h}_{bal}(\bw,\bw) \\
 &-T_{1-Y}A^{a}_{bal}(\bar{R}_\alpha,R) - iT_{1-Y}\partial_\alpha^4 C^{a}_{bal}( \bar{R},R) - T_{1-Y}D^{a}_{bal}(\bar{R},R_\alpha)\\
&+ iT_{J^{-\f52}(1-Y)}A^{a}_{bal}(\bar{\bw},\p_\al^4 \bw) -iT_{J^{-\f52}(1-Y)}\partial_\alpha^4 B^{a}_{bal}(\bar{\bw}, \bw) -iT_{J^{-\f52}(1-Y)}D^{h}_{bal}(\p_\al^4\bar{\bw},\bw).
\end{align*}
Here, we denote by $\mathfrak{a}^{h}_{bal}(\xi, \eta)$  the symbol of $A^{h}_{bal}(R,T_{1-Y}\bw)$, and by
$\mathfrak{a}^{a}_{bal}(\eta, \zeta)$  the symbol of $A^{a}_{bal}\Big(\bar{R}, T_{(1-Y)^2(1+\bar{\bw})}\bw\Big)$. 
This is similarly done for other balanced bilinear forms of holomorphic or mixed type.
Note that for $\mathfrak{b}^{h}_{bal}$ and  $\mathfrak{c}^{h}_{bal}$, their symbols are symmetric with respect to $\xi$ and $\eta$.
To match the balanced paradifferential source terms of the holomorphic type, paradifferential symbols of the holomorphic type solve the following system:
\begin{equation} \label{SymmetrizedSys}
\begin{cases}
(\xi + \eta)\mathfrak{a}^{h}_{bal}-2\eta\mathfrak{b}^{h}_{bal}+2\xi^4\mathfrak{c}^{h}_{bal}= (\xi+\eta)\chi_2(\xi,\eta),\\
(\xi\mathfrak{a}^{h}_{bal})_{sym}+(\xi+\eta)^4\mathfrak{c}^{h}_{bal}= -\f12(\xi+\eta)\chi_{2}(\xi,\eta),\\
(\eta^4 \mathfrak{a}^{h}_{bal})_{sym}-(\xi+\eta)^4\mathfrak{b}^{h}_{bal}=\left(\f72\xi^4+\f{15}{2}\xi^3\eta+5\xi^2\eta^2\right)_{sys}\chi_{2}(\xi,\eta),
\end{cases}
\end{equation}
where $m_{sym}$ stands for the symmetrization of symbol $m$, and $\chi_{2}(\xi, \eta)$ is defined in \eqref{ChiTwohh} to select the balanced frequencies.

From the first equation of the system \eqref{SymmetrizedSys}, we get that
\begin{equation} \label{ahbalsym}
   \mathfrak{a}^{h}_{bal}(\xi, \eta) = \f{2\eta}{\xi + \eta} \mathfrak{b}^{h}_{bal} -  \f{2\xi^4}{\xi + \eta} \mathfrak{c}^{h}_{bal} + \chi_2(\xi,\eta).
\end{equation}
Then we compute two symmetrized symbols
\begin{align*}
&(\eta\mathfrak{a}^{h}_{bal})_{sym}= \f{2\xi\eta}{\xi + \eta} \mathfrak{b}^{h}_{bal} - \f{\xi^5+\eta^5}{\xi+\eta}\mathfrak{c}^{h}_{bal}+ \f12(\xi+\eta)\chi_2(\xi, \eta),\\
&(\xi^4 \mathfrak{a}^{h}_{bal})_{sym}= \f{\xi^5+\eta^5}{\xi + \eta} \mathfrak{b}^{h}_{bal} - \f{2\xi^4\eta^4}{\xi+\eta}\mathfrak{c}^{h}_{bal}+ \f12(\xi^4+\eta^4)\chi_2(\xi, \eta).
\end{align*}
Substituting these identities into the second and the third equation of the system \eqref{SymmetrizedSys}, $\mathfrak{b}^{h}_{bal}$ and $\mathfrak{c}^{h}_{bal}$ solve the system
\begin{equation*}
\begin{cases}
\f{2\xi\eta}{\xi + \eta} \mathfrak{b}^{h}_{bal} +\f{(\xi+\eta)^5-(\xi^5+\eta^5)}{\xi+\eta}\mathfrak{c}^{h}_{bal}= -(\xi+\eta)\chi_{2}(\xi,\eta),\\
\f{(\xi^5+\eta^5)-(\xi+\eta)^5}{\xi + \eta} \mathfrak{b}^{h}_{bal} - \f{2\xi^4\eta^4}{\xi+\eta}\mathfrak{c}^{h}_{bal}=\left(\f54\xi^4+\f{15}{4}\xi^3\eta+5\xi^2 \eta^2 + \f{15}{4}\xi\eta^3 + \f54 \eta^4\right)\chi_{2}(\xi,\eta).
\end{cases}
\end{equation*}
Hence, the solutions of \eqref{SymmetrizedSys} are
\begin{align*}
\mathfrak{b}^{h}_{bal}=&-\f{(\xi+\eta)^2(25\eta^6+100\eta^5\xi+200\eta^4\xi^2+242\eta^3\xi^3+200\eta^2\xi^4+100\eta^2\xi^5+25\xi^6)}{4\xi\eta(25\xi^6+100\xi^5\eta+200\xi^4\eta^2+246\xi^3\eta^3+200\xi^2\eta^4+100\xi\eta^5+25\eta^6)}\chi_{2}(\xi,\eta),\\
\mathfrak{c}^{h}_{bal}=&-\f{5(\xi+\eta)^3(\xi^2+\xi\eta+\eta^2)}{2\xi\eta(25\xi^6+100\xi^5\eta+200\xi^4\eta^2+246\xi^3\eta^3+200\xi^2\eta^4+100\xi\eta^5+25\eta^6)}\chi_{2}(\xi,\eta),
\end{align*}
and $\mathfrak{a}^{h}_{bal}$ is given by \eqref{ahbalsym}.

To match the balanced paradifferential source terms of the mixed type, paradifferential symbols of the mixed type solve the following  system:
\begin{equation*}
\begin{cases}
(\zeta-\eta)\mathfrak{a}^{a}_{bal}-\zeta\mathfrak{b}^{a}_{bal}-\eta^4\mathfrak{c}^{a}_{bal}=(\zeta-\eta)\chi_{2}(\eta,\zeta)1_{\zeta<\eta},\\
\eta\mathfrak{b}^{a}_{bal}+\zeta^4\mathfrak{c}^{a}_{bal}+(\zeta-\eta)\mathfrak{d}^{a}_{bal}=(\zeta-\eta)\chi_{2}(\eta,\zeta)1_{\zeta<\eta},\\
\eta\mathfrak{a}^{a}_{bal}-(\zeta-\eta)^4\mathfrak{c}^{a}_{bal}-\zeta\mathfrak{d}^{a}_{bal}=(\zeta-\eta)\chi_{2}(\eta,\zeta)1_{\zeta<\eta},\\
\zeta^4 \mathfrak{a}^{a}_{bal}-(\zeta-\eta)^4\mathfrak{b}^{a}_{bal}-\eta^4\mathfrak{d}^{a}_{bal}=\left(\f32\zeta^4 -\f52\zeta^3\eta +\f52\zeta \eta^3-\f32\eta^4\right)\chi_{2}(\eta,\zeta)1_{\zeta<\eta}.
\end{cases}
\end{equation*}
The solutions of the above system are given by
\begin{align*}
\mathfrak{a}^a_{bal}=&\f{\zeta(-12\eta^6+51\eta^5\zeta-121\eta^4\zeta^2+147\eta^3\zeta^3-95\eta^2\zeta^4+25\eta\zeta^5+5\zeta^6)}{2\eta(4\eta^6-12\eta^5\zeta+37\eta^4\zeta-54\eta^3\zeta^3+75\eta^2\zeta^4-50\eta\zeta^5+25\zeta^6)}\chi_{2}(\eta, \zeta)1_{\zeta<\eta},\\
\mathfrak{b}^a_{bal}=&\f{2\eta^7-13\eta^6\zeta+22\eta^5\zeta^2-28\eta^4\zeta^3-8\eta^3\zeta^4+30\eta^2\zeta^5-30\eta\zeta^6+5\zeta^7}{2\eta(4\eta^6-12\eta^5\zeta+37\eta^4\zeta-54\eta^3\zeta^3+75\eta^2\zeta^4-50\eta\zeta^5+25\zeta^6)}\chi_{2}(\eta, \zeta)1_{\zeta<\eta},\\
\mathfrak{c}^a_{bal}=&\f{4\eta^4-11\eta^2\zeta+24\eta^2\zeta^2-16\eta\zeta^3+9\zeta^4}{4\eta^6-12\eta^5\zeta+37\eta^4\zeta-54\eta^3\zeta^3+75\eta^2\zeta^4-50\eta\zeta^5+25\zeta^6}\chi_{2}(\eta, \zeta)1_{\zeta<\eta},\\
\mathfrak{d}^a_{bal}=&-\f{10\eta^7-35\eta^6\zeta+85\eta^5\zeta^2-125\eta^4\zeta^3+133\eta^3\zeta^4-109\eta^2\zeta^5+59\eta\zeta^6-18\zeta^7}{2(4\eta^6-12\eta^5\zeta+37\eta^4\zeta-54\eta^3\zeta^3+75\eta^2\zeta^4-50\eta\zeta^5+25\zeta^6)}\chi_{2}(\eta, \zeta)1_{\zeta<\eta}.
\end{align*}
The resulting expressions for the symbols confirm that the normal form corrections $(\tilde{\bw}, \tilde{R})$ satisfy the estimate:
\begin{equation*}
    \|(\tilde{\bw}, \tilde{R})\|_{\H^s} \lesssim \mathcal{A}_{0}\|(\bw,R)\|_{\H^s}.
\end{equation*}
Moreover, $(\tilde{\bw}, \tilde{R})$ solve equations 
\begin{align*}
&T_{D_t}\tilde{\bw} + T_{b_\alpha}\tilde{\bw} + \partial_\alpha T_{1-\bar{Y}}T_{1+\bw} \tilde{R} \\
=&  T_{1- \bar{Y}}\p_\al \Pi(\bw, R)+ T_{1-Y}\partial_\alpha \nP \Pi( \bar{R}, \bw) +T_{(1-\bar{Y})^2(1+\bw)}\partial_\alpha \nP\Pi(\bar{\bw}, R) + G,
\end{align*}    
and 
\begin{align*}
&T_{D_t}\tilde{R}+ T_{b_\alpha}\tilde{R} -iT_{J^{-\f32}(1-Y)^2}\p^4_{\al}\tilde{\bw}+\f{15}{2}iT_{J^{-\f32}(1-Y)^3}T_{\bw_{\al}}\p^{3}_{\al}\tilde{\bw}\\
&+\f52iT_{J^{-\f52}(1-Y)}T_{\bar{\bw}_{\al}}\p^3_{\al}\tilde{\bw}+10iT_{J^{-\f32}(1-Y)^3}T_{\bw_{\al\al}}\tilde{\bw}_{\al\al}+\f{15}{2}iT_{J^{-\f32}(1-Y)^3}T_{\p^3_{\al}\bw}\tilde{\bw}_{\al}\\
&-\f52 iT_{J^{-\f52}(1-Y)}T_{\p^3_{\al}\bar{\bw}}
\tilde{\bw}_{\al} +\f72iT_{(1-Y)^3J^{-\f32}}T_{\p^4_{\al}\bw}\tilde{\bw}- \f32i T_{J^{-\f52}(1-Y)}T_{\p^4_{\al}\bar{\bw}} \tilde{\bw} \\
&-\f{65}{2}iT_{J^{-\f32}(1-Y)^4\bw_{\al}^2}\tilde{\bw}_{\al\al}-10iT_{J^{-\f52}(1-Y)^2|\bw_{\al}|^2}\tilde{\bw}_{\al\al}-\f{5}{2}iT_{J^{-\f72}\bar{\bw}_{\al}^2}\tilde{\bw}_{\al\al} \\
=&T_{1-\bar{Y}}\Pi(R_\alpha, R) + T_{1-Y}\nP \p_\al\Pi(\bar{R}, R) +\f72i T_{(1-Y)^3J^{-\f32}}\Pi(\p^4_{\al}\bw,\bw) \\
&+\f{15}{2}iT_{J^{-\f32}(1-Y)^3}\Pi(\p^3_{\al}\bw,\bw_{\al}) +5iT_{J^{-\f32}(1-Y)^3}\Pi(\bw_{\al\al},\bw_{\al\al})\\
&+\f32i T_{J^{-\f52}(1-Y)}\nP\Pi(\bar{\bw}, \p^4_{\al}\bw)+\f52i T_{J^{-\f52}(1-Y)}\nP\Pi(\bar{\bw}_{\al}, \p^3_{\al}\bw)\\
&-\f52i T_{J^{-\f52}(1-Y)}\nP\Pi(\p^3_{\al}\bar{\bw}, \bw_{\al}) -\f32i T_{J^{-\f52}(1-Y)}\nP\Pi(\p^4_{\al}\bar{\bw}, \bw)   + K.
\end{align*}    

Thus, we define $(\bw_{NF}, R_{NF}): = (\bw, R)+(\tilde{\bw}, \tilde{R})$, which satisfy the norm equivalence
\begin{equation} \label{DiffNfBound}
    \|(\bw_{NF}-\bw , R_{NF}-R) \|_{\H^s} \lesssim \CalAZ \|(\bw, R)\|_{\H^s}, \quad s>0.
\end{equation}
Moreover, the pair $(\bw_{NF}, R_{NF})$ solves the following linear paradifferential system with perturbative source terms.
\begin{equation} \label{LinearHatEqn}
 \begin{cases}
&T_{D_t}\hat{w} + T_{b_\alpha}\hat{w} + \partial_\alpha T_{1-\bar{Y}}T_{1+\bw} \hat{r}  = G,\\
&T_{D_t}\hat{r} + T_{b_\alpha}\hat{r} -iT_{J^{-\f32}(1-Y)^2}\p^4_{\al}\hat{w}+\f{15}{2}iT_{J^{-\f32}(1-Y)^3}T_{\bw_{\al}}\p^{3}_{\al}\hat{w}\\
&+\f52iT_{J^{-\f52}(1-Y)}T_{\bar{\bw}_{\al}}\p^3_{\al}\hat{w}+10iT_{J^{-\f32}(1-Y)^3}T_{\bw_{\al\al}}\hat{w}_{\al\al}+\f{15}{2}iT_{J^{-\f32}(1-Y)^3}T_{\p^3_{\al}\bw}\hat{w}_{\al}\\
&-\f52 iT_{J^{-\f52}(1-Y)}T_{\p^3_{\al}\bar{\bw}}
\hat{w}_{\al} +\f72iT_{(1-Y)^3J^{-\f32}}T_{\p^4_{\al}\bw}\hat{w}- \f32i T_{J^{-\f52}(1-Y)}T_{\p^4_{\al}\bar{\bw}} \hat{w}\\
&-\f{65}{2}iT_{J^{-\f32}(1-Y)^4\bw_{\al}^2}\hat{w}_{\al\al}-10iT_{J^{-\f52}(1-Y)^2|\bw_{\al}|^2}\hat{w}_{\al\al}-\f{5}{2}iT_{J^{-\f72}\bar{\bw}_{\al}^2}\hat{w}_{\al\al} = K.
\end{cases}
\end{equation}

\subsection{Further reduction to the paradifferential homogeneous linearized flow} \label{s:ParaFlowReduct}
Having reduced the differentiated hydroelastic waves \eqref{HF19} to the system \eqref{LinearHatEqn}, we now proceed to the derivation of the modified energy estimate for \eqref{LinearHatEqn}.
Rather than analyzing  \eqref{LinearHatEqn} directly, we reduce it to the paradifferential homogeneous linearized flow \eqref{ParadifferentialFlow} with perturbative source terms, and obtain the $\H^s$ modified energy estimate for \eqref{ParadifferentialFlow}. 

Let the pair $(\hat{w}, \hat{r})$ be a solution of the system \eqref{LinearHatEqn},  we consider
\begin{equation} \label{wrHatRelation}
    (w,r): = (\partial_\alpha^{-1}\hat{w}, \partial_\alpha^{-1}T_{1+\bw}\hat{r} - \partial_\alpha^{-1}T_{R_\alpha}\partial_\alpha^{-1}\hat{w}),
\end{equation}
so that $(w_\alpha, r_\alpha) = (\hat{w}, T_{1+\bw}\hat{r} - T_{R_\alpha}\partial_\alpha^{-1}\hat{w})$, and
\begin{equation} \label{wrInvertible}
   \|(w_\alpha, r_\alpha) -(\hat{w}, \hat{r})\|_{\H^s} \lesssim \CalAZ \| (\hat{w}, \hat{r})\|_{\H^s}, \quad \forall s\in \mathbb{R}.
\end{equation}
We then differentiate the left-hand side of the paradifferential flow \eqref{ParadifferentialFlow}, and get for the first equation
\begin{equation*}
    \partial_\alpha(T_{D_t}w+T_{1-\bar{Y}} r_\alpha + T_{(1-\bar{Y})R_\alpha}w) = T_{D_t}\hat{w} +T_{b_\alpha} \hat{w} + \partial_\alpha T_{1-\bar{Y}}T_{1+\bw} \hat{r} = \tilde{G}.
\end{equation*}
As for the second equation of \eqref{ParadifferentialFlow}, we again take the $\alpha$-derivative and use the material derivative of $\bw$ and $R$ to write
\begin{align*}
&\quad\p_\al(T_{D_t}r-i\mathcal{L}_{para}w) \\
=&  T_{D_t}(T_{1+\bw}\hat{r} - T_{R_\alpha}\partial_\alpha^{-1}\hat{w}) + T_{b_\alpha}(T_{1+\bw}\hat{r} - T_{R_\alpha}\partial_\alpha^{-1}\hat{w}) -i\p_\al T_{(1-Y)J^{-\f32}}\p_{\al}^3 \hat{w}\\
&+5i\p_\al T_{(1-Y)^2J^{-\f32}\mathbf{W}_{\al}} \hat{w}_{\al \al}+i\p_\al T_{J^{-\f52}\bar{\bw}_{\al}} \hat{w}_{\al \al}+5i \p_\al T_{(1-Y)^2J^{-\f32}\mathbf{W}_{\al\al}}\hat{w}_{\al}\\
&-i\p_\al T_{J^{-\f52}\bar{\mathbf{W}}_{\al\al}}\hat{w}_{\al}-15i\p_\al T_{(1-Y)^3J^{-\f32}\mathbf{W}^2_{\al}}\hat{w}_{\al} +\f52 i\p_\al T_{(1-Y)^2J^{-\f32}\p^3_{\al}\mathbf{W}}\hat{w}\\
&-\f32 i\p_\al T_{J^{-\f52}\p^3_{\al}\bar{\mathbf{W}}}\hat{w}+ i\p_\al T_{(1-Y)^2J^{-\f32}\p^4_{\al}\mathbf{W}}\p_\al^{-1} \hat{w} \\
=& T_{1+\bw}(T_{D_t}+T_{b_\alpha})\hat{r} - T_{R_\alpha}\partial_\alpha^{-1}(T_{D_t} + T_{b_\alpha})\hat{w} + T_{D_t \bw}\hat{r} - T_{\partial_\alpha D_t R}\partial_\alpha^{-1} \hat{w} \\
&-i T_{(1-Y)J^{-\f32}}\p_{\al}^4 \hat{w} + \f52  i T_{(1-Y)^2J^{-\f32}\bw_\al}\p_{\al}^3 \hat{w}  + \f32  i T_{J^{-\f52}\bar{\bw}_\al}\p_{\al}^3 \hat{w} +5i T_{(1-Y)^2J^{-\f32}\mathbf{W}_{\al}} \p_\al^3\hat{w} \\
&+5i T_{(1-Y)^2J^{-\f32}\mathbf{W}_{\al\al}} \hat{w}_{\al \al} -\f{35}{2}iT_{(1-Y)^3J^{-\f32}\bw_\al^2} \hat{w}_{\al \al} - \f{15}{2}iT_{(1-Y)J^{-\f52}|\bw_\al|^2}\hat{w}_{\al \al} \\
&+i T_{J^{-\f52}\bar{\bw}_{\al}} \p_\al^3\hat{w} +i T_{J^{-\f52}\bar{\bw}_{\al \al}} \hat{w}_{\al \al} - \f52i T_{J^{-\f52}(1-Y) |\bw_{\al}|^2} \hat{w}_{\al \al} - \f52i T_{J^{-\f52}(1-\bar{Y}) \bar{\bw}_{\al}^2} \hat{w}_{\al \al} \\
&+5i  T_{(1-Y)^2J^{-\f32}\mathbf{W}_{\al\al}}\hat{w}_{\al \al} +5i  T_{(1-Y)^2J^{-\f32}\p_\al^3\mathbf{W}}\hat{w}_{\al} -i T_{J^{-\f52}\bar{\mathbf{W}}_{\al\al}}\hat{w}_{\al\al} -i T_{J^{-\f52}\p_\al^3\bar{\mathbf{W}}}\hat{w}_{\al} \\
& -15i T_{(1-Y)^3J^{-\f32}\mathbf{W}^2_{\al}}\hat{w}_{\al \al} +\f52 i T_{(1-Y)^2J^{-\f32}\p^3_{\al}\mathbf{W}}\hat{w}_\al +\f52 i T_{(1-Y)^2J^{-\f32}\p^4_{\al}\mathbf{W}}\hat{w} \\
& -\f32 i  T_{J^{-\f52}\p^3_{\al}\bar{\mathbf{W}}}\hat{w}_\al -\f32 i T_{J^{-\f52}\p^4_{\al}\bar{\mathbf{W}}}\hat{w} + i T_{(1-Y)^2J^{-\f32}\p^4_{\al}\mathbf{W}} \hat{w} + i T_{((1-Y)^2J^{-\f32}\p^4_{\al}\mathbf{W})_\al}\p_\al^{-1} \hat{w} + \tilde{K} \\
=& T_{R_\al} T_{1-\bar{Y}}T_{1+\bw} \hat{r} - T_{(1+\bw)(1-\bar{Y})R_\al-(1+\bw)M} \hat{r}  +iT_{1+\bw}T_{J^{-\f32}(1-Y)^2}\p^4_{\al}\hat{w}\\
&-\f{15}{2}iT_{1+\bw}T_{J^{-\f32}(1-Y)^3}T_{\bw_{\al}}\p^{3}_{\al}\hat{w}-\f52iT_{1+\bw}T_{J^{-\f52}(1-Y)}T_{\bar{\bw}_{\al}}\p^3_{\al}\hat{w}-10iT_{1+\bw}T_{J^{-\f32}(1-Y)^3}T_{\bw_{\al\al}}\hat{w}_{\al\al}\\
&-\f{15}{2}iT_{1+\bw}T_{J^{-\f32}(1-Y)^3}T_{\p^3_{\al}\bw}\hat{w}_{\al}+\f52 iT_{1+\bw}T_{J^{-\f52}(1-Y)}T_{\p^3_{\al}\bar{\bw}}
\hat{w}_{\al} -\f72iT_{1+\bw}T_{(1-Y)^3J^{-\f32}}T_{\p^4_{\al}\bw}\hat{w}\\
&+ \f32i T_{1+\bw}T_{J^{-\f52}(1-Y)}T_{\p^4_{\al}\bar{\bw}} \hat{w}+\f{65}{2}iT_{1+\bw}T_{J^{-\f32}(1-Y)^4\bw_{\al}^2}\hat{w}_{\al\al}+10iT_{1+\bw}T_{J^{-\f52}(1-Y)^2|\bw_{\al}|^2}\hat{w}_{\al\al}\\
&+\f{5}{2}iT_{1+\bw}T_{J^{-\f72}\bar{\bw}_{\al}^2}\hat{w}_{\al\al}-i T_{(1-Y)J^{-\f32}}\p_{\al}^4 \hat{w} + \f{15}{2}  i T_{(1-Y)^2J^{-\f32}\bw_\al}\p_{\al}^3 \hat{w}  + \f32  i T_{J^{-\f52}\bar{\bw}_\al}\p_{\al}^3 \hat{w}  \\
&+10i T_{(1-Y)^2J^{-\f32}\mathbf{W}_{\al\al}} \hat{w}_{\al \al} -\f{65}{2}iT_{(1-Y)^3J^{-\f32}\bw_\al^2} \hat{w}_{\al \al} - 10iT_{(1-Y)J^{-\f52}|\bw_\al|^2}\hat{w}_{\al \al} +i T_{J^{-\f52}\bar{\bw}_{\al}} \p_\al^3\hat{w} \\
&- \f52i T_{J^{-\f52}(1-\bar{Y}) \bar{\bw}_{\al}^2} \hat{w}_{\al \al} +5i  T_{(1-Y)^2J^{-\f32}\p_\al^3\mathbf{W}}\hat{w}_{\al}  -i\f52 T_{J^{-\f52}\p_\al^3\bar{\mathbf{W}}}\hat{w}_{\al}   +\f52 i T_{(1-Y)^2J^{-\f32}\p^3_{\al}\mathbf{W}}\hat{w}_\al \\
&+\f72 i T_{(1-Y)^2J^{-\f32}\p^4_{\al}\mathbf{W}}\hat{w}  -\f32 i T_{J^{-\f52}\p^4_{\al}\bar{\mathbf{W}}}\hat{w} +  T_{(i(1-Y)^2J^{-\f32}\p^4_{\al}\mathbf{W} -D_t R)_\al}\p_\al^{-1} \hat{w} + \tilde{K} \\
=& \tilde{K} ,
\end{align*}
where $(\tilde{G}, \tilde{K})$ belong to perturbative terms that satisfy 
\begin{equation*}
 \|(\tilde{G}, \tilde{K})\|_{\H^{s}} \lesssim_\CalAZ \ASSharp \|(w,r)\|_{\H^{s+1}}.
\end{equation*}

Therefore, we have shown the following result that connects variables $(w,r)$ and $(\hat{w}, \hat{r})$.

\begin{proposition} \label{t:HatwrHsWellposed}
Given $(\hat{w}, \hat{r})$ that solve the system of equation \eqref{LinearHatEqn}, then $(w,r)$ defined in \eqref{wrHatRelation} 
 solve the system of paradifferential equation \eqref{ParadifferentialFlow} for any $s\in \mathbb{R}$.
Moreover, the transformation is invertible in the sense of \eqref{wrInvertible}.
\end{proposition}

Hence, according to the previous Proposition, the $\H^{s+1}$ modified energy estimate for \eqref{ParadifferentialFlow} implies the $\H^s$ modified energy estimate of \eqref{LinearHatEqn} for $s> -1$.

\subsection{$\H^s$ energy estimate of the homogeneous paradifferential system} \label{s:HsHomogeneousSys}
In this section, we generalize the result in Proposition \ref{t:wellposedflow} to $s \geq 0$.
In other words, we prove the following modified energy estimate of \eqref{ParadifferentialFlow} for $s\geq 0$.
\begin{proposition} \label{t:Hswellposedflow}
Assume that $\mathcal{A}_0
\lesssim 1$ and $\mathcal{A}_{\sharp,\f74}  \in L^2_t([0,T])$ for some time $T>0$, then if $(w,r)$ solve the homogeneous paradifferential system \eqref{ParadifferentialFlow} on $[0,T]$, there exists an  energy functional $E^{s,para}_{lin}(w, r)$ such that on $[0,T]$ for $s\geq 0$:
\begin{enumerate}
\item Norm equivalence:
\begin{equation*}
    E^{s,para}_{lin}(w, r) = (1+O(\CalAZ)) \|(w,r)\|_{\H^s}^2.
\end{equation*}
\item The time derivative of $E^{s,para}_{lin}(w, r)$ is bounded by
\begin{equation*}
    \frac{d}{dt}  E^{s,para}_{lin}(w, r) \lesssim_{\CalAZ} \ASSharp \|(w,r)\|_{\H^s}^2.
\end{equation*}
\end{enumerate}  
\end{proposition}

We apply the operator $\langle D \rangle^s$ to the homogeneous paradifferential system \eqref{ParadifferentialFlow}.
Then $(w^s, r^s): = \langle D\rangle^s(w,r)$ solve the paradifferential equations
\begin{equation}  \label{HsFlowSys}
\begin{cases}
T_{D_t}w^s +T_{1-\bar{Y}}r_{\al}^s+T_{T_{1-\bar{Y}}R_{\al}}w^s= \mathcal{G}^s_0 +G_s,\\
T_{D_t}r^s-i\mathcal{L}_{para}w^s = \mathcal{K}^s_0 +K_s,
\end{cases}
\end{equation}
where the $(G_s, K_s)$ are perturbative source terms that satisfy
\begin{equation}
    \| (G_s, K_s)\|_{\mathcal{H}^0} \lesssim_\CalAZ \ASSharp\|(w,r)\|_{\H^s}, \label{GsHsSource}
\end{equation}
and source terms $(\mathcal{G}^s_0, \mathcal{K}^s_0)$ are given by the sum of paradifferential commutator terms:
\begin{align*}
    \mathcal{G}^s_0 =& L(b_\alpha, w^s)-L(\bar{Y}_\alpha, r^s) + L( (T_{1-\bar{Y}}\bar{R}_\alpha)_\alpha,\partial_\alpha^{-1}w^s) ,\\
    \mathcal{K}^s_0 =& L(b_\alpha, r^s) -iL( [(1-Y)J^{-\f32}]_\al, \p_\al^3 w^s) +5iL([(1-Y)^2J^{-\f32}\bw_\al]_\al, \p_\al^2 w^s)\\
    &+iL((J^{-\f52}\bar{\bw}_\al)_\al, \p_\al^2 w^s) +5iL([(1-Y)^2J^{-\f32}\bw_{\al\al}]_\al,  w^s_\al)- iL((J^{-\f52}\bar{\bw}_{\al\al})_\al, w^s_\al) \\
    &-15i L([(1-Y)^3J^{-\f32}\bw_\al^2]_\al,  w^s_\al) + \f52i L([(1-Y)^2J^{-\f32}\p_\al^3\bw]_\al,  w^s) \\
    &-\f32iL([J^{-\f52}\p_\al^3 \bar{\bw}]_\al, w^s) + iL([(1-Y)^2J^{-\f32}\p_\al^4\bw]_\al,  \p_\al^{-1}w^s).
\end{align*}
Here, $L$ denotes the zero order paradifferential commutator
\begin{equation} \label{LDef}
L(f_\alpha, u) = -[\langle D\rangle^s, T_f]\partial_\alpha \langle D\rangle^{-s}u \approx -sT_{f_\alpha} u + \text{lower order terms}.
\end{equation}

To have a clear view of $(\mathcal{G}^s_0, \mathcal{K}^s_0)$, we rewrite them as the sum of non-perturbative terms plus perturbative source terms $(G_s, K_s)$.
\begin{align*}
\mathcal{G}^s_0 &= L(T_{1-\bar{Y}}R_\alpha +T_{1-Y}\bar{R}_\alpha, w^s)-L((1-\bar{Y})^2\bar{\bw}_\alpha, r^s)+ L(T_{1-\bar{Y}}R_{\alpha \alpha}, \partial_\alpha^{-1}w^s)+G_s ,\\
 \mathcal{K}^s_0 &= L(T_{1-\bar{Y}}R_\alpha +T_{1-Y}\bar{R}_\alpha, r^s) +\f52iL( (1-Y)^2J^{-\f32} \bw_\al, \p_\al^3 w^s) +\f32iL(J^{-\f52} \bar{\bw}_\al, \p_\al^3 w^s)\\
 & +5iL((1-Y)^2J^{-\f32}\bw_{\al\al}, \p_\al^2 w^s) +iL(J^{-\f52}\bar{\bw}_{\al\al}, \p_\al^2 w^s)- \f{35}{2}iL((1-Y)^3J^{-\f32}\bw_{\al}^2, \p_\al^2 w^s) \\
 &-\f52iL((1-\bar{Y})J^{-\f52}\bar{\bw}_{\al}^2, \p_\al^2 w^s)- 10iL((1-Y)J^{-\f52}|\bw_{\al}|^2, \p_\al^2 w^s) \\
 &+5iL((1-Y)^2J^{-\f32}\p_\al^3\bw,  w^s_\al)- iL(J^{-\f52}\p_\al^3\bar{\bw}, w^s_\al)  + \f52i L((1-Y)^2J^{-\f32}\p_\al^4\bw,  w^s) \\
 &-\f32iL(J^{-\f52}\p_\al^4 \bar{\bw}, w^s) + iL((1-Y)^2J^{-\f32}\p_\al^5\bw,  \p_\al^{-1}w^s)+K_s.
\end{align*}

We further apply the paradifferential conjugation  $T_{J^{-\f{s}{2}}}$ to the system \eqref{HsTildeFlowSys}, and set $(\tilde{w}^s, \tilde{r}^s):  = T_{J^{-\f{s}{2}}}(w^s,r^s) = T_{J^{-\f{s}{2}}}\langle D\rangle^s(w,r)$.
One can think of \eqref{ParadifferentialFlow} as the quasilinear dispersive equation
\begin{equation*}
    \Big(\p_t + i T_{J^{-\f54}}|D|^\f52 \Big)w = \text{lower order nonlinear terms}. 
\end{equation*}
Since the leading part of the elliptic operator $\Big(T_{J^{-\f54}}|D|^\f52 \Big)^{\f{2s}{5}} \approx T_{J^{-\f{s}{2}}}|D|^s$,  the paradifferential conjugation $T_{J^{-\f{s}{2}}}$ is added to cancel some of the lower order non-perturbative terms.

We first compute each derivatives of $w^s$, writing them as the sum of derivatives of $\tilde{w}_s$.
\begin{align*}
&\p_{\al}w^s=\p_{\al}T_{J^{\f{s}{2}}}\tilde{w}^s+K_s=T_{\f{s}{2}}\tilde{w}^s_{\al}+T_{\p_{\al}J^{\f{s}{2}}}\tilde{w}^s+K_s,\\
&\p^2_{\al}w^s=\p^2_{\al}T_{J^{\f{s}{2}}}\tilde{w}^s+K_s=T_{J^{\f{s}{2}}}\tilde{w}^s_{\al\al}+2T_{\p_{\al}J^{\f{s}{2}}}\tilde{w}^s_{\al}+T_{\p^2_{\al}J^{\f{s}{2}}}\tilde{w}^s+K_s,\\
&\p^3_{\al}w^s=\p^3_{\al}T_{J^{\f{s}{2}}}\tilde{w}^s+K_s=T_{J^{\f{s}{2}}}\p^3_{\al}\tilde{w}^s+3T_{\p_{\al}J^{\f{s}{2}}}\tilde{w}^s_{\al\al}+3T_{\p^2_{\al} J^{\f{s}{2}}}\tilde{w}^s_{\al}+T_{\p^3_{\al}J^{\f{s}{2}}}\tilde{w}^s+K_s,\\
&\p^4_{\al}w^s=\p^4_{\al}T_{J^{\f{s}{2}}}\tilde{w}^s+K_s=T_{J^{\f{s}{2}}}\p^4_{\al}\tilde{w}^s+4T_{\p_{\al}J^{\f{s}{2}}}\p^3_{\al}\tilde{w}^s+6T_{\p^2_{\al}J^{\f{s}{2}}}\tilde{w}^s_{\al\al}+4T_{\p^3_{\al}J^{\f{s}{2}}}\tilde{w}^s_{\al}+T_{\p^4_{\al}J^{\f{s}{2}}}\tilde{w}^s_{\al}+K_s.
\end{align*}
One can also check that
\begin{align*}
\p_{\al}J^{\f{s}{2}}=&\f{s}{2} J^{\f{s}{2}}((1-Y)\bw_{\al}+(1-\bar{Y})\bar{\bw}_{\al}),\\
\p^2_{\al}J^{\f{s}{2}}=&\f{s}{2} J^{\f{s}{2}}((1-Y)\bw_{\al\al}+(1-\bar{Y})\bar{\bw}_{\al\al}) +\f{s}{2}\left(\f{s}{2}-1\right)J^{\f{s}{2}}(1-Y)^2\bw^2_{\al}\\
&+\f{s^2}{2} J^{\f{s}{2}-1}|\bw_{\al}|^2+\f{s}{2}\left(\f{s}{2}-1 \right)J^{\f{s}{2}}(1-\bar{Y})^2\bar{\bw}^2_{\al},\\
\p^3_{\al}J^{\f{s}{2}}=&\f{s}{2} J^{\f{s}{2}}((1-Y)\p^3_{\al}\bw+(1-\bar{Y})\p^3_{\al}\bar{\bw})+\f{3s}{2}\left(\f{s}{2}-1 \right)J^{\f{s}{2}}(1-Y)^2\bw_{\al}\bw_{\al\al}\\
&+ \f{3s^2}{4} J^{\f{s}{2}-1}\bar{\bw}_{\al}\bw_{\al\al}+\f{3s^2}{4} J^{\f{s}{2}-1}(1-\bar{Y})\bw_{\al}\bar{\bw}_{\al\al}+\f{3s}{2}\left(\f{s}{2}-1 \right)J^{\f{s}{2}}(1-\bar{Y})^2\bar{\bw}_{\al}\bar{\bw}_{\al\al} \\
&+\text{cubic terms with at most one derivative on $\bw$ or $\bar{\bw}$},\\
\p^4_{\al}J^{\f{s}{2}}=&\f{s}{2} J^{\f{s}{2}}((1-Y)\p^4_{\al}\bw+(1-\bar{Y})\p^4_{\al}\bar{\bw})+\left(s^2-2s\right)J^{\f{s}{2}}(1-Y)^2\bw_{\al}\p^3_{\al}\bw\\
&+s^2 J^{\f{s}{2}-1}\bar{\bw}_{\al}\p^3_{\al}\bw+s^2 J^{\f{s}{2}-1}(1-\bar{Y})\bw_{\al}\p^3_{\al}\bar{\bw} +(s^2 -2s)J^{\f{s}{2}}(1-\bar{Y})^2\bar{\bw}_{\al}\p^3_{\al}\bar{\bw}\\
&+ \text{cubic terms with at most two derivatives on $\bw$ or $\bar{\bw}$}.
\end{align*}
For non-elastic terms, following Section 5 in \cite{AIT}, we have
\begin{align*}
T_{J^{-\f{s}{2}}}T_{D_t}w^s &=T_{D_t}\tilde{w}^s-\f{s}{2} T_{b_{\al}}\tilde{w}^s+G_s,\\
T_{J^{-\f{s}{2}}}T_{D_t}r^s&=T_{D_t}\tilde{r}^s-\f{s}{2} T_{b_{\al}}\tilde{r}^s+K_s,\\
T_{J^{-\f{s}{2}}}T_{1-\bar{Y}}\p_{\al}r^s&=T_{1-\bar{Y}}\p_{\al}\tilde{r}^s+\f{s}{2}(T_{(1-\bar{Y})}T_{(1+\bw)Y_{\al}}+T_{\bar{Y}_{\al}})\tilde{r}^s+G_s\\
&= T_{1-\bar{Y}}\p_{\al}\tilde{r}^s+  \f{s}{2} T_{J^{-1}\bw_{\al}} \tilde{r}^s+\f{s}{2} T_{(1-\bar{Y})^2\bar{\bw}_{\al}}\tilde{r}^s +G_s ,\\
T_{J^{-\f{s}{2}}}T_{(1-\bar{Y})R_{\al}}w^s&=T_{(1-\bar{Y})R_{\al}}\tilde{w}^s+G_s.
\end{align*}
As for linearized elastic terms,
\begin{align*}
&T_{J^{-\f{s}{2}}}\mathcal{L}_{para}w^s\\
=&T_{(1-Y)J^{-\f32}}\p_{\al}^4\tilde{w}^s+4T_{(1-Y)J^{-\f32}J^{-\f{s}{2}}\p_{\al}J^{\f{s}{2}}}\p^3_{\al}\tilde{w}^s+6T_{(1-Y)J^{-\f32}J^{-\f{s}{2}}\p^2_{\al}J^{\f{s}{2}}}\tilde{w}^s_{\al\al}\\
&+4T_{(1-Y)J^{-\f32}J^{-\f{s}{2}}\p^3_{\al}J^{\f{s}{2}}}\tilde{w}^s_{\al}+T_{(1-Y)J^{-\f32}J^{-\f{s}{2}}\p^4_{\al}J^{\f{s}{2}}}\tilde{w}^s-5T_{(1-Y)^2J^{-\f32}\mathbf{W}_{\al}}\p^3_{\al}\tilde{w}^s\\
&-15T_{(1-Y)^2J^{-\f32}\mathbf{W}_{\al}J^{-\f{s}{2}}\p_{\al}J^{\f{s}{2}}}\tilde{w}^s_{\al\al}-15T_{(1-Y)^2J^{-\f32}\mathbf{W}_{\al}J^{-\f{s}{2}}\p^2_{\al}J^{\f{s}{2}}}\tilde{w}^s_{\al}-5T_{(1-Y)^2J^{-\f32}\mathbf{W}_{\al}J^{-\f{s}{2}}\p^3_{\al}J^{\f{s}{2}}}\tilde{w}^s\\
&-T_{J^{-\f52}\bar{\bw}_{\al}}\p^3_{\al}\tilde{w}^s-3T_{J^{-\f52}\bar{\bw}_{\al}J^{-\f{s}{2}}\p_{\al}J^{\f{s}{2}}}\tilde{w}^s_{\al\al}-3T_{J^{-\f52}\bar{\bw}_{\al}J^{-\f{s}{2}}\p^2_{\al}J^{\f{s}{2}}}\tilde{w}^s_{\al}-T_{J^{-\f52}\bar{\bw}_{\al}J^{-\f{s}{2}}\p^3_{\al}J^{\f{s}{2}}}\tilde{w}^s\\
&-5T_{(1-Y)^2J^{-\f32}\mathbf{W}_{\al\al}}\tilde{w}^s_{\al\al}-10T_{(1-Y)^2J^{-\f32}\mathbf{W}_{\al\al}J^{-\f{s}{2}}\p_{\al}J^{\f{s}{2}}}\tilde{w}^s_{\al}-5T_{(1-Y)^2J^{-\f32}\mathbf{W}_{\al\al}J^{-\f{s}{2}}\p^2_{\al}J^{\f{s}{2}}}\tilde{w}^s
\\
&+T_{J^{-\f52}\bar{\mathbf{W}}_{\al\al}}\tilde{w}^s_{\al\al}+2T_{J^{-\f52}\bar{\mathbf{W}}_{\al\al}J^{-\f{s}{2}}\p_{\al}J^{\f{s}{2}}}\tilde{w}^s_{\al}+T_{J^{-\f52}\bar{\mathbf{W}}_{\al\al}J^{-\f{s}{2}}\p^2_{\al}J^{\f{s}{2}}}\tilde{w}^s +15T_{J^{-\f52}\bar{\mathbf{W}}_{\al\al}}\tilde{w}^s_{\al\al} \\
&+30T_{(1-Y)^3J^{-\f32}\mathbf{W}^2_{\al}J^{-\f{s}{2}}\p_{\al}J^{\f{s}{2}}}\tilde{w}^s_{\al}+15T_{(1-Y)^3J^{-\f32}\mathbf{W}^2_{\al}J^{-\f{s}{2}}\p^2_{\al}J^{\f{s}{2}}}\tilde{w}^s-\f52T_{(1-Y)^2J^{-\f32}\p^3_{\al}\mathbf{W}}\tilde{w}^s_{\al}\\
&-\f52T_{(1-Y)^2J^{-\f32}\p^3_{\al}\mathbf{W}J^{-\f{s}{2}}\p_{\al}J^{\f{s}{2}}}\tilde{w}^s+\f32T_{J^{-\f52}\p^3_{\al}\bar{\mathbf{W}}}\tilde{w}^s_{\al}+\f32T_{J^{-\f52}\p^3_{\al}\bar{\mathbf{W}}J^{-\f{s}{2}}\p_{\al}J^{\f{s}{2}}}\tilde{w}^s\\
&-T_{(1-Y)^2J^{-\f32}\p^4_{\al}\mathbf{W}}\tilde{w}^s+K_s\\
=&T_{(1-Y)J^{-\f32}}\p_{\al}^4\tilde{w}^s+2s T_{(1-Y)^2J^{-\f32}\bw_{\al}}\p^3_{\al}\tilde{w}^s+2s T_{J^{-\f52}\bar{\bw}_{\al}}\p^3_{\al}\tilde{w}^s+3s T_{(1-Y)^2J^{-\f32}\p^2_{\al}\bw}\tilde{w}^s_{\al\al}\\
&+3s T_{J^{-\f52}\p^2_{\al}\bar{\bw}}\tilde{w}^s_{\al\al}+3s\left(\f{s}{2}-1\right)T_{(1-Y)^3J^{-\f32}\bw^2_{\al}}\tilde{w}^s_{\al\al}+3s^2T_{(1-Y)J^{-\f52}|\bw_{\al}|^2}\tilde{w}^s_{\al\al} \\
&+3s\left(\f{s}{2}-1\right)T_{(1-\bar{Y})J^{-\f52}\bar{\bw}_{\al}^2}\tilde{w}^s_{\al\al}+2s T_{(1-Y)^2J^{-\f32}\p^3_{\al}\bw}\tilde{w}^s_{\al}+2s T_{J^{-\f52}\p^3_{\al}\bar{\bw}}\tilde{w}^s_{\al}\\
&+\f{s}{2} T_{(1-Y)^2J^{-\f32}\p^4_{\al}\bw}\tilde{w}^s +\f{s}{2} T_{J^{-\f52}\p^4_{\al}\bar{\bw}}\tilde{w}^s-5T_{(1-Y)^2J^{-\f32}\mathbf{W}_{\al}}\p^3_{\al}\tilde{w}^s-\f{15}{2}s T_{(1-Y)^3J^{-\f32}\mathbf{W}^2_{\al}}\tilde{w}^s_{\al\al} \\
&-\f{15}{2}s T_{(1-Y)J^{-\f52}|\mathbf{W}_{\al}|^2}\tilde{w}^s_{\al\al}-T_{J^{-\f52}\bar{\bw}_{\al}}\p^3_{\al}\tilde{w}^s-\f{3s}{2} T_{(1-Y)J^{-\f52}|\bw_{\al}|^2}\tilde{w}^s_{\al\al}-\f{3s}{2} T_{(1-\bar{Y})J^{-\f52}\bar{\bw}_{\al}^2}\tilde{w}^s_{\al\al} \\
&-5T_{(1-Y)^2J^{-\f32}\mathbf{W}_{\al\al}}\tilde{w}^s_{\al\al} +T_{J^{-\f52}\bar{\mathbf{W}}_{\al\al}}\tilde{w}^s_{\al\al} +15T_{J^{-\f52}\bar{\mathbf{W}}_{\al\al}}\tilde{w}^s_{\al\al}-\f52T_{(1-Y)^2J^{-\f32}\p^3_{\al}\mathbf{W}}\tilde{w}^s_{\al}\\
&+\f32T_{J^{-\f52}\p^3_{\al}\bar{\mathbf{W}}}\tilde{w}^s_{\al}-T_{(1-Y)^2J^{-\f32}\p^4_{\al}\mathbf{W}}\tilde{w}^s+K_s.
\end{align*}

For the action of $T_{J^{-\f{s}{2}}}$ to source terms $(\mathcal{G}^s_0, \mathcal{K}^s_0)$, we compute
\begin{align*}
T_{J^{-\f{s}{2}}} \mathcal{G}^s_0 =& L(b_\alpha, \tilde{w}^s)-L(\bar{Y}_\alpha, \tilde{r}^s) + L( (T_{1-\bar{Y}}\bar{R}_\alpha)_\alpha,\partial_\alpha^{-1}\tilde{w}^s) ,\\
T_{J^{-\f{s}{2}}} \mathcal{K}^s_0 =& L(b_\alpha, \tilde{r}^s) -iL( [(1-Y)J^{-\f32}]_\al,\p^3_{\al}\tilde{w}^s+3T_{J^{-\f{s}{2}}\p_{\al}J^{\f{s}{2}}}\tilde{w}^s_{\al\al}) +5iL([(1-Y)^2J^{-\f32}\bw_\al]_\al, \tilde{w}^s_{\al\al})\\
&+iL((J^{-\f52}\bar{\bw}_\al)_\al, \tilde{w}^s_{\al\al}) +5iL([(1-Y)^2J^{-\f32}\bw_{\al\al}]_\al,  \tilde{w}^s_\al)- iL((J^{-\f52}\bar{\bw}_{\al\al})_\al, \tilde{w}^s_\al) \\
    &-15i L([(1-Y)^3J^{-\f32}\bw_\al^2]_\al,  \tilde{w}^s_\al) + \f52i L([(1-Y)^2J^{-\f32}\p_\al^3\bw]_\al,  \tilde{w}^s_{\al}) \\
    &-\f32iL([J^{-\f52}\p_\al^3 \bar{\bw}]_\al, \tilde{w}^s_{\al}) + iL([(1-Y)^2J^{-\f32}\p_\al^4\bw]_\al,  \p_\al^{-1}\tilde{w}^s) + K_s.
\end{align*}

Collecting all the terms above and simplifying, we get that $(\tilde{w}^s, \tilde{r}^s)$ solve the system of paradifferential equations 
\begin{equation}  
\begin{cases}
T_{D_t}\tilde{w}^s +T_{1-\bar{Y}}\partial_{\al} \tilde{r}^s+T_{T_{1-\bar{Y}}R_{\al}}\tilde{w}^s= \mathcal{G}^s_1+ G_s,\\
T_{D_t}\tilde{r}^s-i\mathcal{L}_{para}\tilde{w}^s = \mathcal{K}^s_1+ K_s,
\end{cases}
\end{equation}
where non-perturbative source terms $(\mathcal{G}^s_1, \mathcal{K}^s_1)$ are given by
\begin{align*}
\mathcal{G}^s_1 =& L(T_{1-\bar{Y}}R_\alpha +T_{1-Y}\bar{R}_\alpha, \tilde{w}^s)-L((1-\bar{Y})^2\bar{\bw}_\alpha, \tilde{r}^s)+ L(T_{1-\bar{Y}}R_{\alpha \alpha}, \partial_\alpha^{-1}\tilde{w}^s)\\
&+\f{s}{2} T_{T_{1-\bar{Y}}R_{\al}}\tilde{w}^s +\f{s}{2} T_{T_{1-Y}\bar{R}_{\al}}\tilde{w}^s-\f{s}{2} T_{J^{-1}\bw_{\al}} \tilde{r}^s-\f{s}{2} T_{(1-\bar{Y})^2\bar{\bw}_{\al}}\tilde{r}^s+G_s ,\\ 
\mathcal{K}^s_1 =& L(T_{1-\bar{Y}}R_\alpha +T_{1-Y}\bar{R}_\alpha, \tilde{r}^s)+2si T_{(1-Y)^2J^{-\f32}\bw_{\al}}\p^3_{\al}\tilde{w}^s +2si T_{J^{-\f52}\bar{\bw}_{\al}}\p^3_{\al}\tilde{w}^s\\
&+3si T_{(1-Y)^2J^{-\f32}\p^2_{\al}\bw}\tilde{w}^s_{\al\al}+3si T_{J^{-\f52}\p^2_{\al}\bar{\bw}}\tilde{w}^s_{\al\al}+3s\left(\f{s}{2}-\f72\right)iT_{(1-Y)^3J^{-\f32}\bw^2_{\al}}\tilde{w}^s_{\al\al}\\
&+3s(s-3)iT_{(1-Y)J^{-\f52}|\bw_{\al}|^2}\tilde{w}^s_{\al\al}+3s \left(\f{s}{2}-\f32 \right)iT_{(1-\bar{Y})J^{-\f52}\bar{\bw}_{\al}^2}\tilde{w}^s_{\al\al}+2si T_{(1-Y)^2J^{-\f32}\p^3_{\al}\bw}\tilde{w}^s_{\al}\\
&+2si T_{J^{-\f52}\p^3_{\al}\bar{\bw}}\tilde{w}^s_{\al}+\f{s}{2}i T_{(1-Y)^2J^{-\f32}\p^4_{\al}\bw}\tilde{w}^s+\f{s}{2}i T_{J^{-\f52}\p^4_{\al}\bar{\bw}}\tilde{w}^s  +\f{s}{2} T_{T_{1-\bar{Y}}R_{\al}}\tilde{r}^s\\
&+\f{s}{2} T_{T_{1-Y}\bar{R}_{\al}}\tilde{r}^s+\f52iL( (1-Y)^2J^{-\f32} \bw_\al, \p_\al^3 \tilde{w}^s) +\f{15}{4}si L( (1-Y)^3J^{-\f32} \bw^2_\al, \p_\al^2 \tilde{w}^s)\\ 
&+6si L( (1-Y)J^{-\f52} |\bw_\al|^2, \p_\al^2 \tilde{w}^s)+\f32iL(J^{-\f52} \bar{\bw}_\al, \p_\al^3 \tilde{w}^s) +\f94 si L( (1-\bar{Y})J^{-\f52} \bar{\bw}^2_\al, \p_\al^2 \tilde{w}^s) \\ 
& +5iL((1-Y)^2J^{-\f32}\bw_{\al\al}, \p_\al^2 \tilde{w}^s) +iL(J^{-\f52}\bar{\bw}_{\al\al}, \p_\al^2 \tilde{w}^s)- \f{35}{2}iL((1-Y)^3J^{-\f32}\bw_{\al}^2, \p_\al^2 \tilde{w}^s) \\ 
&-\f52iL((1-\bar{Y})J^{-\f52}\bar{\bw}_{\al}^2, \p_\al^2 w^s)- 10iL((1-Y)J^{-\f52}|\bw_{\al}|^2, \p_\al^2 \tilde{w}^s) \\ 
&+5iL((1-Y)^2J^{-\f32}\p_\al^3\bw,  \tilde{w}^s_\al)- iL(J^{-\f52}\p_\al^3\bar{\bw}, \tilde{w}^s_\al)  + \f52i L((1-Y)^2J^{-\f32}\p_\al^4\bw,  \tilde{w}^s) \\ 
&-\f32iL(J^{-\f52}\p_\al^4 \bar{\bw}, \tilde{w}^s) + iL((1-Y)^2J^{-\f32}\p_\al^5\bw,  \p_\al^{-1}\tilde{w}^s)+K_s.
\end{align*}

Since $(\mathcal{G}^s_1, \mathcal{K}^s_1)$ still have non-perturbative components, we cannot apply Proposition \ref{t:wellposedflow} to the system \eqref{HsFlowSys} directly to obtain Proposition \ref{t:Hswellposedflow}.
We will construct normal form variables $(w_{NF}^s, r_{NF}^s)$ as the sum of $(\tilde{w}^s, \tilde{r}^s)$, the quadratic normal form corrections $(w^s_1, r^s_1)$, and the cubic normal form corrections $(w^s_2, r^s_2)$ satisfying
\begin{equation*}
    \|(w^s_1, r^s_1)\|_{\H^0} \lesssim \CalAZ \|(w,r)\|_{\H^s}, \quad \|(w^s_2, r^s_2)\|_{\H^0} \lesssim \mathcal{A}^2_0 \|(w,r)\|_{\H^s}.
\end{equation*}
Moreover, $(w_{NF}^s, r_{NF}^s)$ solve the system  
\begin{equation}  \label{HsTildeFlowSys}
\begin{cases}
T_{D_t}w_{NF}^s +T_{1-\bar{Y}}\partial_{\al} r_{NF}^s+T_{T_{1-\bar{Y}}R_{\al}}w_{NF}^s= \mathcal{G}^{s, res}_{2}+G_s,\\
T_{D_t}r_{NF}^s-i\mathcal{L}_{para}w_{NF}^s = \mathcal{K}^{s, res}_{2}+K_s,
\end{cases}
\end{equation}
where $(\mathcal{G}^{s, res}_{2}, \mathcal{K}^{s, res}_{2})$ are non-perturbative cubic terms that may have resonances.
Proposition \ref{t:wellposedflow} can then be applied to the paradifferential system \eqref{HsTildeFlowSys}.
The modified energy
\begin{equation} \label{EsParalinDef}
 E^{s,para}_{lin}(w, r) := E^{0,para}_{lin}(w_{NF}^s , r_{NF}^s) + E^s_{4,cor}(\tilde{w}^s, \tilde{r}^s)
\end{equation}
is the final energy required for Proposition \ref{t:Hswellposedflow},
where the energy $E^{0,para}_{lin}(w,r)$ is the modified energy constructed in Proposition \ref{t:wellposedflow}, and $E^s_{4,cor}$ is a quartic modified energy.

In the rest of this section, we will construct the quadratic normal form corrections $(w^s_1, r^s_1)$, and the cubic normal form corrections $(w^s_2, r^s_2)$ and the quartic modified energy $E^s_{4,cor}(\tilde{w}^s, \tilde{r}^s)$.

\subsubsection{Construction of quadratic normal form corrections $(w^s_1, r^s_1)$}
We will first construct quadratic normal form corrections $(w^s_1, r^s_1)$ such that 
\begin{align*}
 &\p_t w^s_1+T_{1-\bar{Y}}\p_\al r^s_1 + \text{cubic and higher terms}\\
 =&  -\partial_\alpha L(R_\alpha , T_{1-\bar{Y}}\partial_\alpha^{-1}\tilde{w}^s) -\f{s}{2} T_{T_{1-\bar{Y}}R_{\al}}\tilde{w}^s +\f{s}{2} T_{J^{-1}\bw_{\al}} \tilde{r}^s \\
 &  - L(\bar{R}_\alpha , T_{1-Y}\tilde{w}^s) +L(\bar{\bw}_\alpha, T_{(1-\bar{Y})^2}\tilde{r}^s)  -\f{s}{2} T_{T_{1-Y}\bar{R}_{\al}}\tilde{w}^s+\f{s}{2} T_{(1-\bar{Y})^2\bar{\bw}_{\al}}\tilde{r}^s+ G_s,\\
&\p_t r^s_1-iT_{J^{-\frac{3}{2}}(1-Y)}\p_\al^4w^s_1 + \text{cubic and higher terms}\\
=& -L(R_\alpha , T_{1-\bar{Y}}\tilde{r}^s) -\f{s}{2} T_{T_{1-\bar{Y}}R_{\al}}\tilde{r}^s -2si T_{(1-Y)^2J^{-\f32}\bw_{\al}}\p^3_{\al}\tilde{w}^s -\f52iL\Big(\bw_\al, T_{(1-Y)J^{-\f32}}\p_\al^3 \tilde{w}^s\Big) \\
&-3si T_{(1-Y)^2J^{-\f32}\p^2_{\al}\bw}\tilde{w}^s_{\al\al}-5iL\Big(\bw_{\al\al}, T_{(1-Y)^2J^{-\f32}}\p_\al^2 \tilde{w}^s\Big)-2si T_{(1-Y)^2J^{-\f32}\p^3_{\al}\bw}\tilde{w}^s_{\al}\\
&-5iL\Big(\p_\al^3\bw,  T_{(1-Y)^2J^{-\f32}} \tilde{w}^s_\al\Big)-\f{s}{2}i T_{(1-Y)^2J^{-\f32}\p^4_{\al}\bw}\tilde{w}^s-\f52i L\Big(\p_\al^4\bw,  T_{(1-Y)^2J^{-\f32}} \tilde{w}^s \Big) \\
 &-iL\Big(\p_\al^5\bw,  T_{(1-Y)^2J^{-\f32}}\p_\al^{-1}\tilde{w}^s \Big) -L(\bar{R}_\alpha, T_{1-Y}\tilde{r}^s) -\f{s}{2} T_{T_{1-Y}\bar{R}_{\al}}\tilde{r}^s -2si T_{J^{-\f52}\bar{\bw}_{\al}}\p^3_{\al}\tilde{w}^s\\
 &-\f32iL\Big(\bar{\bw}_\al, T_{J^{-\f52}}\p_\al^3 \tilde{w}^s \Big) -3si T_{J^{-\f52}\p^2_{\al}\bar{\bw}}\tilde{w}^s_{\al\al}-iL\Big(\bar{\bw}_{\al\al}, T_{J^{-\f52}}\p_\al^2 \tilde{w}^s \Big)\\
 &-2si T_{J^{-\f52}\p^3_{\al}\bar{\bw}}\tilde{w}^s_{\al}+ iL\Big(\p_\al^3\bar{\bw}, T_{J^{-\f52}}\tilde{w}^s_\al \Big) +\f{s}{2}i T_{J^{-\f52}\p^4_{\al}\bar{\bw}}\tilde{w}^s-\f32iL\Big(\p_\al^4 \bar{\bw}, T_{J^{-\f52}}\tilde{w}^s \Big) + K_s.
\end{align*}

We consider normal form transformations $(w^s_1, r^s_1)$ as the sum of balanced paradifferential bilinear forms of the following type:
\begin{equation*} 
\begin{aligned}
w^s_1 &=B^s_{1,h}\left(\bw , T_{1-Y}\tilde{w}^s\right) + C^s_{1,h}\left(R , T_{J^{\frac{1}{2}}(1+\bw)^2}\tilde{r}^s\right)+ B^s_{1,a}\left(\bar{\bw}, T_{1-\bar{Y}}\tilde{w}^s\right) + C^s_{1,a}\left(\bar{R} , T_{J^{\frac{3}{2}}}\tilde{r}^s\right),\\
r^s_1 &=  A^s_{1,h} \left( \bw, T_{1-Y}\tilde{r}^s\right) + D^s_{1,h}\left(R, \tilde{w}^s\right)+A^s_{1,a} \left(\bar{\bw}, T_{1-\bar{Y}}\tilde{r}^s\right) + D^s_{1,a}\left( \bar{R}, T_{(1-Y)(1+\bar{\bw})}\tilde{w}^s\right).
\end{aligned}
\end{equation*}
For above paradifferential bilinear forms, we compute 
\begin{align*}
&\partial_t w^s_1+T_{1-\bar{Y}} \partial_\alpha r^s_1 + \text{cubic and higher terms} \\
=& T_{J^{-1}}\partial_\alpha A^s_{1,h}(\bw, \tilde{r}^s) -T_{J^{-1}}B^s_{1,h}(\bw, \tilde{r}^s_{\al}) +  T_{J^{-1}}C^s_{1,h}(i\p^4_{\al}\bw, \tilde{r}^s) \\
&-T_{1-\bar{Y}}B^s_{1,h}( R_\alpha, \tilde{w}^s)+T_{1-\bar{Y}}C^s_{1,h}(R,i\p_\al^4 \tilde{w}^s) + T_{1-\bar{Y}}\partial_\alpha D^s_{1,h}(R, \tilde{w}^s)\\
&+T_{(1-\bar{Y})^2}\p_{\al}A^s_{1,a}(\bar{\bw}, \tilde{r}^s)-T_{(1-\bar{Y})^2}B^s_{1,a}(\bar{\bw}, \tilde{r}^s_{\al})-T_{(1-\bar{Y})^2}C^s_{1,a}(i\p^4_{\al}\bar{\bw}, \tilde{r}^s)\\
 &-T_{1-Y}B^s_{1,a}(\bar{R}_\alpha, \tilde{w}^s)+T_{1-Y}C^s_{1,a}(\bar{R},i\p_\al^4 \tilde{w}^s) + T_{1-Y}\partial_\alpha D^s_{1,a}(\bar{R}, \tilde{w}^s),\\
&\partial_t r^s_1  -iT_{J^{-\frac{3}{2}}(1-Y)}\p_\al^4 w^s_1+ \text{cubic and higher  terms} \\
=& -T_{1-\bar{Y}}A^s_{1,h}( R_\alpha, \tilde{r}^s) - iT_{1-\bar{Y}}\partial_\alpha^4 C^s_{1,h}(R, \tilde{r}^s) - T_{1-\bar{Y}}D^s_{1,h}(R, \tilde{r}^s_\alpha)\\
&+ iT_{J^{-\frac{3}{2}}(1-Y)^2}A^s_{1,h}(\bw,\p_\al^4 \tilde{w}^s) -iT_{J^{-\frac{3}{2}}(1-Y)^2}\partial_\alpha^4 B^s_{1,h}(\bw, \tilde{w}^s) +iT_{J^{-\frac{3}{2}}(1-Y)^2}D^s_{1,h}(\p_\al^4\bw, \tilde{w}^s)\\
 &-T_{1-Y}A^s_{1,a}(\bar{R}_\alpha, \tilde{r}^s) - iT_{1-Y}\partial_\alpha^4 C^s_{1,a}( \bar{R}, \tilde{r}^s) - T_{1-Y}D^s_{1,a}(\bar{R}, \tilde{r}^s_\alpha)\\
&+ iT_{J^{-\f52}}A^s_{1,a}(\bar{\bw},\p_\al^4 \tilde{w}^s) -iT_{J^{-\f52}}\partial_\alpha^4 B^s_{1,a}(\bar{\bw}, \tilde{w}^s) -iT_{J^{-\f52}}D^s_{1,a}(\p_\al^4\bar{\bw}, \tilde{w}^s).
\end{align*}
We write $\mathfrak{a}^s_{1, h}(\xi, \eta)$ for the symbol of $A^{1,h}_{bal}(\bw, T_{1-Y}\tilde{r}^s)$, $\mathfrak{a}^s_{1,a}(\eta, \zeta)$ for the symbol of $A^{1,a}_{bal}(\bar{\bw}, T_{1-\bar{Y}}\tilde{r}^s)$ and similarly for other bilinear forms.
To match  paradifferential source terms of holomorphic type in $(\mathcal{G}_1^s, \mathcal{K}_1^s)$, paradifferential symbols of the holomorphic type solve the following algebraic systems:
\begin{equation*}
\begin{cases}
(\xi + \eta)\mathfrak{a}^s_{1, h}-\eta\mathfrak{b}^s_{1, h}+\xi^4\mathfrak{c}^s_{1, h}= \f{s}{2}\xi \chi_1(\xi,\eta),\\
\xi\mathfrak{b}^s_{1, h}-\eta^4\mathfrak{c}^s_{1, h}-(\xi+\eta)\mathfrak{d}^s_{1, h}= (\xi+\xi^2\eta^{-1})\chi_{3}(\xi,\eta)+\f{s}{2}\xi\chi_1(\xi,\eta),\\
\xi\mathfrak{a}^s_{1, h}+(\xi+\eta)^4\mathfrak{c}^s_{1, h}+\eta\mathfrak{d}^s_{1, h}=\xi\chi_{3}(\xi,\eta) + \f{s}{2}\xi\chi_1(\xi,\eta),\\
\eta^4 \mathfrak{a}^s_{1, h}-(\xi+\eta)^4\mathfrak{b}^s_{1, h}+\xi^4\mathfrak{d}^s_{1, h}=-\left(\f52 \xi\eta^3+5\xi^2 \eta^2+5\xi^3\eta+\f52\xi^4+\xi^5\eta^{-1}\right)\chi_{3}(\xi,\eta) \\
-s(2\xi \eta^3 + 3\xi^2\eta^2 + 2\xi^3\eta + \f12\xi^4)\chi_1(\xi,\eta),
\end{cases}
\end{equation*}
where $\chi_3(\xi, \eta)$ represents the symbolic relation for the bilinear form $L$ defined in \eqref{LDef} so that $|\xi| \ll |\eta|$.
The expressions for the bilinear symbols of holomorphic type are given by
\begin{align*}
\mathfrak{a}^s_{1, h}=&\f{-10\xi^7 + 5\xi^6\eta + 80\xi^5\eta^2 + 190\xi^4\eta^3+242\xi^3\eta^4+ 200\xi^2\eta^5+100\xi\eta^6+ 25\eta^7}{2\eta(25\xi^6+100\xi^5\eta+200\xi^4\eta^2+246\xi^3\eta^3+200\xi^2\eta^4+100\xi\eta^5+25\eta^6)}\chi_{3}(\xi,\eta)\\ 
&+\f{-2\xi^7 + 45\xi^6\eta + 190\xi^5\eta^2 + 390\xi^4\eta^3+487\xi^3\eta^4+ 400\xi^2\eta^5+200\xi\eta^6+ 50\eta^7}{4\eta(25\xi^6+100\xi^5\eta+200\xi^4\eta^2+246\xi^3\eta^3+200\xi^2\eta^4+100\xi\eta^5+25\eta^6)}s\chi_{1}(\xi,\eta),\\
\mathfrak{b}^s_{1, h}=&\f{25\xi^7 + 125\xi^6\eta + 300\xi^5\eta^2 + 442\xi^4\eta^3+442\xi^3\eta^4+ 300\xi^2\eta^5+125\xi\eta^6+ 25\eta^7}
{2\eta(25\xi^6+100\xi^5\eta+200\xi^4\eta^2+246\xi^3\eta^3+200\xi^2\eta^4+100\xi\eta^5+25\eta^6)}\chi_{3}(\xi,\eta)\\
&+ \f{-5\xi^7 + 35\xi^6\eta + 180\xi^5\eta^2 + 385\xi^4\eta^3+487\xi^3\eta^4+ 400\xi^2\eta^5+200\xi\eta^6+ 50\eta^7}
{4\eta(25\xi^6+100\xi^5\eta+200\xi^4\eta^2+246\xi^3\eta^3+200\xi^2\eta^4+100\xi\eta^5+25\eta^6)}s\chi_{1}(\xi,\eta) ,\\
\mathfrak{c}^s_{1, h}=&\f{5(\xi^4+3\xi^3\eta + 4\xi^2\eta^2+3\xi\eta^3+\eta^4)}
{\eta(25\xi^6+100\xi^5\eta+200\xi^4\eta^2+246\xi^3\eta^3+200\xi^2\eta^4+100\xi\eta^5+25\eta^6)}\chi_{3}(\xi,\eta)\\
&+\f{\xi^3(\xi+\eta)}
{\eta(25\xi^6+100\xi^5\eta+200\xi^4\eta^2+246\xi^3\eta^3+200\xi^2\eta^4+100\xi\eta^5+25\eta^6)}s\chi_{1}(\xi,\eta),\\
\mathfrak{d}^s_{1, h}=&-\f{5( 5\xi^7 + 20\xi^6\eta + 40\xi^5\eta^2 + 50\xi^4\eta^3+42\xi^3\eta^4+ 24\xi^2\eta^5+9\xi\eta^6+ 2\eta^7)}
{2\eta(25\xi^6+100\xi^5\eta+200\xi^4\eta^2+246\xi^3\eta^3+200\xi^2\eta^4+100\xi\eta^5+25\eta^6)}\chi_{3}(\xi,\eta)\\
&-\f{ 5\xi^7 + 10\xi^6\eta + 10\xi^5\eta^2 + 5\xi^4\eta^3+2\xi^3\eta^4}
{4\eta(25\xi^6+100\xi^5\eta+200\xi^4\eta^2+246\xi^3\eta^3+200\xi^2\eta^4+100\xi\eta^5+25\eta^6)}s\chi_{1}(\xi,\eta).
\end{align*}

Note that $\chi_3(\xi, \eta) = -s\chi_1(\xi,\eta) + \text{lower order terms}$.
These symbols of holomorphic type are of the following type:
\begin{equation*}
\mathfrak{a}^s_{1, h} = O(\xi \eta^{-1})\chi_{1}(\xi, \eta), \quad \mathfrak{b}^s_{1, h} = O( \xi \eta^{-1})\chi_{1}(\xi, \eta),  \quad \mathfrak{c}^s_{1, h} =O(\eta^{-3})\chi_{1}(\xi, \eta), \quad \mathfrak{d}^s_{1, h} = O(1) \chi_{1}(\xi, \eta).   
\end{equation*}

Similarly, to match  paradifferential source terms of the mixed type in $(\mathcal{G}_0^s, \mathcal{K}_0^s)$, paradifferential symbols of the mixed type solve the following algebraic system:
\begin{equation*}
\begin{cases}
(\zeta-\eta)\mathfrak{a}^s_{1, a}-\zeta\mathfrak{b}^s_{1, a}-\eta^4\mathfrak{c}^s_{1, a}=-\eta\chi_{3}(\eta,\zeta) -\f{s}{2}\eta \chi_1(\eta, \zeta),\\
\eta\mathfrak{b}^s_{1, a}+\zeta^4\mathfrak{c}^s_{1, a}+(\zeta-\eta)\mathfrak{d}^s_{1, a}=\eta\chi_{3}(\eta,\zeta) + \f{s}{2}\eta \chi_1(\eta, \zeta),\\
\eta\mathfrak{a}^s_{1, a}-(\zeta-\eta)^4\mathfrak{c}^s_{1, a}-\zeta\mathfrak{d}^s_{1, a}=\eta\chi_{3}(\eta,\zeta)+ \f{s}{2}\eta \chi_1(\eta, \zeta),\\
\zeta^4 \mathfrak{a}^s_{1, a}-(\zeta-\eta)^4\mathfrak{b}^s_{1, a}-\eta^4\mathfrak{d}^s_{1, a}=\left(\f32 \eta\zeta^3-\eta^2\zeta^2-\eta^3\zeta +\f32\eta^4\right)\chi_{3}(\eta,\zeta)\\
+ s\left(2 \eta\zeta^3-3\eta^2\zeta^2+2\eta^3\zeta -\f12\eta^4\right)\chi_{1}(\eta,\zeta),
\end{cases}
\end{equation*}
where $\chi_3(\eta, \zeta)$ represents the symbolic relation for the bilinear form $L$ defined in \eqref{LDef} so that $|\eta| \ll |\zeta|$.
The solutions of the above system are given by
\begin{align*}
\mathfrak{a}^s_{1, a}=&\f{ 8\eta^6-24\eta^5\zeta+49\eta^4\zeta^2-58\eta^3\zeta^3+75\eta^2\zeta^4-50\eta\zeta^5+25\zeta^6}{8\eta^6-24\eta^5\zeta+74\eta^4\zeta^2-108\eta^3\zeta^3+150\eta^2\zeta^4-100\eta\zeta^5+50\zeta^6}\chi_{3}(\eta, \zeta) + \f{s}{2}\chi_1(\eta, \zeta),\\
\mathfrak{b}^s_{1, a}=&\f{-2\eta^6+11\eta^5\zeta - 11\eta^4\zeta^2 + 17\eta^3\zeta^3+25\eta^2\zeta^4-25\eta\zeta^5+25\zeta^6}{8\eta^6-24\eta^5\zeta+74\eta^4\zeta^2-108\eta^3\zeta^3+150\eta^2\zeta^4-100\eta\zeta^5+50\zeta^6}\chi_{3}(\eta, \zeta) + \f{s}{2}\chi_1(\eta, \zeta),\\
\mathfrak{c}^s_{1, a}=&\f{10(\eta^2\zeta -\eta\zeta^2 + \zeta^3)}{8\eta^6-24\eta^5\zeta+74\eta^4\zeta^2-108\eta^3\zeta^3+150\eta^2\zeta^4-100\eta\zeta^5+50\zeta^6}\chi_{3}(\eta, \zeta),\\
\mathfrak{d}^s_{1, a}=&-\f{5(2\eta^6-5\eta^5\zeta+12\eta^4\zeta^2-13\eta^3\zeta^3+12\eta^2\zeta^4-5\eta\zeta^5+2\zeta^6)}{8\eta^6-24\eta^5\zeta+74\eta^4\zeta^2-108\eta^3\zeta^3+150\eta^2\zeta^4-100\eta\zeta^5+50\zeta^6}\chi_{3}(\eta, \zeta).
\end{align*}
Since $\chi_3(\eta, \zeta) = -s\chi_1(\eta,\zeta) + \text{lower order terms}$, these symbols of the mixed type are of the following type:
\begin{equation*}
\mathfrak{a}^s_{1, a} = O(\eta^3 \zeta^{-3})\chi_{1}(\eta, \zeta), \quad \mathfrak{b}^s_{1, a} = O(\eta\zeta^{-1})\chi_{1}(\eta, \zeta),  \quad \mathfrak{c}^s_{1, a} =O(\eta^{-3})\chi_{1}(\eta, \zeta), \quad \mathfrak{d}^s_{1, a}  = O(1)\chi_{1}(\eta, \zeta).
\end{equation*}
Therefore, we get the upper bound of $(w^s_1, r^s_1)$ in $\H^0$.
\begin{equation*}
    \|(w^s_1, r^s_1)\|_{\H^0} \lesssim \CalAZ \|(\tilde{w}^s, \tilde{r}^s)\|_{\H^0}  \lesssim \CalAZ \|(w,r)\|_{\H^s}.
\end{equation*}

\subsubsection{$\H^s$ normal form analysis for cubic terms of the homogeneous paradifferential system}
Then we construct cubic normal form corrections $(w^s_2, r^s_2)$ to eliminate remaining non-resonant non-perturbative cubic terms.
According to the computations for quadratic normal form corrections $(w^s_1, r^s_1)$, we get 
\begin{equation*} 
\begin{cases}
T_{D_t}w^s_1 +T_{1-\bar{Y}}\p_\al r^s_1+T_{T_{1-\bar{Y}}R_{\al}}w^s_1= -\mathcal{G}^s_1 + \mathcal{G}^s_2+G_s,\\
T_{D_t}r^s_1-i\mathcal{L}_{para}w^s_1 = -\mathcal{K}^s_1  +\mathcal{K}^s_2  +K_s,
\end{cases}
\end{equation*}
where $(\mathcal{G}^s_2, \mathcal{K}^s_2)$ are cubic and higher terms given by
\begin{align*}
\mathcal{G}^s_2 &=  C_1 T_{J^{-1}(1-\bar{Y}) |\bw_\al|^2} \partial_\al^{-1}\tilde{r}^s + C_2i T_{(1-\bar{Y})^3 \bar{\bw}_\al^2} \partial_\al^{-1}\tilde{r}^s,\\
\mathcal{K}^s_2 &=  C_3 T_{J^{-\f32}(1-Y)^3\bw^2_\al} \tilde{w}_{\al \al}^s + C_4i T_{J^{-\f52}(1-Y)|\bw_\al|^2} \tilde{w}_{\al \al}^s + C_5i T_{J^{-\f52}(1-\bar{Y})^3 \bar{\bw}^2_\al} \tilde{w}_{\al \al}^s ,
\end{align*}
for some constants $C_i$ that depend on $s$.
The non-resonant part of the cubic terms in $(\mathcal{G}^s_2, \mathcal{K}^s_2)$ are 
\begin{align*}
\mathcal{G}^{s,non}_2 &=   C_2i T_{(1-\bar{Y})^3 \bar{\bw}_\al^2} \partial_\al^{-1}\tilde{r}^s,\\
\mathcal{K}^{s,non}_2 &=  C_3i T_{J^{-\f32}(1-Y)^3\bw^2_\al} \tilde{w}_{\al \al}^s + C_5i T_{J^{-\f52}(1-\bar{Y})^3 \bar{\bw}^2_\al} \tilde{w}_{\al \al}^s .
\end{align*}
According to the discussion in Appendix \ref{s:Discuss}, we can construct cubic modified variables  $(w^s_2, r^s_2) = O_{\H^0}(\mathcal{A}_0^2\|(\tilde{w}^s, \tilde{r}^s)\|^2_{\H^0})$ to eliminate $(\mathcal{G}^{s,non}_2, \mathcal{K}^{s,non}_2)$.
They do not produce any extra non-perturbative terms.

For cubic terms that may have resonance, 
\begin{equation*}
 \mathcal{G}^{s,res}_2 =  C_1 T_{J^{-1}(1-\bar{Y}) |\bw_\al|^2} \partial_\al^{-1}\tilde{r}^s, \quad \mathcal{K}^{s,res}_2 = C_4 T_{J^{-\f52}(1-Y)|\bw_\al|^2} \tilde{w}_{\al \al}^s,
\end{equation*}
one cannot use cubic normal form directly.
Setting 
\begin{equation*}
   w^s_{NF}: = \tilde{w}^s + w^s_1 + w^s_2, \quad  r^s_{NF}: = \tilde{r}^s + r^s_1 + r^s_2.
\end{equation*}
$(w^s_{NF}, r^s_{NF})$ solve the system \eqref{HsTildeFlowSys}. 
We get that
\begin{equation*}
   E^{0,para}_{lin}(w_{NF}^s , r_{NF}^s) = (1+O(\mathcal{A}_0))\|(w^s_{NF},r^s_{NF})\|^2_{\H^0} = (1+O(\mathcal{A}_0))\|(w,r)\|^2_{\H^s}.
\end{equation*}

In addition, according to the computation in Section \ref{s:HomoFlow},
\begin{align*}
    &\f{d}{dt} E^{0,para}_{lin}(w_{NF}^s , r_{NF}^s) = -2\Re \int  T_{J^{-\f54}} \partial_\al\mathcal{G}^{s,res}_2\cdot \p_\al^2 \bar{w}^s_{NF} + 2\Re \int r^s_{NF} \cdot T_{J^\f14}\bar{\mathcal{K}}^{s,res}_2 \,d\al \\
    =&-2C_2\Re\int i T_{J^{-\f94}(1-\bar{Y})|\bw_\al|^2}\tilde{r}^s \cdot \p_\al^2 \bar{\tilde{w}}^s\,d\al - 2C_4 \Re \int i\tilde{r}^s \cdot   T_{J^{-\f94}(1-\bar{Y})|\bw_\al|^2} \p_\al^2 \bar{\tilde{w}}^s\,d\al \\
    &+ O\left(\ASSharp \right) \|(\tilde{w}^s, \tilde{r}^s)\|_{\H^0}^2 \\
    =&-2(C_2 + C_4)\Re\int i T_{J^{-\f94}(1-\bar{Y})|\bw_\al|^2}\tilde{r}^s \cdot \p_\al^2 \bar{\tilde{w}}^s \,d\al + O\left(\ASSharp \right) \|(w, r)\|_{\H^s}^2 .
\end{align*}
We choose the quartic energy correction
\begin{equation*}
   E^s_{4,cor}(\tilde{w}^s, \tilde{r}^s) = (C_2 + C_4)\Re\int \p_\al^{-1} \tilde{r}^s\cdot T_{J^{-\f34}|\bw_\al|^2} \p_\al^{-1} \bar{\tilde{r}}^s\,d\al.
\end{equation*}
The energy satisfies
\begin{equation*}
|E^s_{4,cor}(\tilde{w}^s, \tilde{r}^s)| \lesssim \|\bw \|^2_{C^\epsilon_*} \|\tilde{r}^s\|^2_{L^2} \lesssim \mathcal{A}_0^2 \|r\|^2_{H^s}.
\end{equation*}
For its time derivative, 
\begin{align*}
\f{d}{dt}E^s_{4,cor}(\tilde{w}^s, \tilde{r}^s) =& 2 (C_2 + C_4)\Re\int \p_\al^{-1} \tilde{r}^s\cdot T_{J^{-\f34}|\bw_\al|^2} \p_\al^{-1} \bar{\tilde{r}}^s_t\,d\al \\
&+ (C_2 + C_4)\Re\int \p_\al^{-1} \tilde{r}^s\cdot T_{\p_t(J^{-\f34}|\bw_\al|^2)} \p_\al^{-1} \bar{\tilde{r}}^s\,d\al \\
=& 2(C_2 + C_4)\Re\int i T_{J^{-\f94}(1-\bar{Y})|\bw_\al|^2}\tilde{r}^s \cdot \p_\al^2 \bar{\tilde{w}}^s \,d\al + O\left(\ASSharp \right) \|(w, r)\|_{\H^s}^2 .
\end{align*}

Hence, choosing the modified energy  defined in \eqref{EsParalinDef}, we finish the proof of Proposition \ref{t:Hswellposedflow}.

As a direct corollary of Proposition \ref{t:HatwrHsWellposed} and Proposition \ref{t:Hswellposedflow}, we obtain the modified energy estimate of \eqref{LinearHatEqn}.

\begin{proposition} 
Assume that $\mathcal{A}_0
\lesssim 1$ and $\mathcal{A}_{\sharp,\f74}  \in L^2_t([0,T])$ for some time $T>0$, then if $(\hat{w},\hat{r})$ solve the homogeneous paradifferential system \eqref{LinearHatEqn} on $[0,T]$, there exists an  energy functional $\tilde{E}^{s,para}_{lin}(w, r)$ such that on $[0,T]$ for $s\geq 0$, we have the following two properties:
\begin{enumerate}
\item Norm equivalence:
\begin{equation*}
    \tilde{E}^{s,para}_{lin}(\hat{w},\hat{r}) = (1+O(\CalAZ)) \|(\hat{w},\hat{r})\|_{\H^s}^2.
\end{equation*}
\item The time derivative of $\tilde{E}^{s,para}_{lin}(\hat{w},\hat{r})$ is bounded by
\begin{equation*}
    \frac{d}{dt}  \tilde{E}^{s,para}_{lin}(\hat{w},\hat{r}) \lesssim_{\CalAZ} \ASSharp \|(\hat{w},\hat{r})\|_{\H^s}^2.
\end{equation*}
\end{enumerate}  
\end{proposition}
Given $(\bw, R)$ that solve the hydroelastic waves \eqref{HF19} on $[0,T]$, since the corresponding modified variables $(\bw_{NF}, R_{NF})$  satisfy  the bound \eqref{DiffNfBound} and solve \eqref{LinearHatEqn}.
We get that 
\begin{equation*}
    E_s(\bw, R): = \tilde{E}^{s,para}_{lin}(\bw_{NF}, R_{NF})
\end{equation*}
is the desired modified energy that finishes the proof of Theorem \ref{t:MainEnergyEst}.

\section{Low regularity well-posedness of the hydroelastic waves} \label{s:LowWellPosed}
In this final section, we give an outline of how to prove the local well-posedness of differentiated hydroelastic waves \eqref{HF19}.
For a detailed exposition of the argument, we refer the interested reader to Section $7$ of the low-regularity well-posedness theory for gravity water waves  \cite{AIT}.

First, for any initial data with high regularity $\mathcal{H}^{\f52k}$ for $k>2$, a solution exists based on the result in \cite{MR4104949}.
Moreover, by the modified energy estimate in Theorem \ref{t:MainEnergyEst}, if $\mathcal{A}_0(t) \lesssim 1$ and $\ASSharp(t) \in L^2([0,T])$, using Gronwall's inequality
\begin{equation*}
   \|(\bw, R)(t)\|_{\mathcal{H}^{\f52k}}  \lesssim_\CalAZ \exp{ \ \left\{\int_0^t \ASSharp(\tau) \,d\tau \right\}  }  \|(\bw, R)(0)\|_{\mathcal{H}^{\f52k}}.
\end{equation*}
By Sobolev embedding, if $(\bw, R)(t) \in \H^s$ for $s>\f34$, then $\mathcal{A}_{\sharp, \f74}(t) \lesssim \|(\bw, R)(t) \|_{\H^s}$.

We then construct rough solutions as the unique limit of smooth solutions.
Given solution $(\bw_0, R_0) \in \H^s$ for $s>\f34$, we perform frequency truncation on the initial data, and obtain $(\bw_0^k, R_0^k) = P_{<k}(\bw_0, R_0) \in \mathcal{H}^{\f52k}$ for $k>2$. 
The corresponding solutions $(\bw^k, R^k)$ exist with a uniform lifespan bound.
Utilizing the method of frequency envelopes (see Definition $7.1 $ in \cite{AIT}) and energy estimate Theorem \ref{t:MainEnergyEst}, we obtain an upper bound for $(\bw^k, R^k)$ in $\H^{\f{5}{2}k}$ for $k>2$. 
On the other hand, 
\begin{equation*}
    (w^k, r^k) = (\partial_k W^{k}, \partial_k Q^{k} - R^{k} \partial_k W^{k} )
\end{equation*}
solve the corresponding linearized equations around $(\bw^{k}, R^{k})$.
Using the modified energy estimate for the linearized hydroelastic waves Theorem \ref{t:LinearizedWellposed},  
one can establish the bound for the difference $(\bw^{k+1}-\bw^{k}, R^{k+1}-R^{k})$ in $\mathcal{H}^{0}$.
Summing over $k$ and using interpolation, the sequence $(\bw^{k}, R^{k})$ converges to a solution $(\bw, R)$ with uniform $\mathcal{H}^{s}$ bound in time interval $[0,T]$.
This process also establishes uniqueness, as the solution is defined as the unique limit of regular solutions.

Finally, we demonstrate the continuous dependence on initial data for rough solutions.
The proof follows essentially the same argument as the final paragraph of Section $7.2$ in \cite{AIT} using frequency envelopes, and is therefore omitted here.

\section*{Acknowledgments}
Jiaqi Yang is supported by National Natural Science Foundation of China under Grant: 12471225.

\appendix

\section{Paradifferential estimates} \label{s:Norms}
In this first part of the appendix, we list the definition of norms and recall paraproducts and paradifferential estimates we have used in previous sections.
Many of these definitions and estimates are relatively standard.
They can be found in for instance \cite{MR2931520, MR3260858} or the textbooks \cite{MR2768550, MR2418072}.

\subsection{Norms and function spaces}
We recall the  Littlewood-Paley frequency decomposition,
\begin{equation*}
    I = \sum_{k\in \mathbb{N}} P_k, 
\end{equation*}
where for each $k\geq 1$, $P_k$ are smooth symbols  localized at  frequency $2^k$,  and $P_0$ selects the low frequency components $|\xi|\leq 1$.
\begin{enumerate}
\item Let $s\in \mathbb{R}$, and $p,q \in [1, \infty]$.
The non-homogeneous Besov space $B^s_{p,q}(\mathbb{R})$ is defined as the space of all tempered distributions $u$ such that
\begin{equation*}
\| u\|_{B^s_{p,q}} : = \left\|(2^{ks}\|P_k u \|_{L^p})_{k=0}^\infty \right\|_{l^q} < +\infty.
\end{equation*}
\item When $p = q = \infty$, Besov space $B^s_{\infty, \infty}$ coincides with the \textit{Zygmund space} $C^s_{*}$.
When $p = q =2$, the Besov space $B^s_{2,2}$ becomes the \textit{Sobolev space} $H^s$.
\item Let $1\leq p_1 \leq p_2 \leq \infty$, $1\leq r_1 \leq r_2 \leq \infty$, then for any real number $s$,
\begin{equation*}
    B^s_{p_1, r_1}(\mathbb{R}) \hookrightarrow B^{s-(\frac{1}{p_1} - \frac{1}{p_2})}_{p_2, r_2}(\mathbb{R}).
\end{equation*}
As a special case when $p_1 = r_1 =2$ and $p_2 = r_2 = \infty$, 
\begin{equation}
 H^{s+\frac{1}{2}}(\mathbb{R}) \hookrightarrow C^s_{*}(\mathbb{R}) \quad \forall s, \label{HsCsEmbed}
\end{equation}
the Sobolev space $H^{s+\frac{1}{2}}(\mathbb{R})$ can be embedding into the Zygmund space $C^s_{*}(\mathbb{R})$.
\item Let $k\in \mathbb{N}$, we let $W^{k,\infty}(\mathbb{R})$ the space of all functions such that $\partial_x^j u \in L^\infty(\mathbb{R})$, $0\leq j \leq k$. 
For $\rho = k+ \sigma$ with $k\in \mathbb{N}$ and $\sigma \in (0,1)$, we denote $W^{\rho, \infty}(\mathbb{R})$  the
space of all function $u\in W^{k,\infty}(\mathbb{R})$ such that the $\partial_x^k u$ is $\sigma$- H\"{o}lder continuous on $\mathbb{R}$. 
\item The Zygmund space $C^s_{*}(\mathbb{R})$ is just the H\"{o}lder space $W^{s, \infty}(\mathbb{R})$ when $s\in (0,\infty)\backslash \mathbb{N}$.
One has the embedding properties
\begin{align*}
  &C_{*}^s(\mathbb{R}) \hookrightarrow L^\infty(\mathbb{R}), \quad s>0; \qquad L^\infty(\mathbb{R}) \hookrightarrow C_{*}^s, \quad s<0;\\
  &C_{*}^{s_1}(\mathbb{R})\hookrightarrow C_{*}^{s_2}(\mathbb{R}), \quad H^{s_1}(\mathbb{R})\hookrightarrow H^{s_2}(\mathbb{R}), \qquad s_1>s_2.
\end{align*}
\end{enumerate}

\subsection{ Paradifferential and  Moser type estimates}
\begin{definition}
\begin{enumerate}
\item Let $\rho\in [0,\infty)$, $m\in \mathbb{R}$. 
$\Gamma^m_\rho(\mathbb{R})$ denotes the space of locally bounded functions $a(x, \xi)$ on $\mathbb{R}\times (\mathbb{R}\backslash \{0\})$, which are $C^\infty$ with respect to $\xi$ for $\xi \neq 0$ and such that for all $k \in \mathbb{N}$ and $\xi \neq 0$, the function $x\mapsto \partial_\xi^k a(x,\xi)$ belongs to $W^{\rho,\infty}(\mathbb{R})$ and there exists a constant $C_k$ with
\begin{equation*}
\forall |\xi|\geq \frac{1}{2}, \quad \|\partial_\xi^k a(\cdot,\xi) \|_{W^{\rho,\infty}} \leq C_k (1+ |\xi|)^{m-k}.
\end{equation*}
Let $a\in \Gamma^m_\rho$,  we define the semi-norm
\begin{equation*}
M^m_{\rho}(a) = \sup_{k \leq \frac{3}{2}+\rho} \sup_{|\xi|\geq \frac{1}{2}} \|(1+ |\xi|)^{k-m}\partial_\xi^k a(\cdot,\xi)  \|_{W^{\rho,\infty}}.
\end{equation*}
\item Given $a\in \Gamma^m_\rho(\mathbb{R})$, let $C^\infty$ functions $\chi(\theta, \eta)$ and $\psi(\eta)$ be such that for some $0<\epsilon_1 < \epsilon_2<1$,
\begin{align*}
    &\chi(\theta, \eta) = 1,  \text{ if } |\theta| \leq \epsilon_1(1+ |\eta|), \qquad \chi(\theta, \eta) = 0,  \text{ if } |\theta| \geq \epsilon_2(1+ |\eta|),\\
    &\psi(\eta) = 0, \text{ if } |\eta|\leq \frac{1}{5}, \qquad \psi(\eta) =1, \text{ if } |\eta|\geq \frac{1}{4}.
\end{align*}
We define the paradifferential operator $T_a$ by
\begin{align*}
 \widehat{T_a u}(\xi) = \frac{1}{2\pi}\int \chi(\xi -\eta, \eta) \hat{a}(\xi-\eta, \eta)\psi(\eta)\hat{u}(\eta) d\eta,
\end{align*}
where $\hat{a}(\theta, \xi)$ is the Fourier transform of a with respect to the  variable $x$.


\item Let $m\in \mathbb{R}$, an operator  is said to be of order $m$ if, for all $s\in \mathbb{R}$, it is bounded from $H^s$ to $H^{s-m}$.
\item  Let $\rho\in (-\infty, 0)$, $m\in \mathbb{R}$. 
$\Gamma^m_\rho(\mathbb{R})$ denotes the space of distributions $a(x, \xi)$ on $\mathbb{R}\times (\mathbb{R}\backslash \{0\})$, which are $C^\infty$ with respect to $\xi$ for $\xi \neq 0$ and such that for all $k \in \mathbb{N}$ and $\xi \neq 0$, the function $x\mapsto \partial_\xi^k a(x,\xi)$ belongs to $C^{\rho}_* (\mathbb{R})$ and there exists a constant $C_k$ with
\begin{equation*}
\forall |\xi|\geq \frac{1}{2}, \quad \|\partial_\xi^k a(\cdot,\xi) \|_{C^{\rho}_*} \leq C_k (1+ |\xi|)^{m-k}.
\end{equation*}
Let $a\in \Gamma^m_\rho$,  we define the semi-norm
\begin{equation*}
M^m_{\rho}(a) = \sup_{k \leq \frac{3}{2}+|\rho|} \sup_{|\xi|\geq \frac{1}{2}} \|(1+ |\xi|)^{k-m}\partial_\xi^k a(\cdot,\xi)  \|_{C^{\rho}_*}.
\end{equation*}
\end{enumerate}
\end{definition}

We recall the basic symbolic calculus for paradifferential operators in the following result.
\begin{lemma}[Symbolic calculus, \cite{MR2418072}]
Let $m\in \mathbb{R}$ and $\rho\in [0, +\infty)$.
\begin{enumerate}
\item If $a\in \Gamma^m_0$, then the paradifferential operator $T_a$ is of order m. Moreover, for all $s\in \mathbb{R}$, there exists a positive constant $K$ such that
\begin{equation}
\|T_a\|_{H^s\rightarrow H^{s-m}} \leq K M^m_0(a). \label{TABound}
\end{equation}
\item If $a\in \Gamma^m_\rho$, and $b\in \Gamma^{m^{'}}_\rho$ with $\rho>0$, then the operator $T_a T_b -T_{a\sharp b}$ is of order $m+m^{'}-\rho$, where the composition
\begin{equation*}
    a \sharp b := \sum_{\alpha<\rho} \frac{(-i)^\alpha}{\alpha !} \partial^\alpha_\xi a(x,\xi) \partial^\alpha_x b(x,\xi).
\end{equation*}
Moreover,  for all $s\in \mathbb{R}$, there exists a positive constant $K$ such that
\begin{equation}
  \|T_a T_b-T_{a\sharp b} \|_{H^s \rightarrow H^{s-m-m^{'}+\rho}} \leq K\left(M^m_\rho(a)M^{m^{'}}_0(b) + M^m_0(a)M^{m^{'}}_\rho(b)\right). \label{CompositionPara}
\end{equation}
\item Let $a\in \Gamma^m_\rho$ with $\rho > 0$. 
Denote by $(T_a)^{*}$ the adjoint operator of $T_a$ and by $\bar{a}$ the complex conjugate of $a$.
Then $(T_a)^{*} -T_{a^{*}}$ is of order $m - \rho$, where 
\begin{equation*}
    a^{*} =  \sum_{\alpha<\rho} \frac{1}{i^\alpha \alpha !} \partial^\alpha_\xi \partial^\alpha_x \bar{a}.
\end{equation*}
Moreover,  for all $s\in \mathbb{R}$, there exists a positive constant $K$ such that
\begin{equation}
 \|(T_a)^{*} -T_{a^{*}} \|_{H^s \rightarrow H^{s-m+\rho}} \leq KM^m_\rho(a). \label{AdjointBound}
\end{equation}
In particular, if $a$ is a function that is independent of $\xi$, then  $(T_a)^* = T_{\bar{a}}$.
\end{enumerate}
\end{lemma}

In Besov spaces, we have similar results for symbolic calculus.
\begin{lemma}[\hspace{1sp}\cite{MR3585049}]
Let $m, m^{'}, s\in \mathbb{R}$, $q\in [1,\infty]$ and $\rho\in [0, + \infty)$.
\begin{enumerate}
\item If $a\in \Gamma^m_0$, then there exists a positive constant $K$ such that
\begin{equation*}
\|T_a\|_{B^s_{\infty, q}\rightarrow B^{s-m}_{\infty, q}} \leq KM^m_0(a).
\end{equation*}
\item If $a\in \Gamma^m_\rho$, and $b\in \Gamma^m_\rho$, then there exists a positive constant $K$ such that
\begin{equation}
  \|T_a T_b-T_{a \sharp b} \|_{B^s_{\infty, q} \rightarrow B^{s-m-m^{'}+\rho}_{\infty, q}} \leq K\left(M^m_\rho(a)M^{m^{'}}_0(b) + M^m_0(a)M^{m^{'}}_\rho(b)\right). \label{CompositionTwo}
\end{equation}
\end{enumerate}
In particular, when $q = \infty$, above symbolic calculus results hold for Zygmund spaces $C^s_{*}$.
\end{lemma}

When $a$ is only a function of $x$, $T_a u$ is the low high paraproduct.
We then define 
\begin{equation*}
\Pi(a, u) := au -T_a u -T_u a
\end{equation*}
to be the high-high paraproduct.
For later use, we record below some estimates for paraproducts.

\begin{lemma} \label{t:ParaProductEst}
\begin{enumerate}
\item Let $\alpha, \beta \in \mathbb{R}$. 
If $\alpha+ \beta >0$, then
\begin{align}
&\|\Pi(a, u)\|_{H^{\alpha + \beta}(\mathbb{R})} \lesssim \| a\|_{C_{*}^\alpha(\mathbb{R})} \| u\|_{H^\beta(\mathbb{R})}, \label{HCHEstimate}\\
&\|\Pi(a, u)\|_{H^{\alpha + \beta}(\mathbb{R})} \lesssim \| a\|_{W^{\alpha,4}(\mathbb{R})} \| u\|_{W^{\beta,4}(\mathbb{R})}, \label{HLLEstimate}\\
& \|\Pi(a, u)\|_{C_{*}^{\alpha + \beta}(\mathbb{R})} \lesssim \| a\|_{C_{*}^\alpha(\mathbb{R})} \| u\|_{C_{*}^\beta(\mathbb{R})}. \label{CCCEstimate}
\end{align}
\item Let $m > 0$ and $s\in \mathbb{R}$, then
\begin{align}
&\|T_{a} u \|_{H^{s-m}(\mathbb{R})} \lesssim \|a\|_{C_{*}^{-m}(\mathbb{R})} \| u \|_{H^s(\mathbb{R})} \label{HsCmStar}, \\
&\|T_{a} u \|_{H^{s}(\mathbb{R})} \lesssim \|a\|_{L^\infty(\mathbb{R})} \| u \|_{H^s(\mathbb{R})}, \label{HsLinfty}\\
&\|T_{a} u \|_{H^{s-m}(\mathbb{R})} \lesssim \|a\|_{W^{-m,4}(\mathbb{R})} \| u \|_{W^{s,4}(\mathbb{R})}, \label{HsLFour}\\
&\|T_{a} u \|_{H^{s-m}(\mathbb{R})} \lesssim \|a\|_{H^{-m}(\mathbb{R})} \| u \|_{C_{*}^s(\mathbb{R})}, \label{HsHmCStar}\\
&\|T_{a} u \|_{C_{*}^{s-m}(\mathbb{R})} \lesssim \|a\|_{C_{*}^{-m}(\mathbb{R})} \| u \|_{C_{*}^s(\mathbb{R})}, \label{CsCmStar}\\
&\|T_{a} u \|_{C_{*}^{s}(\mathbb{R})} \lesssim \|a\|_{L^\infty(\mathbb{R})} \| u \|_{C_{*}^s(\mathbb{R})}. \label{CsLInfty}
\end{align}
\item Let  a smooth function $F\in C^\infty(\mathbb{C}^N)$ satisfying $F(0) = 0$.
There exists a nondecreasing function $\mathcal{F}: \mathbb{R}_{+} \rightarrow \mathbb{R}_{+}$ such that,
\begin{align}
&\|F(u) \|_{H^s} \leq \mathcal{F}(\|u\|_{L^\infty}) \|u\|_{H^s},\quad s\geq 0,  \label{MoserOne}\\
&\|F(u) \|_{C_{*}^s} \leq \mathcal{F}(\|u\|_{L^\infty}) \|u\|_{C_{*}^s},\quad s>0. \label{MoserTwo}
\end{align}
\item Let $s_1 > s_2 > 0$, then
\begin{equation}
    \|uv\|_{C^{-s_2}_*}\lesssim \|u\|_{C^{s_1}_*} \|v\|_{C^{-s_2}_*}. \label{CNegativeAlpha}
\end{equation}
\end{enumerate}
\end{lemma}

When we need to commute the para-coefficients and the balanced paraproducts, we need the following results of the para-associativity.
\begin{lemma}[Para-associativity, \cite{AIT}] \label{t:Paraassociavity} 
 For $s+s_2 > 0$, $s+s_1+s_2>0$, and $s_2<1$, we have
\begin{align*}
&\| T_f\Pi(v,u)-\Pi(v, T_fu)\|_{ C_{*}^{s+s_1+s_2}} \lesssim \|f\|_{C^{s_1}_{*}}\|v\|_{C^{s_2}_{*}}\|u\|_{C^{s}_{*}}, \\
&\| T_f\Pi(v,u)-\Pi(v, T_fu)\|_{ H^{s+s_1+s_2}} \lesssim \|f\|_{C^{s_1}_{*}}\|v\|_{C^{s_2}_{*}}\|u\|_{H^s}, \\
&\| T_f\Pi(v,u)-\Pi(v, T_fu)\|_{ W^{s+s_1+s_2,4}} \lesssim \|f\|_{C^{s_1}_{*}}\|v\|_{C^{s_2}_{*}}\|u\|_{W^{s,4}}.
\end{align*}
\end{lemma}
This result shows that para-coefficients act like constant coefficients modulo perturbative error, so that we can freely commute them with balanced paraproducts.

Finally, to paralinearize functions in Besov spaces, we need the following results in Section 2.8 of \cite{MR2768550}.
\begin{lemma}[Paralinearization \cite{MR2768550}] \label{t:Paralinear}
Let $s, \rho>0$, and $F(u)$ be a smooth function of $u$.
Assume that $\rho$ is not an integer.
Let $p, r_1, r_2 \in [1, \infty]$ and such that $r_2 \geq r_1$.
Let $r \in [1,\infty]$ be defined by $\frac{1}{r} = \min \{1, \frac{1}{r_1}+\frac{1}{r_2} \}$.
Then for any $u\in B^s_{p, r_1}\cap B^\rho_{\infty, r_2}$,
\begin{equation*}
    \|F(u)-F(0)-T_{F^{'}(u)}u\|_{B^{s+\rho}_{p,r}} \leq C(\|u\|_{L^\infty})\|u\|_{B^\rho_{\infty, r_2}}\|u\|_{B^s_{p. r_1}}.
\end{equation*}
\end{lemma}
We remark that this lemma also works for multivariable functions $F$. 
We simply need to replace $F^{'}$ by partial derivatives of $F$.
See for instance Lemma $3.26$ in ~\cite{MR2805065}.
We will apply this paralinearization result with $p=r=r_1 =r_2 = \infty$, so that this is an estimate in Zygmund spaces.

\subsection{Paradiffential estimates for bilinear forms}
In the following, we consider the estimates for the bilinear forms.
Consider a pseudodifferential operator $A(x, D)$ with symbol $a(x,\xi)$ and a function $u(x)$.
Let $\chi_1(\theta_1, \theta_2), \chi_2(\theta_1, \theta_2)$ be two non-negative smooth functions 
\begin{equation}
    \chi_1(\theta_1, \theta_2) = \left\{
\begin{aligned}
1, \text{ when } |\theta_1|\leq \frac{1}{20}|\theta_2|,\\
0, \text{ when } |\theta_1|\geq \frac{1}{10}|\theta_2|,
\end{aligned}
\right.  \label{ChiOnelh}
\end{equation}
\begin{equation}
  \chi_2(\theta_1, \theta_2) = \left\{
\begin{aligned}
&1, \text{ when } \frac{1}{10}\leq \frac{|\theta_1|}{|\theta_2|} \leq 10,\\
&0, \text{ when } |\theta_1|\leq \frac{1}{20}|\theta_2| \text{ or } |\theta_2|\leq \frac{1}{20}|\theta_1|,
\end{aligned}
\right.  \label{ChiTwohh}
\end{equation}
and such that $\chi_1(\theta_1, \theta_2) + \chi_1(\theta_2, \theta_1) + \chi_2(\theta_1, \theta_2) =1$.
For bilinear forms $B(u,v)$ with symbol $m(\xi, \eta)$, we can define  the paradifferential bilinear forms in Weyl quantization:
\begin{itemize}
\item Low-high part and high-high part of the holomorphic bilinear forms:
\begin{align*}
 \widehat{B_{lh}(u,v)}(\zeta) &= \int_{\zeta = \xi +\eta} \chi_1\left(\xi, \eta+\xi\right) m(\xi, \eta)\hat{u}(\xi)\hat{v}(\eta) d\xi,\\
  \widehat{B_{hh}(u,v)}(\zeta) &= \int_{\zeta = \xi +\eta} \chi_2\left(\xi, \eta+\xi\right) m(\xi, \eta)\hat{u}(\xi)\hat{v}(\eta) d\xi.
\end{align*}
\item Low-high part and high-high part of the mixed bilinear forms:
\begin{align*}
 \widehat{B_{lh}(u,v)}(\eta) &= 1_{\eta>0}\int_{\eta = \zeta-\xi} \chi_1\left(\xi, \zeta-\xi\right) m(\xi, \zeta)\bar{\hat{u}}(\xi)\hat{v}(\zeta) d\xi,\\
  \widehat{B_{hh}(u,v)}(\eta) &=1_{\eta>0}\int_{\eta = \zeta-\xi} \chi_2\left(\xi, \zeta-\xi\right) m(\xi, \zeta)\bar{\hat{u}}(\xi)\hat{v}(\zeta) d\xi.
\end{align*}
\end{itemize}
These represent low-high and high-high paradifferential parts of the bilinear forms $B(u,v)$, respectively $\nP B(\bar{u},v)$, restricted to the holomorphic class. 
We will always assume that bilinear symbols $m$ are homogeneous, and smooth away from $(0,0)$.

When the bilinear symbol $m$  is homogeneous, we have the following direct generalization of Lemma \ref{t:ParaProductEst} for bilinear forms, see Coifman-Meyer \cite{MR0518170}, Kenig-Stein \cite{MR1682725},  Muscalu \cite{MR2371442}, and Muscalu-Tao-Thiele \cite{MR1887641}.
\begin{lemma} \label{t:BMuBound}
Let $B^\mu(f,g)$ be a homogeneous  bilinear form of order $\mu\geq 0$ as above. 
For the high-high bilinear forms, when $\alpha+\beta+\mu>0$, $\mu_1 + \mu_2 = \mu$
\begin{align}
&\|B^\mu_{hh}(f, g)\|_{H^{\alpha + \beta}} \lesssim \| f\|_{C_{*}^{\alpha + \mu_1}} \| g\|_{H^{\beta + \mu_2}}, \label{HHBilinearHCH}\\
& \|B^\mu_{hh}(f, g)\|_{C_{*}^{\alpha + \beta}} \lesssim \| f\|_{C_{*}^{\alpha+\mu_1}} \| g\|_{C_{*}^{\beta+\mu_2}}.  \label{HHBilinearCCC}
\end{align}

For the estimate of low-high bilinear form,
\begin{align}
&\|B^\mu_{lh}(f, g) \|_{H^{s-m}} \lesssim \|f\|_{C_{*}^{-m}} \| g \|_{H^{s+\mu}} \label{BFHsCmStar}, \\
&\|B^\mu_{lh}(f, g) \|_{H^{s}} \lesssim \|f\|_{L^\infty(\mathbb{R})} \| g \|_{H^{s+\mu}}, \label{BFHsLinfty}\\
&\|B^\mu_{lh}(f, g) \|_{C_{*}^{s-m}} \lesssim \|f\|_{C_{*}^{-m}} \| g \|_{C_{*}^{s+\mu}}, \label{BFCsCmStar}\\
&\|B^\mu_{lh}(f, g) \|_{C_{*}^{s}} \lesssim \|f\|_{L^\infty} \| g \|_{C_{*}^{s+\mu}}. \label{BFCsLInfty}     
\end{align}
\end{lemma}

\section{Hydroelastic waves related estimates} \label{s:HydroWaveEst}
In the second part of the appendix we recall Sobolev and Zygmund estimates for auxiliary functions in hydroelastic waves.
The proofs can be found in Section 2 of \cite{wan2024}.

We first recall the estimate for the frequency shift $a = i\left(\bar \nP [\bar{R} R_\alpha]- \nP[R\bar{R}_\alpha]\right)$ and the advection velocity $b = 2\Re \nP[(1-\bar{Y})R]$.
\begin{lemma}
The frequency shift $a$ satisfies the estimate
\begin{equation} \label{aEst}
 \|a\|_{H^\epsilon} \lesssim \|R \|^2_{W^{\frac{1}{2}+\epsilon, 4}} \lesssim \ASSharp.  
\end{equation}
The advection velocity $b$ satisfies estimates
\begin{equation} \label{BCOneStar}
\| b\|_{C^{\f14}_{*}} \lesssim_\CalAZ \mathcal{A}_{\f74},  \quad \quad
\|b\|_{W^{\f12,4}} \lesssim_{\CalAZ} \mathcal{A}_{\sharp, \f74},
\end{equation}
as well as the Sobolev estimate
\begin{equation} \label{BHsEst}
\| b\|_{H^s}  \lesssim_{\CalAZ} \|R\|_{H^{s}}, \quad s>0.
\end{equation}
\end{lemma}

Next, we recall estimates for $Y : = \frac{\bw}{1+\bw}$ and $M = 2\Re \nP[R\bar{Y}_\alpha - \bar{R}_\alpha Y]$.
\begin{lemma}
For $s>0$, the auxiliary function $Y$ satisfies
\begin{equation}
\|Y\|_{H^s} \lesssim_\CalAZ \|\bw\|_{H^s}, \quad \|Y\|_{C^s_{*}} \lesssim_\CalAZ \|\bw\|_{C^s_{*}}, \quad \|Y\|_{W^{s,4}}  \lesssim_\CalAZ \|\bw\|_{H^s}. \label{YMoser}
\end{equation}
In particular, $\|Y\|_{C^{\f74}_{*}} \lesssim_\CalAZ \mathcal{A}_{\f74}$ and $\|Y\|_{W^{2,4}} \lesssim_\CalAZ \mathcal{A}_{\sharp,\f74}$.
Moreover, one can write
\begin{equation}
    Y = T_{(1-Y)^2}\bw + E, \label{YWExpression}
\end{equation}
where the error $E$ satisfies the bounds
\begin{equation*}
\|E\|_{C_*^{s+\f74}} \lesssim_\CalAZ \mathcal{A}_{\f74} \|\bw\|_{C_*^{s}}, \quad\|E\|_{H^{s+\f74}} \lesssim_\CalAZ  \mathcal{A}_{\f74} \|\bw\|_{H^{s}}.
\end{equation*}
The auxiliary function $M$ satisfies bounds
\begin{equation} \label{MBound}
\| M\|_{C^\f14_{*}} \lesssim \mathcal{A}_1 \mathcal{A}_{\f74} \lesssim \ASSharp, \quad \| M\|_{H^\f12} \lesssim \ASSharp.
\end{equation}
\end{lemma}

\section{Discussion on the resonances} \label{s:Discuss}
In the final part of the appendix, we have a brief discussion on the three-wave and four-wave interactions in the analysis.
We conclude with a brief discussion of why resonances do not impede the analysis, so that the normal forms can remove those non-perturbative quadratic terms and certain non-resonant parts of cubic terms.

Recall that the linearization of the hydroelastic waves around the zero solution is given by \eqref{e:ZeroLinear}, and its dispersion relation is then
\begin{equation} \label{TauXiRelation}
    \tau = \pm |\xi|^{\f52}, \quad \xi<0.
\end{equation}
Thus resonances in bilinear interactions correspond to zeroes of the expression $|\xi|^{\f52} \pm |\eta|^{\f52} \pm |\xi + \eta|^{\f52}$.
A direct calculation shows:
\begin{align*}
    &\prod_{\pm}|\xi|^{\f52}\pm|\eta|^{\f52}\pm|\xi+\eta|^{\f52}\\
=& \xi^2\eta^2[25(\xi^3+2\xi^2\eta+2\xi\eta^2+\eta^3)^2-4\xi^3\eta^3]\\
=&\xi^2\eta^2(25\xi^6+100\xi^5\eta+200\xi^4\eta^2+246\xi^3\eta^3+200\xi^2\eta^4+100\xi\eta^5+25\eta^6), \\
=& \xi^2\eta^2[25(\xi^4+\eta^4)(\xi+\eta)^2+50\xi^2\eta^2(\xi+\eta)^2+\xi^2\eta^2(125\xi^2+146\xi\eta+125\eta^2)] \\
=& \xi^2 \eta^2 \left\{ 25(\xi^2+\eta^2)^2(\xi+\eta)^2+\xi^2\eta^2\left[\left(25\xi+\frac{73}{25}\eta\right)^2+\left(125-\frac{73^2}{125}\right)\eta^2\right]  \right\} \geq 0.
\end{align*}
Equality holds only if $\xi = \eta = 0$.
Since the bilinear symbols for our paradifferential quadratic normal forms involve denominators that are non-zero for $\xi\ne \eta$ (as is the case for low-high or balanced terms), three-wave resonances cannot occur. 
Thus, the bilinear symbols are well-defined.

We then consider the situation of four-wave resonances.
For trilinear terms, suppose the frequencies of three factors are $\xi_1$, $\xi_2$ and $\xi_3$, then the output frequency is $\xi_0 = -\xi_1-\xi_2-\xi_3$. 
Four-wave resonance occurs if
\begin{equation*}
    |\xi_0|^\f52 \pm |\xi_1|^\f52\pm |\xi_2|^\f52\pm|\xi_3|^\f52 = 0
\end{equation*}
for certain choices of plus or minus sign.
Since $f(\xi) = |\xi|^{\f52}$ is strictly convex, this equality is possible only if the frequencies are paired, i.e.
\begin{equation*}
\xi_i = \pm \xi_j, \quad \xi_k = \pm \xi_l, \quad \{i,j,k,l \} = \{0,1,2,3 \},
\end{equation*}
and $\pm$ signs are also chosen property for the cancellation.
For non-perturbative cubic terms that we want to apply cubic normal forms for elimination, the frequencies of each factor satisfy either
\begin{equation*}
    |\xi_i| \ll |\xi_j|< |\xi_k|, \quad \{i,j,k \} = \{1,2,3 \},
\end{equation*}
or
\begin{equation*}
     |\xi_i| , |\xi_j|\ll |\xi_k|, \quad \{i,j,k \} = \{1,2,3 \}, \quad \xi_i,\xi_j \text{ have the same sign}.
\end{equation*}
In either of the situations, four-wave resonances cannot happen. 
Consequently, four-wave resonances are avoided in the construction of the cubic normal forms and quartic energy corrections.

To compute the expression of cubic normal forms or quartic  energy corrections, one will have to solve $8\times 8$ algebraic systems, as in the computation of quadratic normal forms or cubic energy corrections.
While computing the exact expressions of cubic normal forms  or quartic energy corrections involves solving complex algebraic systems, their exact forms are not critical.
Their primary function is to eliminate non-perturbative terms.
Hence, we will not compute them explicitly in the analysis.

In the final paragraph, we give a qualitative explanation of why these cubic normal forms or quartic modified energies can remove non-perturbative cubic terms or quartic energies in our analysis.
Each normal form transformation generates higher-order terms with lower derivative counts. 
For example, when we compute quadratic normal forms to eliminate all non-perturbative quadratic terms, it produces extra cubic and higher terms with at least one lower order.
Similarly, when we compute cubic normal forms to eliminate remaining non-perturbative cubic terms, it produces extra quartic and higher terms with at least one lower order.
Given that hydroelastic waves are dispersive equations of order $\f52$, and non-perturbative source terms have order of at most $\f32$,  the remaining quartic and higher terms of order less than or equal to zero are perturbative after performing quadratic and cubic normal forms.
The same logic applies to the integral corrections.
Every time we compute the integral corrections, it produces extra higher integral terms with lower order.
Hence, after the construction of cubic and quartic integral corrections, the remaining quintic and higher integral terms are perturbative.

\bibliography{HWW}
\bibliographystyle{plain}
	
\end{document}